\documentclass[a4paper]{article}

\usepackage{amssymb}
\usepackage{amsmath}
\usepackage{amsthm}
\usepackage{bm}
\usepackage{dsfont}
\usepackage{epic}
\usepackage{eepic}
\usepackage{enumerate}
\usepackage[OT2,OT1]{fontenc}
\usepackage{graphicx}
\usepackage{ifthen}
\usepackage{pifont}
\usepackage{textcomp}
\usepackage{wasysym}

\newcommand{\mc}[1]{{\mathcal{#1}}}

\newcommand{\bb}[1]{{\mathbb{#1}}}

\makeatletter
\newcommand\appendixsection{\@startsection {section}{1}{\z@}
	{-3.5ex \@plus -1ex \@minus -.2ex}{2.3ex \@plus.2ex}
	{\normalfont\Large\bfseries\hspace*{-15pt}Appendix\ }}
\makeatother

\numberwithin{equation}{section}
\swapnumbers
\theoremstyle{plain}
	\newtheorem{lemma}{Lemma}[section]
	\newtheorem{proposition}[lemma]{Proposition}
	\newtheorem{theorem}[lemma]{Theorem}
	\newtheorem{corollary}[lemma]{Corollary}
\theoremstyle{definition}
	\newtheorem{definition}[lemma]{Definition}
\theoremstyle{remark}
	\newtheorem{remark}[lemma]{Remark}
	\newtheorem{example}[lemma]{Example}

\newcommand\cyr{%
\renewcommand\rmdefault{wncyr}%
\renewcommand\sfdefault{wncyss}%
\renewcommand\encodingdefault{OT2}%
\normalfont
\selectfont}
\DeclareTextFontCommand{\textcyr}{\cyr}

\newcommand{\nontanglim}{\stackrel{{\tiny\varangle}}{\longrightarrow}}

\newcommand{\zlim}{z\nontanglim x}

\DeclareMathOperator{\dom}{dom}
\DeclareMathOperator{\loc}{loc}
\newcommand{\mr}{\mathring}
\DeclareMathOperator{\mul}{mul}
\DeclareMathOperator{\ran}{ran}
\DeclareMathOperator{\mult}{mult}
\DeclareMathOperator{\rank}{rank}
\DeclareMathOperator{\reg}{reg}
\DeclareMathOperator{\spn}{span}
\DeclareMathOperator{\supp}{supp}
\DeclareMathOperator{\tr}{tr}

\DeclareMathOperator{\RE}{Re}
\DeclareMathOperator{\IM}{Im}
\renewcommand{\Re}{\RE}
\renewcommand{\Im}{\IM}

\newlength{\maxlabwidth}
\newenvironment{axioms}[1]{
	\setlength{\maxlabwidth}{#1}
	\begin{list}{}{
	\setlength{\rightmargin}{2mm}
	\setlength{\leftmargin}{\maxlabwidth}\addtolength{\leftmargin}{2mm}
	\setlength{\labelsep}{0mm}
	\setlength{\labelwidth}{\maxlabwidth}
	\setlength{\itemindent}{0mm}
	
	}
	}{
	\end{list}
	}

\begin{document}

{\Large\bf
\begin{flushleft}
	Spectral multiplicity of selfadjoint 
\\Schr\"odinger operators on star-graphs with
\\standard interface conditions
\end{flushleft}
}
\vspace*{3mm}
\begin{center}
	{\sc Sergey Simonov, Harald Woracek}
\end{center}

\begin{abstract}
	\noindent
	We analyze the singular spectrum of selfadjoint operators which arise from pasting a finite number of boundary relations with a standard interface condition. A model example for this situation is a Schr\"odinger operator on a star-shaped graph with continuity and Kirchhoff conditions at the interior vertex.	We compute the multiplicity of the singular spectrum in terms of the	 spectral measures of the Weyl functions associated with the single (independently considered) boundary relations. This result is a generalization and refinement of a Theorem of I.S.Kac.
\end{abstract}
\begin{flushleft}
	{\small
	{\bf AMS MSC 2010:} 34\,B\,45, 34\,B\,20, 34\,L\,40, 47\,E\,05, 47\,J\,10 
	}
\end{flushleft}



%
%
%
\section{Introduction}

In the present paper we undertake an analysis of the singular spectrum of selfadjoint operators
which are constructed by pasting a finite number of boundary triples (relations) by means of
a standard interface condition.

For the purpose of explaining our results without having to introduce too much terminology,
we consider a model example: A Schr\"odinger operator on a star-graph.
Consider a star-shaped graph having finitely many edges, say, $E_1,\dots,E_n$. We think of the
edges as (finite or infinite) intervals $E_l=[0,e_l)$, where the endpoint $0$
corresponds to the interior vertex. A selfadjoint operator can be constructed from the
following data:
\begin{enumerate}[$(1)$]
	\item On each edge $E_l$, a real-valued potential $q_l\in L^1_{\loc}([0,e_l))$.
	\item Boundary conditions at outer vertices $e_l$, if Weyl's limit circle case prevails for $q_l$ at $e_l$.
	\item An interface condition at the interior vertex.
\end{enumerate}
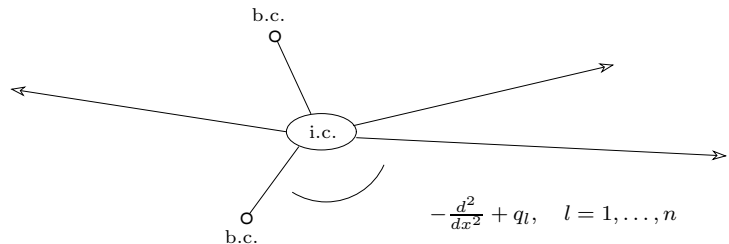
\begin{center}
	\setlength{\unitlength}{0.00035in}
	\begingroup\makeatletter\ifx\SetFigFont\undefined%
	\gdef\SetFigFont#1#2#3#4#5{%
	  \reset@font\fontsize{#1}{#2pt}%
	  \fontfamily{#3}\fontseries{#4}\fontshape{#5}%
	  \selectfont}%
	\fi\endgroup%
	{\renewcommand{\dashlinestretch}{30}
	\begin{picture}(10554,4100)(0,-700)
	\put(4650.041,1147.365){\arc{1871.777}{0.4267}{2.1319}}
	\put(4579,1255){\ellipse{1034}{540}}
	\path(4422,1525)(3927,2605)
	\path(5052,1345)(8877,2245)
	 \whiten\path(8715.527,2145.368)(8877.000,2245.000)(8688.043,2262.178)(8754.349,2216.141)(8715.527,2145.368)
	\path(4242,1030)(3522,40)
	\path(4062,1255)(12,1885)
	 \whiten\path(199.083,1916.620)(12.000,1885.000)(180.639,1798.046)(136.503,1865.633)(199.083,1916.620)
	\path(5097,1165)(10542,895)
	 \whiten\path(10359.249,843.988)(10542.000,895.000)(10365.192,963.841)(10416.155,901.240)(10359.249,843.988)
	 \put(4600,1170){\makebox(0,0)[b]{{\SetFigFont{7}{12.0}{\rmdefault}{\mddefault}{\updefault}\text{i.c.}}}}
	 \put(3900,2570){\makebox(0,0)[b]{{\SetFigFont{10}{12.0}{\rmdefault}{\mddefault}{\updefault}${\bm\circ}$}}}
	 \put(3800,2900){\makebox(0,0)[b]{{\SetFigFont{7}{12.0}{\rmdefault}{\mddefault}{\updefault}\text{b.c.}}}}
	 \put(3480,-130){\makebox(0,0)[b]{{\SetFigFont{10}{12.0}{\rmdefault}{\mddefault}{\updefault}$\bm\circ$}}}
	 \put(3400,-400){\makebox(0,0)[b]{{\SetFigFont{7}{12.0}{\rmdefault}{\mddefault}{\updefault}\text{b.c.}}}}
	 \put(8000,-200){\makebox(0,0)[b]{{\SetFigFont{8}{12.0}{\rmdefault}{\mddefault}{\updefault}$-\frac{d^2}{dx^2}+q_l,\quad l=1,\ldots,n$}}}
	\end{picture}
	}
\end{center}
The operator $A$ one can associate with this data acts in the space $H:=\prod\limits_{l=1}^n L_2(0,e_l)$ as
\begin{equation}\label{K10}
	A\begin{pmatrix}
	u_1
	\\
	\vdots
	\\
	u_n
	\end{pmatrix}
	:=
	\begin{pmatrix}
	-u_1''
	\\
	\vdots
	\\
	-u_n''
	\end{pmatrix}
	+
	\begin{pmatrix}
	q_1u_1
	\\
	\vdots
	\\
	q_nu_n
	\end{pmatrix}
	\,,
\end{equation}
on the domain
\begin{equation}\label{K11}
\begin{aligned}
	\dom A := \Bigg\{&(u_1, \dots,u_n) \in\prod_{l=1}^n L_2(0,e_l):
		\\
	& u_l,u_l'\text{ are absolutely continuous}, -u_l''+q_lu_l\in L_2(0,e_l),
		\\[2mm]
	& u_l\text{ satisfies the boundary condition at outer vertex (if present),}
		\\[-1mm]
	& u_1,\dots,u_n\text{ satisfy the interface condition at the inner vertex}\Bigg\}
		\ .
\end{aligned}
\end{equation}
A frequently used interface condition, sometimes called the ``standard condition'', is
\begin{equation}\label{K1}
	u_1(0)=\ldots=u_n(0) \quad\text{and}\quad \sum_{l=1}^n u_l'(0)=0
	\,.
\end{equation}
In the case ``$n=2$'' the condition \eqref{K1} arises when investigating a whole-line Schr\"odinger operator with the classical method of Titchmarsh and Kodaira.

The task now is to describe the projection-valued spectral measure $E$ of $A$ in terms of the scalar spectral measures $\mu_l$
of the non-interacting operators $A_l$, $l=1,\ldots,n$,
which are defined by the potentials $q_l$ on the edges $E_l$ independently (imposing Dirichlet boundary conditions at the inner
vertex $0$ for each of them).

A precise description of the absolutely continuous part\footnote{Notice that the notions of
absolute continuity and singularity of measures make sense also if the involved measures have different ranges. Moreover, Lebesgue
decompositions of a projection-valued measure with respect to a scalar measure (in this case, the Lebesgue measure) exist.}
$E_{ac}$ of $E$, including computation of its spectral multiplicity, is readily available.
It states that $E_{ac}$ is equivalent (in the sense of mutual absolute continuity) to the sum $\mu_{1,ac}+\dots+\mu_{n,ac}$ of the
absolutely continuous parts of the measures $\mu_l$. Moreover, informally speaking, the local spectral multiplicity corresponding to
$E_{ac}$ is equal to the
number of overlaps of $\mu_{1,ac},\dots,\mu_{n,ac}$ (for a precise formulation see Theorem~\ref{K2}). These facts follow from
\cite[Theorem 6.6]{Gesztesy-Tsekanovskii-2000}, a result which can be viewed as a higher-dimensional analogue (and refinement) of one half
of Aronszajn-Donoghue theory for rank one perturbations. Namely, of the part which asserts stability of absolutely continuous spectrum, cf.\
\cite[Theorem 1]{Aronszajn-1957}, \cite[Theorems 2 and 6]{Donoghue-1965}\footnote{See also \cite[Theorem 3.2, $(i)$--$(iii)$]{Gesztesy-Tsekanovskii-2000} for a summary.}. Another approach proceeds via scattering theory and uses a modification of the Kato-Rosenblum
theorem \cite{Birman-Krein-1962,Kato-1965}, see also \cite[Theorem 1.9]{Yafaev-2000}. Since the operator $A$ is a finite dimensional perturbation (in the resolvent sense) of the operator $\bigoplus_{l=1}^nA_l$, wave operators exist and are complete, which in turn means that the absolutely continuous parts of these operators are unitarily equivalent.

In the present paper, we describe the singular part $E_s$ of $E$, including a formula for spectral multiplicity.
Our main result is Theorem~\ref{K12} below (where we provide the formulation for the Schr\"odinger case;
for the general situation see Theorem~\ref{K38}). Again speaking informally, it says that:
\begin{enumerate}[{\rm(I)}]
	\item One part of $E_s$ appears where at least two of the singular parts $\mu_{l,s}$ of the measures $\mu_l$ overlap.
		Where only one singular part $\mu_{l,s}$ is present, the spectrum disappears.
	\item For the part of $E_s$ described in {\rm(I)}, the local spectral multiplicity is equal to the number of overlaps of
		$\mu_{1,s},\dots,\mu_{n,s}$ minus $1$. In particular, the multiplicity cannot exceed $n-1$.
	\item The remaining part of $E_s$ is mutually singular to each of the spectral measures $\mu_l$ and has multiplicity $1$.
\end{enumerate}
This theorem is a generalization and refinement of a theorem given by I.S.Kac in
\cite{Kac-1962}\footnote{Full proofs are provided in
\cite{Kac-1963} (in Russian). An English translation of this paper is not available, however, the proof was reproduced by D.Gilbert in
\cite{Gilbert-1998}: The operator-theoretic half of Kac' theorem is \cite[Theorem 5.1]{Gilbert-1998}, the measure-theoretic half is
\cite[Theorem 5.5, $(i)$]{Gilbert-1998}. An interesting approach to Kac' theorem was given recently by B.Simon in \cite{Simon-2005} who
proceeds via rank-one perturbations and uses Aronszajn-Donoghue theory.}. He considered the case of two edges and showed that the
spectral multiplicity of the singular part $E_s$ is always $1$.
Kac' Theorem corresponds to the upper bound for multiplicity in {\rm(II)} and simplicity of
spectrum in {\rm(III)}. Realizing a change of boundary condition of a half-line operator as an interface condition
with an ``artificial second edge'', we can also reobtain the half of Aronszajn-Donoghue theory which asserts disjointness of singular spectra
for different boundary conditions, see again \cite{Aronszajn-1957}, \cite{Donoghue-1965}, or \cite[Theorem 3.2, $(iv)$]{Gesztesy-Tsekanovskii-2000}.
This corresponds to the fact in {\rm(I)} that, if only one spectrum is present, it disappears.

By using the abstract framework of boundary relations, instead of just discussing a Schr\"odinger operator on a star-graph, we
achieve a slight generalization and a significant increase of flexibility in applications (various kinds of
operators, not necessarily being differential operators, can be treated). This bonus comes without additional effort,
since our proofs proceed via an analysis of Weyl functions and associated measures, and do not rely on the concrete form of the operators on edges.

The description of the absolutely continuous part $E_{ac}$ is not specific for the geometry of a star-graph and/or
the use of standard interface conditions: the mentioned result \cite[Theorem 6.6]{Gesztesy-Tsekanovskii-2000} holds for
arbitrary finite-rank perturbations. Contrasting this, the description of the singular part $E_s$ given in {\rm(I)}--{\rm(III)} is specific for
the particular situation. This is seen, for example, from some known formulas for the maximal multiplicity of an eigenvalue of a
Schr\"odinger operator on a graph (not necessarily a star-graph). It turns out that this number
depends on the geometry of the graph (rather than rank of the perturbation),
see \cite{Kac-Pivovarchik-2011} and the references therein.
Another good example is a theorem due to J.Howland, cf.\ \cite[2.Theorem]{Howland-1986}. There for a certain type of finite-rank perturbation
a behavior is witnessed which is fully in opposite to {\rm(I)}--{\rm(III)}. Also we should mention that, although the considered
operator $A$ is ``only'' a rank-one perturbation of the direct sum of the non-interacting operators $A_l$, classical perturbation
theory does not give much information. For example, the Kato-Rosenblum theorem deals with absolutely continuous spectrum, or
the theorem \cite[Satz~10.18]{Weidmann-2000} on the ranges of spectral projections yields information only for isolated eigenvalues.
The singular (continuous and possibly embedded) spectrum is much more instable, and its behavior is much harder to control.

Let us give a brief outline of the organization of the paper.
In the second part of this introductory section, we explain the structure of the spectrum of $A$ in some more detail (old and new results).
Section 2 is of preparatory nature. We set up notation and collect some results from the literature concerning:
spectral multiplicity, Borel measures, and Cauchy integrals.

In Section 3, we recall some facts about boundary relations and the
Titchmarsh-Kodaira formula. We define the main object of our studies, the pasting of boundary relations with standard interface conditions, cf. Definition \ref{K42},
and compute its matrix valued Weyl function in terms of the Weyl functions of the single boundary relations.
Moreover, we carry out the calculations required to determine the point spectrum.
Though this is of course included in our main result, we find it worth to be formulated and proved independently;
it serves as an elementary accessible, yet precise, model for the behavior of singular continuous spectrum.

Section 4 forms the core of the paper. In this section we formulate and prove our main result Theorem~\ref{K38}; the major task
is to get control of the singular continuous (possibly embedded) part of the singular spectrum.
The proof can be outlined as follows: We further divide the singular part $E_s$ into two summands.
Namely, setting $\mu=\sum_{l=1}^n\mu_l$, we decompose $E_s$ into the sum of a measure which is
absolutely continuous with respect to $\mu$ and one which is singular with respect to $\mu$.
First, we show that on null sets of the measure $\mu$ only simple spectrum of $A$ may appear, and this shows item {\rm(III)}.
Second, we consider points having certain ``good'' properties regarding existence of derivatives of involved measures and pointwise
asymptotics of their Poisson and Cauchy integrals. For such points the multiplicity of the spectrum can be calculated, and this shows items
{\rm(I)} and {\rm(II)} on the set of ``good'' points. Finally, we show that this set of ``good'' points in fact supports the full singular part of
$\mu$, and thereby complete the proof of items {\rm(I)} and {\rm(II)}.

The paper closes with two appendices. In the first appendix we provide some examples which show that all possibilities permitted by
{\rm(I)}--{\rm(III)} indeed may occur. These are not difficult to obtain and are based on classical theory and some more recent results
on concrete potentials on the half-line. This section will not hold many surprises for the specialist in the field; we
include it to give a fuller picture. In the second appendix we show how to reobtain from our present results
the classical theorem of Aronszajn and Donoghue on singular spectra associated with different boundary conditions. Moreover, we include
a short discussion of some (a few) interface conditions different from the standard condition.

There occurs an obvious open problem: Is it true that also for other finite-rank perturbations
the singular continuous spectrum behaves in the same way as the point spectrum (concerning its multiplicity)?
In a very general setting, one may think of investigating arbitrary finite rank perturbations; optimally obtaining a full
higher-dimensional analogue of Aronszajn-Donoghue theory for singular spectra. However, this is probably wishful thinking:
Keeping in mind the difficulties which arise when considering eigenvalues in the case of standard (Kirchhoff)
interface conditions on a graph with a somewhat more complicated geometric structure, already a thorough investigation of this
situation seems a challenging task.

At present, the answer to whatever version of the above posed question is not at all clear.
The computations we use in this paper are specific for the case ``star-graph+standard interface
conditions''. We plan to address this problem in future work.

\subsection*{Detailed description of the structure of $\bm{\sigma(A)}$}

Again, for the purpose of explaining, we consider a Schr\"odinger operator $A$ on a star-graph which is given
by the data $(1)$--$(3)$.

A first, rough, insight into the structure of the spectrum is provided by the
classical Titchmarsh-Kodaira formula. We may consider the operator $A$ as a
selfadjoint extension of the symmetry $S$ whose domain is defined by requiring
that $u_l(0)=u_l'(0)=0$, $l=1,\dots,n$. This symmetry is completely
non-selfadjoint, and has defect index $(n,n)$. The spectral multiplicity of
$A$ cannot exceed $n$: There exists an $n\!\times\!n$-matrix valued measure $\Omega$ such that the operator $A$ is unitarily equivalent to the operator of multiplication by the independent variable in the space $L_2(\mathbb R,\Omega)$. A measure $\Omega$ with this property can be constructed using Weyl theory. Since $A$ is an extension of $S$, there exists a matrix-valued Weyl function $M(z)$ corresponding to $A$. The measure $\Omega$ in the Herglotz-integral representation of $M$ has the required properties.

Since the spectral projection of the multiplication operator in $L_2(\mathbb R,\Omega)$ onto a Borel set $\Delta$ is the multiplication operator with the indicator function of $\Delta$, it follows that $\Omega$ and $E$ are mutually
absolutely continuous. If we set $\rho:=\tr\Omega$, then it is easy to see that $\Omega$ and $\rho$ are mutually absolutely continuous. We call the measure $\rho$ from this construction a \emph{scalar spectral measure} corresponding
to the operator $A$ (this measure is of course not unique).

The same procedure can be carried out for each of the operators $A_l$. For the operator $A_l$ the defect index of the minimal
operator is $(1,1)$, and one gets a unitary equivalence to the multiplication operator the space $L_2(\mathbb R,\mu_l)$,
where $\mu_l$ is the (now scalar) measure taken from the Herglotz-integral representation of the Weyl function associated with $A_l$.

Let $N_A(x)$ be the \emph{spectral multiplicity function} of $A$ which is defined $\rho$-a.e. The detailed definition of $N_A$ requires
some background; we recall it in \S2, see \eqref{K19}. Moreover, set $\mu:=\sum_{i=1}^n\mu_l$, and
\begin{equation}\label{r(x)}
	r(x):=\#\big\{l\in\{1,\dots,n\}:\,D_{\mu}{\mu_l}(x)>0\big\}
	\,.
\end{equation}
Here $D_{\mu}{\mu_l}(x)$ denotes the Radon-Nikodym derivative of $\mu_l$ with respect to $\mu$, and the function $r(x)$ is
defined $\mu$-a.e. Note that $\sum_{l=1}^nD_\mu\mu_l=1$ and hence $r(x)\geq 1$ for $\mu$-a.a.\ points $x\in\bb R$.

A complete description of the absolutely continuous part of the spectral measure $E$ of $A$ follows from
\cite[Theorem 6.6]{Gesztesy-Tsekanovskii-2000}.
Notation: We use $\sim$ to denote mutual absolute continuity of two measures.

\begin{theorem}[\cite{Gesztesy-Tsekanovskii-2000}]\label{K2}
	Let $A$ be a Schr\"odinger operator on a star-graph given by the data $(1)$--$(3)$ using the standard
	interface condition \eqref{K1}.
	Denote by $E$ the projection valued spectral measure of $A$, let $\mu$ be the sum of the scalar spectral measures of the
	non-interacting operators $A_l$, and let $E_{ac}$ and $\mu_{ac}$ be their absolutely continuous parts with respect to the Lebesgue measure.
	Moreover, let $N_A$ be the spectral multiplicity function of $A$, and let $r(x)$ be as in \eqref{r(x)}.	
	Then
	\begin{enumerate}[{\rm(I)}]
		\item $E_{ac}\sim\mu_{ac}$.
		\item $N_A(x)=r(x)$ for $E_{ac}$-a.a.\ points $x\in\bb R$.
	\end{enumerate}
\end{theorem}

\noindent
The following complete description of the singular part of the spectral measure $E$ of $A$ is the main result of this paper
(formulated for the Schr\"odinger case; the general statement is Theorem~\ref{K38}).
Notation: If $X$ is a Borel set, we write $\mathds{1}_X\cdot\nu$ for the measure acting as $(\mathds{1}_X\cdot\nu)(\Delta)=\nu(X\cap\Delta)$.

\begin{theorem}\label{K12}
	Let $A$ be a Schr\"odinger operator on a star-graph given by the data $(1)$--$(3)$ using the standard
	interface condition \eqref{K1}.
	Denote by $E$ the projection-valued spectral measure of $A$, let $\mu$ be the sum of the scalar spectral measures of the
	non-interacting operators $A_l$, and let $E_s$ and $\mu_s$ be their singular parts with respect to the Lebesgue measure.
	Let $E_{s,ac}$ and $E_{s,s}$ be the absolutely continuous and singular parts of $E_s$ with respect to $\mu$.
	Moreover, let $N_A$ be the spectral multiplicity function of $A$ and (as in \eqref{r(x)})
	\[
		r(x):=\#\big\{l\in\{1,\dots,n\}:\,D_{\mu}{\mu_l}(x)>0\big\}
		\,.
	\]
	Then
	\begin{enumerate}[{\rm(I)}]
		\item $E_{s,ac}\sim\mathds{1}_{X_{>1}}\cdot\mu_s$ where $X_{>1}:=r^{-1}(\{2,\dots,n\})$ .
		\item $N_A(x)=r(x)-1$ for $E_{s,ac}$-a.a.\ points $x\in\bb R$.
		\item $N_A(x)=1$ for $E_{s,s}$-a.a.\ points $x\in\bb R$.
	\end{enumerate}
\end{theorem}

\noindent
Notice that the Radon-Nikodym derivatives $D_{\mu}\mu_l$ and the number $r$ are defined $\mu$-a.e. The functions $N_A$ and $r$ should be considered as representatives of the equivalence classes under different equivalence relations. However the equality in item (II) makes sense and holds true $E_{s,ac}$-a.e. for any choice of such representatives, because the measure $E_{s,ac}$ is absolutely continuous with respect to both $E$ and $\mu$. In turn the set $X_{>1}$ is defined up to a $\mu$-zero set, but the measure $\mathds{1}_{X_{>1}}\cdot\mu_s$ is defined uniquely.

Finally, let us make explicit the behavior of the point spectrum.

\begin{theorem}\label{K3}
	Let $A$ be a Schr\"odinger operator on a star-graph given by the data $(1)$--$(3)$ using the standard
	interface condition \eqref{K1}. Denote by $m_l$ the Weyl functions of the non-interacting operators $A_l$, and $r(x)$ be as in \eqref{r(x)}.
	Then a point $x\in\bb R$ belongs to $\sigma_p(A)$, if and only if one of the following alternatives takes place.
	\begin{itemize}
		\item[{\rm(I/II)}] The point $x$ belongs to at least two of the point spectra $\sigma_p(A_l)$.
			In this case the multiplicity of the eigenvalue $x$ is equal to
			\[
				\#\big\{l=1,\dots,n:\,x\in\sigma_p(A_l)\big\}-1
				\,.
			\]
		\item[{\rm(III)}] The limits $m_l(x):=\lim_{\varepsilon\downarrow 0}m_l(x+i\varepsilon)$ all exist, are real,
			we have $\lim_{\varepsilon\downarrow 0}\frac 1{i\varepsilon}\big(m_l(x+i\varepsilon)-m_l(x)\big)\in[0,\infty)$, and
			$\sum_{j=1}^n m_l(x)=0$. In this case $x$ is a simple eigenvalue.
	\end{itemize}
\end{theorem}

\noindent
The connection of Theorem~\ref{K3} with the general result Theorem~\ref{K12} is made as follows:
For a point belonging to the point spectrum of at least one of the
operators $A_l$, we have $D_{\mu}\mu_k(x)>0$ if and only if $x\in\sigma_p(A_k)$ ($k\in\{1,\ldots,n\}$). Hence, for such points,
\[
	r(x)=\#\big\{l=1,\dots,n:\,x\in\sigma_p(A_l)\big\}
	\,.
\]
Moreover: If $x$ is an eigenvalue of only one operator $A_l$ it disappears. And the set of all points $x$ which satisfy the conditions stated
in {\rm(III)} is $\mu$-zero.

\section{Preliminaries}
\subsection{Some terminology concerning measures}

First of all, let us fix some measure theoretic language. We denote by $\mc B$ the $\sigma$-algebra of
all Borel sets on $\bb R$. All measures $\nu$ are understood to be Borel measures, and this includes the requirement that
compact sets have finite measure. Whenever writing ``$\nu(X)$'', this implicitly includes that $X\in\mc B$. If we speak of
a \emph{positive measure} $\nu$, this measure needs not to be finite. For a \emph{complex measure} $\nu$, we denote by $|\nu|$ its
\emph{total variation}, and this is always a finite positive measure. If a complex measure takes only real values,
we also speak of a \emph{real measure}.

In some places we have to deal with sets which are not necessarily Borel sets, and with functions which are not necessarily Borel measurable.
We say that $X$ is a \emph{$\nu$-zero set}, if $X\subseteq\bb R$ and there exists a Borel set $X'\supseteq X$ such that $\nu(X')=0$.
We say that a set $X\subseteq\mathbb R$ is \emph{$\nu$-full}, if its complement is $\nu$-zero.
A property is said to hold \emph{$\nu$-a.e.} or \emph{for $\nu$-a.a.\ points $x$}, if the set of all points where it holds is $\nu$-full.
Moreover, we say that a partially defined function $f$ is \emph{$\nu$-measurable},
if its domain is $\nu$-full and there exists a Borel measurable function
which coincides $\nu$-a.e.\ with $f$. Integrals $\int_{\bb R}f\,d\nu$ of $\nu$-measurable functions $f$ are defined accordingly.

Of course, such terminology could be avoided by considering $\nu$ as a measure on the completion
of the $\sigma$-algebra $\mc B$ with respect to $\nu$, and understanding measurability with respect to this larger $\sigma$-algebra.
However, then one has to work with different $\sigma$-algebras for different measures, and this would make things technically laborious.

When $\nu$ is a (positive or complex) measure, and $\sigma$ is a positive measure, we say that $\nu$ is \emph{absolutely continuous}
with respect to $\sigma$ (and write $\nu\ll\sigma$), if each $\sigma$-zero set is also $\nu$-zero. We say that $\nu$ and $\sigma$ are
\emph{mutually singular} (and write $\nu\perp\sigma$), if there exists a Borel set $\Delta$ which is $\nu$-full and $\sigma$-zero. Moreover,
we say that $\nu$ and $\sigma$ are \emph{mutually absolutely continuous} (and write $\nu\sim\sigma$), if $\nu\ll\sigma$ and $\sigma\ll\nu$.

Each measure $\nu$ has a (essentially unique) decomposition into a sum $\nu=\nu_{ac}+\nu_s$ of a measure $\nu_{ac}$ with $\nu_{ac}\ll\sigma$ and
a measure $\nu_s$ with $\nu_s\perp\sigma$; this is called the \emph{Lebesgue decomposition} of $\nu$ with respect to $\sigma$.
If $\nu\ll\sigma$, then there exists a (essentially unique) Borel measurable function $D_\sigma\nu$ with
\[
	\nu(\Delta)=\int_{\Delta}D_\sigma\nu\,d\sigma,\quad \Delta\in\mc B
	\,.
\]
This function is called the \emph{Radon-Nikodym derivative} of $\nu$ with respect to $\sigma$. It belongs to $L^1(\sigma)$, if
$\nu$ is a complex measure, and to $L^1_{\loc}(\sigma)$, if $\nu$ is a positive measure.

Let $\nu$ be a complex measure, and let $f\in L^1(\nu)$. Then we denote by $f\cdot\mu$ the measure which is absolutely continuous with respect to
$\nu$ and has Radon-Nikodym derivative $f$, i.e.,
\[
	(f\cdot\nu)(\Delta):=\int_{\Delta}f\,d\nu,\quad \Delta\in\mc B
	\,.
\]
In particular, if $X$ is a Borel set, we have $(\mathds{1}_X\cdot\nu)(\Delta)=\nu(X\cap\Delta)$, $\Delta\in\mc B$, where
$\mathds{1}_X$ denotes the indicator function of the set $X$. If $\nu$ is a positive measure, the same notation will be
applied when $f\in L^1_{\loc}(\nu)$, $f\geq 0$, and the product $f\cdot\nu$ will again be a positive measure.

The \emph{support} of a measure $\nu$ is the set
\[
    \supp\nu:=\big\{x\in\bb R:\,\nu([x-\varepsilon,x+\varepsilon])>0,\varepsilon>0\big\}
    =\bigcap_{\substack{A\text{ closed,}\\ \nu\text{-full}}}A
    \,.
\]
This notion must be distinguished from the notion of a \emph{minimal support} of the measure $\nu$. By this one means a
any Borel set $S$ with $\nu(\bb R\setminus S)=0$ and such that any set $S_0\subseteq S$ with $\nu(S_0)=0$ is also
Lebesgue zero.

All these notions also make sense when $\nu$ is a projection valued measure (like the spectral measure of a selfadjoint operator)
or a matrix valued measure (like the measure in the Herglotz integral representation of a matrix valued Herglotz function).

\subsection{The spectral multiplicity function}

In order to define the spectral multiplicity function, which measures the local multiplicity of the spectrum,
we have to provide some background material. These topics are of course classical, see, e.g.,
\cite{Akhiezer-Glazman-1993}, \cite{Birman-Solomyak-1987}, \cite{Reed-Simon-1980}.
Let $A$ be a (possibly unbounded) selfadjoint operator acting in some Hilbert space $H$, and denote by $E$ its projection-valued spectral measure.
A linear subspace $G$ of $H$ is called \emph{generating for $A$}, if (``$\bigvee$'' denotes the closed linear span)
\[
	\bigvee\big\{E(\Delta)G:\,\Delta\in\mathcal B\big\}=H
\]
The \emph{spectral multiplicity} of the operator $A$ is defined as the minimal dimension of a generating subspace, and denoted by $\mult A$.
If $\mult A=1$, one also says that $A$ has \emph{simple spectrum}.
For the sake of simplicity (and because this is all we need), we assume throughout the following that $\mult A<\infty$.

There exist (see, e.g., \cite[Theorem 7.3.7]{Birman-Solomyak-1987}) elements $g_l$, $l=1,\ldots,\mult A$, such that
the subspaces
$H_l:=\bigvee\{E(\Delta)g_l:\,\Delta\in\mathcal B\}$
are mutually orthogonal, the subspace $G:=\spn\{g_1,\ldots,g_{\mult A}\}$
is generating, and the scalar measures $\nu_l$ defined as $\nu_l(\Delta):=(E(\Delta)g_l,g_l)$, $\Delta\in\mc B$, satisfy
\begin{equation}\label{measures}
	\nu_{\mult A}\ll\dots\ll\nu_2\ll\nu_1\sim E
	\,.
\end{equation}
Each set $\{g_1,\ldots,g_{\mult A}\}$ with these properties is called a \emph{generating basis}.

If $\{g_1,\dots,g_{\mult A}\}$ is a generating basis,
the subspaces $H_l$ are mutually orthogonal and each of them reduces the operator $A$. The operator $A|_{H_l}$ has simple spectrum and is unitarily
equivalent to the operator of multiplication by the independent variable in the space $L_2(\mathbb R,\nu_{g_l})$. Thus $A$ is unitarily equivalent
to the multiplication operator in the space $\prod_{l=1}^{\mult A}L_2(\mathbb R,\nu_l)$. Consider the sets (which are defined up to
$E$-zero sets)
\[
    Y_l:=\big\{x\in\bb R:\,(D_{\nu_{g_1}}\nu_{g_l})(x)>0\big\}
    \,.
\]
Then
\[
	\mathbb R\supseteq Y_1\supseteq Y_2\supseteq\ldots\supseteq Y_{\mult A}
	\,,
\]
and these sets may be considered as layers of the spectrum. Hence, it is natural to define the \emph{spectral multiplicity function} of $A$ as
\begin{equation}\label{K19}
    N_A(x):=\#\big\{l:x\in Y_l\big\}
    \,.
\end{equation}
This function is defined almost everywhere with respect to the projection valued spectral measure of $A$ and does not depend on the choice of a generating basis $\{g_1,\ldots,g_{\mult A}\}$, cf.\
\cite[Theorem 7.4.2]{Birman-Solomyak-1987}. Spectral multiplicity function is a unitary invariant.

If $A$ is a selfadjoint linear relation, it can be orthogonally decomposed into a sum of a selfadjoint linear operator $A_{op}$ and a pure multivalued selfadjoint linear relation (the pure relational part of $A$). In this case, define $N_A:=N_{A_{op}}$. Obviously, this definition is also unitarily invariant.

Of course, the spectral multiplicity function of an eigenvalue is equal to the dimension of the corresponding eigenspace.

\subsection{Symmetric derivatives of measures}

In this subsection we recall the notion of the symmetric derivative of measures, and the formula of de la Vall\'ee-Poussin which describes the
Lebesgue decomposition of one positive Borel measure with respect to another. These topics are again classical, see, e.g., \cite{saks-1964}.
A presentation in an up-to-date language can be found, e.g., in \cite{diBenedetto-2002}.

\begin{theorem}[\cite{diBenedetto-2002}]\label{K5}
	Let $\nu$ and $\sigma$ be positive measures. Then there exists a Borel set $\mc E_{\nu,\sigma}\subseteq\supp\nu\cap\supp\sigma$ with
	\[
		\nu(\mc E_{\nu,\sigma})=\sigma(\mc E_{\nu,\sigma})=0
		\,,
	\]
	such that for each $x\in\supp\sigma\setminus\mc E_{\nu,\sigma}$ the limit
	\[
	    \lim_{\varepsilon\downarrow 0}\frac{\nu\big([x-\varepsilon,x+\varepsilon]\big)}{\sigma\big([x-\varepsilon,x+\varepsilon]\big)}
	\]
	exists in $[0,\infty]$ and defines a Borel measurable function.
\end{theorem}

\noindent
Due to this proposition, we can naturally define a function which is partially defined, $\nu$-measurable, and $\sigma$-measurable (but may be not Borel measurable).

\begin{definition}\label{K20}
	Let $\nu$ and $\sigma$ be positive measures. Then the \emph{symmetric derivative $\frac{d\nu}{d\sigma}$} of
	$\nu$ with respect to $\sigma$ is the partially defined function
	\begin{equation}\label{K9}
		\frac{d\nu}{d\sigma}(x):=
		\begin{cases}
			\lim\limits_{\varepsilon\downarrow 0}\frac{\nu([x-\varepsilon,x+\varepsilon])}{\sigma([x-\varepsilon,x+\varepsilon])}
				&\hspace*{-3mm},\quad x\in\supp\sigma\text{ and the limit exists in $[0,\infty]$},\\
			\infty &\hspace*{-3mm},\quad x\in\supp\nu\setminus\supp\sigma.\\
		\end{cases}
	\end{equation}

\end{definition}

\noindent
Note that this definition is symmetric in $\nu$ and $\sigma$ in the following sense:
\textit{If a point $x$ belongs to the domain of $\frac{d\nu}{d\sigma}$, then it also belongs to the domain of $\frac{d\sigma}{d\nu}$ and
$\frac{d\sigma}{d\nu}(x)=\big(\frac{d\nu}{d\sigma}(x)\big)^{-1}$.}

The symmetric derivative $\frac{d\nu}{d\sigma}$ can be used to explicitly construct the Lebesgue decomposition of $\nu$ with respect to
$\sigma$. To formulate this fact, denote
\[
    X_\infty(\nu,\sigma):=\Big\{x\in(\supp\nu\cup\supp\sigma)\setminus\mc E_{\nu,\sigma}:\,\frac{d\nu}{d\sigma}(x)=\infty\Big\}
    \,.
\]

\begin{theorem}[de la Vall\'ee-Poussin]\label{K7}
	Let $\nu$ and $\sigma$ be positive measures. Then
	\begin{enumerate}[$(i)$]
		\item The function $\frac{d\nu}{d\sigma}$ belongs to $L^1_{\loc}(\sigma)$. In particular, $\sigma(X_\infty(\nu,\sigma))=0$.
		\item For each Borel set $X\subseteq\bb R$, we have
			\[
				\nu(X)=\nu\big(X\cap X_\infty(\nu,\sigma)\big)+\int\limits_X\frac{d\nu}{d\sigma}(x)\,d\sigma(x)
				\,.
			\]
	\end{enumerate}
\end{theorem}

\noindent
Let $\nu=\nu_{ac}+\nu_s$ be the Lebesgue decomposition of $\nu$ with respect to $\sigma$. Then, indeed, the formula of de la Vall\'ee-Poussin says
that
\begin{equation}\label{K21}
	\nu_{ac}=\frac{d\nu}{d\sigma}(x)\cdot\sigma,\quad \nu_s=\mathds{1}_{X_\infty(\nu,\sigma)}\cdot\nu
	\,.
\end{equation}
In particular, if $\nu\ll\sigma$, then $\frac{d\nu}{d\sigma}$ is a Radon-Nikodym derivative of $\nu$ with respect to $\sigma$.

In the sequel we extensively use the following immediate consequences of the de la Vall\'ee-Poussin theorem.

\begin{corollary}\label{cor Vallee-Poussin sets}
	Let $\nu$ and $\sigma$ be positive measures, and let $X\subseteq\bb R$.
	\begin{enumerate}[$(i)$]
		\item If $\frac{d\nu}{d\sigma}(x)=0$ for all $x\in X$, then $X$ is $\nu$-zero.
		\item If the set $X$ is $\nu$-zero, then $\frac{d\nu}{d\sigma}(x)=0$ for $\sigma$-a.a. $x\in X$.
		\item If $X$ is a Borel set and $\frac{d\nu}{d\sigma}(x)\in[0,\infty)$ for all $x\in X$, then $\mathds{1}_X\cdot\nu\ll\sigma$.
	\end{enumerate}
\end{corollary}
\begin{proof}
	\hspace*{0pt}\\[0mm]\textit{Item $(i)$:}
	The set
	\[
		X_0(\nu,\sigma):=\Big\{x\in(\supp\nu\cup\supp\sigma)\setminus\mc E_{\nu,\sigma}:\,\frac{d\nu}{d\sigma}(x)=0\Big\}
	\]
	is a Borel set, and by Theorem~\ref{K7} we have $\nu(X_0(\nu,\sigma))=0$. Since $X\subseteq\mc E_{\nu,\sigma}\cup X_0(\nu,\sigma)$,
	the set $X$ is $\nu$-zero.

	\hspace*{0pt}\\[-2mm]\textit{Item $(ii)$:}
	There exists a Borel set $X'\supseteq X$ such that $\nu(X')=0$. Theorem~\ref{K7} gives $\int_{X'}\frac{d\nu}{d\sigma}(x)\,d\sigma(x)=0$.
	This shows that $\frac{d\nu}{d\sigma}(x)=0$ for $\sigma$-a.a.\ $x\in X'$, and hence for $\sigma$-a.a.\ $x\in X$.

	\hspace*{0pt}\\[-2mm]\textit{Item $(iii)$:}
	Let $X'\in\mc B$ with $\sigma(X')=0$ be given. Then also $\sigma(X'\cap X)=0$, and hence
	\[
		\nu(X'\cap X)=\int_{X'\cap X}\frac{d\nu}{d\sigma}(x)\,d\sigma(x)=0
		\,.
	\]
\end{proof}

\begin{corollary}\label{cor Vallee-Poussin measures}
	Let $\nu$ and $\sigma$ be positive measures on $\bb R$.	Let $\nu=\nu_{ac}+\nu_s$ and $\sigma=\sigma_{ac}+\sigma_s$
	be the Lebesgue decompositions of $\nu$ with respect to $\sigma$ and of $\sigma$ with respect to $\nu$, respectively.
	Then, the following hold:
	\begin{enumerate}[$(i)$]
	\item $\frac{d\nu}{d\sigma}(x)\in[0,\infty)$, $\sigma$-a.e.
	\item $\frac{d\nu}{d\sigma}(x)\in(0,\infty]$, $\nu$-a.e.
	\item $\frac{d\nu}{d\sigma}(x)\in(0,\infty)$, $\nu_{ac}$-a.e. and $\sigma_{ac}$-a.e.
	\item $\frac{d\nu}{d\sigma}(x)=\infty$, $\nu_s$-a.e.
	\item $\frac{d\nu}{d\sigma}(x)=0$, $\sigma_s$-a.e.
	\end{enumerate}
\end{corollary}
\begin{proof}
	Item $(i)$ is immediate from Theorem~\ref{K7}, $(i)$, and item $(ii)$ follows by exchanging the roles of $\nu$ and $\sigma$ and
	remembering that the symmetric derivative is symmetric in $\nu$ and $\sigma$.
	For $(iii)$, note that
	\[
		X:=\Big\{x\in\bb R:\,\frac{d\nu}{d\sigma}(x)\in(0,\infty)\Big\}^c\subseteq
		\mc E_{\nu,\sigma}\cup X_0(\nu,\sigma)\cup X_\infty(\nu,\sigma)
		\,.
	\]
	We have $\nu(\mc E_{\nu,\sigma})=\nu(X_0(\nu,\sigma))=0$, and $\sigma(X_\infty(\nu,\sigma))=0$. Thus the union of these sets
	is $\nu_{ac}$-zero. Exchanging the roles of $\nu$ and $\sigma$ yields that $X$ is also $\sigma_{ac}$-zero.

	From \eqref{K21}, immediately, $\nu_s(X_\infty(\nu,\sigma)^c)=0$. Hence, $(iv)$ holds.
	Item $(v)$ follows from $(iv)$ again by exchanging the roles of $\nu$ and $\sigma$.
\end{proof}

\begin{remark}\label{K22}
	One can also define a symmetric derivative of a complex measure $\nu$ with respect to a positive measure $\sigma$ by using the
	same limit
	\begin{equation}\label{K39}
		\frac{d\nu}{d\sigma}(x):=\lim\limits_{\varepsilon\downarrow 0}
		 \frac{\nu([x-\varepsilon,x+\varepsilon])}{\sigma([x-\varepsilon,x+\varepsilon])}
	\end{equation}
	whenever it exists in $\bb C$. However, satisfactory knowledge can only be obtained
	when $\nu\ll\sigma$. In fact, the following holds:
	\textit{If $\nu\ll\sigma$, then the limit \eqref{K39} exists $\sigma$-a.e., and is a Radon-Nikodym derivative
	of $\nu$ with respect to $\sigma$.} This follows since we can decompose $\nu$ as $\nu=(\nu_{r,+}-\nu_{r,-})+i(\nu_{i,+}-\nu_{i,-})$ with
	four positive and finite measures which are all absolutely continuous with respect to $\sigma$.
\end{remark}

\begin{remark}\label{K23}
	Sometimes the following facts are useful.
	\begin{enumerate}[$(i)$]
		\item Existence and value of the symmetric derivative are local properties in the sense that
			\[
				 \frac{d\nu}{d\sigma}(x)=\frac{d(\mathds{1}_X\cdot\nu)}{d(\mathds{1}_X\cdot\sigma)}(x)
			\]
			whenever $X$ is a Borel set which contains $x$ in its interior.
		\item If $f$ is a continuous and nonnegative function on $\bb R$, then
			\[
				\frac{d(f\cdot\nu)}{d\nu}(x)=f(x),\quad x\in\bb R
				\,.
			\]
	\end{enumerate}
\end{remark}

\subsection{Boundary behavior of Cauchy integrals}

Let us recall the notion of \emph{Herglotz functions}\footnote{Often also called \emph{Nevanlinna functions}.}.
In the present section our main focus lies on scalar valued functions. However, in view of our needs in Section 3, let us provide
the definition and the integral representation for matrix valued functions.

\begin{definition}\label{K53}
	A function $M:\bb C\setminus\bb R\to\bb C^{n\times n}$ is called a (\emph{$n\!\times\!n$-matrix valued}) \emph{Herglotz function}, if
	\begin{enumerate}[$(i)$]
		\item $M$ is analytic and satisfies $M(\overline z)=M(z)^*$, $z\in\bb C\setminus\bb R$.
		\item For each $z\in\bb C^+$, the matrix $\Im M(z):=\frac 1{2i}(M(z)-M(z)^*)$ is positive semidefinite.
	\end{enumerate}
\end{definition}

\noindent
The following statement is known as the \emph{Herglotz-integral representation}. For the scalar case, it goes back as far as to
\cite{Herglotz-1911}. For the matrix valued case see \cite[Theorem 5.4]{Gesztesy-Tsekanovskii-2000}, where also an extensive list of references
is provided.

\begin{theorem}\label{K52}
	Let $M$ be a $n\!\times\!n$-matrix valued Herglotz function. Then there exists a finite positive $n\!\times\!n$-matrix valued
	measure\footnote{By a \emph{positive matrix valued measure} we understand a measure which takes positive semidefinite matrices as values.}
	$\Omega$, a selfadjoint matrix $a$, and a positive semidefinite matrix $b$, such that
	\begin{equation}\label{K51}
		M(z)=a+bz+\int_{\bb R}\frac{1+xz}{x-z}\,d\Omega(x),\quad z\in\bb C\setminus\bb R
		\,.
	\end{equation}
	Conversely, each function of this form is a Herglotz function.
	
	The data $a,b,\Omega$ in the representation \eqref{K51} is uniquely determined by $M$. In fact, $\Omega$ can be recovered by the Stieltjes inversion formula, $b$ from the non-tangential asymptotics of $M(z)$ towards $+i\infty$, and $a$ from the real part of $M(i)$.
\end{theorem}

\noindent
In the literature this integral representation is often written in the form
\begin{equation}\label{K26}
	M(z)=a+bz+\int_{\bb R}\Big(\frac 1{x-z}-\frac x{1+x^2}\Big)\,d\tilde\Omega(x),\quad z\in\bb C\setminus\bb R
	\,,
\end{equation}
where $\tilde\Omega$ is a positive measure with $\int_{\bb R}\frac{d\tilde\Omega(x)}{1+x^2}<\infty$. The measures in the representations
\eqref{K51} and \eqref{K26} are related as $\tilde\Omega=(1+x^2)\cdot\Omega$.

In the present work we also consider the Cauchy-type integral in \eqref{K51} for complex (scalar valued) measures.

\begin{definition}\label{K57}
	Let $\nu$ be a complex (scalar valued) measure. Then we denote by $m_\nu$ the Cauchy-type integral
	\begin{equation}\label{K25}
		m_\nu(z):=\int_{\bb R}\frac{1+xz}{x-z}\,d\nu(x),\quad z\in\bb C\setminus\bb R
		\,.
	\end{equation}
\end{definition}

\noindent
Clearly, the function $m_\nu$ is analytic on $\bb C\setminus\bb R$. Moreover, note that for a real measure $\nu$
it can be written as the difference of two Herglotz functions.

\begin{remark}\label{K24}
	Here is the reason why we decided to write the Herglotz integral representation in the form \eqref{K51} rather than \eqref{K26}:
	For a positive measure $\Omega$ the multiplication $(1+x^2)\cdot\Omega$ is always defined and is again a positive measure (for scalar
	measures this is immediate, for matrix-valued measures use that $\Omega$ is mutually absolutely continuous with its trace
	measure $\rho:=\tr\Omega$).
	Contrasting this, for a complex measure $\nu$, the multiplication $(1+x^2)\cdot\nu$ cannot anymore be interpreted as a measure, but only
	as a distribution of order $0$. Since we want to avoid using the machinery of distributions, we decided for the representation \eqref{K51}.
\end{remark}

\noindent
For a finite positive measure $\nu$, the imaginary part of $m_\nu(z)$ is (as a Poisson integral) well-behaved and several explicit
relations between $\nu$ and the boundary behavior of $\Im m_\nu(z)$ at the real line are known. In the present context, the following
two pointwise relations play a role. The first one is standard, see, e.g., \cite[\S2.3]{Pearson-1988}. Matching the literature is done using Remark~\ref{K23}, $(ii)$.

\begin{theorem}[\cite{Pearson-1988}]\label{prop Herglotz}
	Let $\nu$ be a finite positive measure, and denote by $\lambda$ the Lebesgue measure.
	\begin{enumerate}[$(i)$]
		\item Assume that $\frac{d\nu}{d\lambda}(x)$ exists in $[0,\infty]$. Then $\Im m_\nu(z)$ has a normal boundary value at $x$,
			in fact,
			\[
				\lim_{\varepsilon\downarrow 0}\Im m_\nu(x+i\varepsilon)=\pi(1+x^2)\frac{d\nu}{d\lambda}(x)
				\,.
			\]
		\item Conversely, assume that $\Im m_\nu(z)$ has a finite normal boundary value at $x$. Then
			$\frac{d\nu}{d\lambda}(x)$ exists.
	\end{enumerate}
\end{theorem}

\noindent
Let us note explicitly that no conclusion is drawn if $\Im m_\nu(z)$ has an infinite normal boundary value at $x$.

The second result is in the same flavor, but may be less widely known. It is proved in \cite[Lemma 1]{Kac-1963}.

\begin{theorem}[\cite{Kac-1963}]\label{Kac lemma}
	Let $\nu$ be a real measure and $\sigma$ be a finite positive measure.
	Assume that $\frac{d\nu}{d\sigma}(x)$ exists in $\bb R$, and that $\frac{d\sigma}{d\lambda}(x)$ exists (possibly
	equal to $\infty$) and is nonzero. Then
	\[
		\lim_{\varepsilon\downarrow 0}\frac{\Im m_{\nu}(x+i\varepsilon)}{\Im m_{\sigma}(x+i\varepsilon)}=\frac{d\nu}{d\sigma}(x)
		\,.
	\]
\end{theorem}

\noindent
The real part of a Cauchy integral is (as a singular integral) much harder to control than its imaginary part. We make
use of the following two rather recent results, which deal with boundary values of $\Re m_\nu(z)$, or even of $m_\nu(z)$ itself.

The first one says that the set of points $x\in\bb R$ for which $|\Re m_\nu(z)|$ dominates $\Im m_\nu(z)$ when $z$ approaches $x$, is small.
It has been shown in \cite[Theorem 2.6]{Poltoratski-2003} for Cauchy integrals of measures on the unit circle. The half-plane version stated
below follows using the standard fractional-linear transform.

\begin{theorem}[\cite{Poltoratski-2003}]\label{theorem Poltoratski-1}
	Let $\nu$ be a finite positive measure.
	Then the set of all points $x\in\bb R$ for which there exists a continuous non-tangential path $\gamma_x$ from $\mathbb C_+$ to $x$,
	such that
	\[
		\lim_{\substack{z\to x\\ z\in\gamma_x}}\frac{|\Re m_{\nu}(z)|}{\Im m_{\nu}(z)}=\infty
		\,,
	\]
	is a $\nu$-zero set.
\end{theorem}

\noindent
The second result on singular integrals says that for certain points $x\in\bb R$ the Radon-Nikodym derivative of two measures can be
calculated from the boundary behavior of
the respective Cauchy integrals. This fact is shown in \cite{Poltoratski-1994} for the disk, the half-plane version stated below is
\cite[Theorem 1.1]{Poltoratski-2009}

\begin{theorem}[\cite{Poltoratski-1994}]\label{theorem Poltoratski-2}
	Let $\nu$ and $\sigma$ be finite positive measures, assume that $\nu\ll\sigma$, and let $\sigma_s$ be the singular part
	of $\sigma$ with respect to the Lebesgue measure. Then:
	\begin{enumerate}[$(i)$]
		\item For $\sigma$-a.a.\ points $x\in\bb R$ the non-tangential limit $\lim\limits_{\zlim}\frac{m_{\nu}(z)}{m_{\sigma}(z)}$
			exists in $[0,\infty)$.
		\item For $\sigma_s$-a.a.\ points $x\in\bb R$ we have
			\[
				 \lim_{\zlim}\frac{m_{\nu}(z)}{m_{\sigma}(z)}=\frac{d\nu}{d\sigma}(x)
				\,.
			\]
	\end{enumerate}
\end{theorem}

\noindent
Together these two theorems imply a statement which is essential for our present purposes.

\begin{corollary}\label{lemma Poltoratski}
	Let $\nu$ and $\sigma$ be finite positive measures, and assume that $\nu\ll\sigma$. Then for $\sigma$-a.a.\ points $x\in\bb R$ there
	exists a sequence $\{\varepsilon_n\}_{n\in\bb N}$ (which may depend on $x$), such that $\varepsilon_n\downarrow 0$ and the limit
	\[
		\lim_{n\to\infty}\frac{\Re m_{\nu}(x+i\varepsilon_n)}{\Im m_{\sigma}(x+i\varepsilon_n)}
	\]
	exists and is finite.
\end{corollary}
\begin{proof}
	Consider the sets
	\begin{align*}
		\mc E_1 & :=\Big\{x\in\bb R:\,\lim_{\varepsilon\downarrow 0}
			\frac{|\Re m_{\sigma}(x+i\varepsilon)|}{\Im m_{\sigma}(x+i\varepsilon)}=\infty\Big\}
			\,,
			\\
		\mc E_2 & :=\Big\{x\in\bb R:\,\lim_{\zlim}\frac{m_{\nu}(z)}{m_{\sigma}(z)}\text{ does not exist in $\bb C$}\Big\}
			\,.
	\end{align*}
	Then $\sigma(\mc E_1\cup\mc E_2)=0$. Let $x\in(\mc E_1\cup\mc E_2)^c$, and choose a sequence $\varepsilon_j\downarrow0$ such that the
	limit $\lim_{j\to\infty}\frac{\Re m_{\sigma}(x+i\varepsilon_j)}{\Im m_{\sigma}(x+i\varepsilon_j)}$ exists in $\bb R$. Since
	\[
		\frac{m_{\nu}(z)}{m_{\sigma}(z)}=\frac{m_{\nu}(z)}{\Im m_{\sigma}(z)\left(i+\frac{\Re m_{\sigma}(z)}{\Im m_{\sigma}(z)}\right)}
		,\quad z\in\bb C\setminus\bb R
		\,,
	\]
	it follows that
	\[
		\lim_{j\to\infty}\frac{\Re m_{\nu}(x+i\varepsilon_j)}{\Im m_{\sigma}(x+i\varepsilon_j)}=
		\Re\bigg[
		 \Big(\lim_{j\to\infty}\frac{m_{\nu}(x+i\varepsilon_j)}{m_{\sigma}(x+i\varepsilon_j)}\Big)
		\cdot
		\Big(i+\lim_{j\to\infty}\frac{\Re m_{\sigma}(x+i\varepsilon_j)}{\Im m_{\sigma}(x+i\varepsilon_j)}\Big)
		\bigg].
	\]
\end{proof}

\noindent
We need the above stated facts in a slightly more general situation. Namely, for arbitrary Herglotz functions rather than
Cauchy-type integrals. This is easy to deduce.

\begin{corollary}\label{K50}
	The above statements \ref{prop Herglotz}, \ref{Kac lemma}, \ref{theorem Poltoratski-1}, \ref{theorem Poltoratski-2}, $(ii)$, and
	\ref{lemma Poltoratski} remain true when the Cauchy-type
	integrals $m_\nu$ and $m_\sigma$ are substituted by arbitrary scalar valued Herglotz functions having the measures $\nu$ and
	$\sigma$ in their Herglotz integral representation.
\end{corollary}
\begin{proof}
	Throughout this proof, let $m,\tilde m$ be Herglotz functions, and write $m(z)=a+bz+m_\nu(z)$ and
	$\tilde m(z)=\tilde a+\tilde bz+m_\sigma(z)$.

	\hspace*{0pt}\\[-2mm]\textit{Theorem~\ref{prop Herglotz}:}
	We have
	\[
		\lim_{\varepsilon\downarrow 0}\Im\big(a+b(x+i\varepsilon)\big)=0,\quad x\in\bb R
		\,.
	\]
	Hence the limit $\lim_{\varepsilon\downarrow 0}\Im m(x+i\varepsilon)$ exists if and only if
	$\lim_{\varepsilon\downarrow 0}\Im m_\nu(x+i\varepsilon)$ does. Moreover, if these limits exist, they coincide.

	\hspace*{0pt}\\[-2mm]\textit{Theorem~\ref{Kac lemma}:}
	Since $\frac{d\sigma}{d\lambda}(x)$ exists and is nonzero, we have
	\[
		\lim_{\varepsilon\downarrow 0}\Im m_\sigma(x+i\varepsilon)>0
		\,.
	\]
	Thus
	\begin{multline*}
		\lim_{\varepsilon\downarrow 0}\frac{\Im m(x+i\varepsilon)}{\Im \tilde m(x+i\varepsilon)}=
		\lim_{\varepsilon\downarrow 0}\bigg[\Big(\frac{b\varepsilon}{\Im m_\sigma(x+i\varepsilon)}+
		\frac{\Im m_\nu(x+i\varepsilon)}{\Im m_\sigma(x+i\varepsilon)}\Big)\cdot
		\\
		\cdot\Big(1+\frac{\tilde b\varepsilon}{\Im m_\sigma(x+i\varepsilon)}\Big)^{-1}\bigg]=
		\lim_{\varepsilon\downarrow 0}\frac{\Im m_\nu(x+i\varepsilon)}{\Im m_\sigma(x+i\varepsilon)}
		\,.
	\end{multline*}
	
	\hspace*{0pt}\\[-2mm]\textit{Theorem~\ref{theorem Poltoratski-1}:}
	The set of all points $x\in\bb R$ with $\frac{d\nu}{d\lambda}(x)\in(0,\infty]$ is $\nu$-full. Hence, we may restrict all
	considerations to points $x$ belonging to this set, and hence assume that
	$\lim_{\varepsilon\downarrow 0}\Im m_\nu(x+i\varepsilon)>0$. Now use the estimate
	\begin{multline*}
		\frac{|\Re m_\nu(z)|}{\Im m_\nu(z)}-\frac{|a|+|b|\cdot|z|}{\Im m_\nu(z)}\leq
		\frac{|\Re m(z)|}{\Im m(z)}\Big(1+\frac{b\Im z}{\Im m_\nu(z)}\Big)\leq
		\\
		\leq\frac{|\Re m_\nu(z)|}{\Im m_\nu(z)}+\frac{|a|+|b|\cdot|z|}{\Im m_\nu(z)}
		\,.
	\end{multline*}
	
	\hspace*{0pt}\\[-2mm]\textit{Theorem~\ref{theorem Poltoratski-2}, $(ii)$:}
	For $\sigma_s$-a.a.\ points $x\in\bb R$ we have $\frac{d\sigma}{d\lambda}(x)=\infty$, and hence
	$\lim_{\varepsilon\downarrow 0}\Im m_\sigma(x+i\varepsilon)=\infty$. Thus also
	$\lim_{\varepsilon\downarrow 0}|m_\sigma(x+i\varepsilon)|=\infty$, and it follows that
	\[
		\lim_{\zlim}\frac{m(z)}{\tilde m(z)}=
		\lim_{\zlim}\bigg[\Big(\frac{a+bz}{m_\sigma(z)}+
		\frac{m_\nu(z)}{m_\sigma(z)}\Big)\cdot\Big(1+\frac{\tilde a+\tilde bz}{m_\sigma(z)}\Big)^{-1}\bigg]=
		\lim_{\zlim}\frac{m_\nu(z)}{m_\sigma(z)}
		\,.
	\]
	
	\hspace*{0pt}\\[-2mm]\textit{Corollary~\ref{lemma Poltoratski}:}
	For $\sigma$-a.a.\ points $x\in\bb R$ we have $\frac{d\sigma}{d\lambda}(x)\in(0,\infty]$, and hence
	$\lim_{\varepsilon\downarrow 0}\Im m_\sigma(x+i\varepsilon)>0$. Thus also $\lim_{\varepsilon\downarrow 0}\Im\tilde m(x+i\varepsilon)>0$,
	and it follows that
	\[
		\lim_{n\to\infty}\frac{\Re m(x+i\varepsilon_n)}{\Im\tilde m(x+i\varepsilon_n)}=
		\frac{a+bx}{\lim_{n\to\infty}\Im m_{\sigma}(x+i\varepsilon_n)}+
		\lim_{n\to\infty}\frac{\Re m_{\nu}(x+i\varepsilon_n)}{\Im m_{\sigma}(x+i\varepsilon_n)}
		\,.
	\]
\end{proof}

\noindent
\textbf{Convention:} \textit{When referring to one of the above statements \ref{prop Herglotz}, \ref{Kac lemma}, \ref{theorem Poltoratski-1},
\ref{theorem Poltoratski-2}, $(ii)$, \ref{lemma Poltoratski}, we mean their general versions provided in
Corollary~\ref{K50}.}

\section{Boundary relations}

Throughout the following we use without further notice the language and theory of linear relations. In particular, we
will think of a linear operator $T$ interchangeably as a map or as a linear relation
(i.e., identify the operator $T$ with its graph). Notationally, we interchangeably write $y=Tx$ or $(x;y)\in T$.

Our standard references for the theory of boundary triples are \cite{Derkach-Hassi-Malamud-deSnoo-2006} and the survey
article \cite{Derkach-Hassi-Malamud-deSnoo-2009}. There also some basic notations and results about linear relations can be
found.

\subsection{Boundary relations and Weyl families}

Boundary relations provide a general framework to study symmetric operators and their extensions.
Let us recall their definition, see, e.g., \cite[Definition 3.1]{Derkach-Hassi-Malamud-deSnoo-2009}.

\begin{definition}\label{K16}
	Let $S$ be a closed symmetric linear relation in a Hilbert space $H$, and let $B$ be an auxiliary Hilbert space.
	A linear relation $\Gamma\subseteq H^2\times B^2$ is called a \emph{boundary relation for $S^*$}, if
	\begin{axioms}{14mm}
	\item[BR1] The domain of $\Gamma$ is contained in $S^*$ and is dense there.
	\item[BR2] For each two elements $((f;g);(\alpha;\beta)),((f';g');(\alpha';\beta'))\in \Gamma$ the abstract Green's identity
		\[
			 (g,f')_{H}-(f,g')_{H}=(\beta,\alpha')_B-(\alpha,\beta')_B
		\]
		holds.
	\item[BR3] The relation $\Gamma$ is maximal with respect to the properties {\rm(BR1)} and {\rm(BR2)}.
	\end{axioms}
\end{definition}

\noindent
If the auxiliary space $B$ is finite-dimensional, the theory of boundary relations becomes significantly simpler. Since this is all we
need in the present paper, we will in most cases assume that $\displaystyle \dim B<\infty$.

A central notion is the Weyl family associated with a boundary relation, cf.\ \cite[Definition 3.4]{Derkach-Hassi-Malamud-deSnoo-2009}.

\begin{definition}\label{K18}
	Let $\Gamma\subseteq H^2\times B^2$ be a boundary relation for $S^*$.
	Then, for each $z\in\bb C\setminus\bb R$, we define a linear relation $M(z)$ as
	\[
		M(z):=\big\{(\alpha;\beta)\in B^2:\ \exists\,f\in H\text{ with }\big((f;zf);(\alpha;\beta)\big)\in\Gamma\big\}
		\,.
	\]
	This family of relations is called the \emph{Weyl family of $\Gamma$}. If $\mul M(z)=\{0\}$ for all $z$, one also refers
	to $M$ as the \emph{Weyl function of $\Gamma$}.
\end{definition}

\begin{definition}
A family of boundary relations $M(z)$ in the Hilbert space $B$ is called a \emph{Nevanlinna family}, if
	\begin{enumerate}[$(i)$]
		\item for each $z\in\bb C^+$, the relation $M(z)$ is maximal dissipative;
		\item $\displaystyle M(z)^*=M(\overline z),\quad z\in\bb C\setminus\bb R$;
		\item for some $w\in\bb C^+$ the operator valued function $z\mapsto(M(z)+w)^{-1}$ is holomorphic in $\bb C^+$.
	\end{enumerate}
\end{definition}

\noindent
The basic representation theorem for Weyl families reads as follows, see, e.g., \cite[Theorem~3.6]{Derkach-Hassi-Malamud-deSnoo-2009} or
\cite[Theorem~3.9]{Derkach-Hassi-Malamud-deSnoo-2006}.

\begin{theorem}[\cite{Derkach-Hassi-Malamud-deSnoo-2009}]\label{K62}
	Let $\Gamma\subseteq H^2\times B^2$ be a boundary relation for $S^*$, and let $M$ be its Weyl family. Then $M$ is	a Nevanlinna family. Conversely, each Nevanlinna family can be represented as the Weyl family of a boundary relation for the adjoint of some symmetric linear relation $S$. Moreover, $S$
	can be chosen to be completely non-selfadjoint (\emph{simple}).
\end{theorem}

\noindent
This representation theorem is accompanied by the following uniqueness result, see corresponding part in the proof of
\cite[Theorem~3.9]{Derkach-Hassi-Malamud-deSnoo-2006}.

\begin{theorem}[\cite{Derkach-Hassi-Malamud-deSnoo-2006}]\label{K60}
	Let $S_j$ be closed symmetric simple linear relations in Hilbert spaces $H_j$, and $\Gamma_j\subseteq H_j^2\times B^2$ be boundary relations for $S_j^*$, let $M_j$ be their Weyl families, $j=1,2$. If $M_1=M_2$, then there exists a unitary operator $U$ of $H_1$ onto $H_2$ such that
	\begin{equation}\label{unitarily equivalent relations} \Gamma_2=\Big\{\big((Uf;Ug);(\alpha;\beta)\big):\,\big((f;g);(\alpha;\beta)\big)\in\Gamma_1\Big\}
		\,.
	\end{equation}
\end{theorem}

\noindent
Two boundary relations which are related as in \eqref{unitarily equivalent relations} are called \emph{unitarily equivalent}.

The following properties, which a boundary relation may or may not possess, play a role in the present paper.

\begin{definition}\label{K58}
	Let $\Gamma\subseteq H^2\times B^2$ be a boundary relation for $S^*$.
	\begin{enumerate}[$(i)$]
		\item $\Gamma$ is called of \emph{function type}, if
			\[
				\mul\Gamma\cap\big(\{0\}\times B\big)=\{0\}
				\,.
			\]
		\item $\Gamma$ is called a \emph{boundary function}, if
			\[
				\mul\Gamma=\{0\}
				\,.
			\]
	\end{enumerate}
\end{definition}

\noindent
It is an important fact that these properties reflect in properties of the Weyl family associated with $\Gamma$.
For the following statement see, e.g., \cite[Proposition~3.7]{Derkach-Hassi-Malamud-deSnoo-2009} and
\cite[Lemma~4.1]{Derkach-Hassi-Malamud-deSnoo-2006}.

\begin{theorem}[\cite{Derkach-Hassi-Malamud-deSnoo-2006}]\label{K59}
	Let $\dim B<\infty$ and $\Gamma\subseteq H^2\times B^2$ be a boundary relation for $S^*$. Then
	\begin{enumerate}[$(i)$]
		\item $\Gamma$ is of function type, if and only if $\mul M(z)=\{0\}$, $z\in\bb C\setminus\bb R$.
		\item If $\Gamma$ is a boundary function, then $M(z)$ is an invertible operator in $B$ for every $z\in\mathbb C\setminus\mathbb R$.
	\end{enumerate}
\end{theorem}

\noindent
Given a boundary relation $\Gamma$ for $S^*$ which is of function type, we can single out a particular selfadjoint
extension of $S$. Namely, let $\pi_1:B^2\to B$ be the projection onto the first component, and set
\begin{equation}\label{K64}
	A:=\ker\big[\pi_1\circ\Gamma\big]
	\,.
\end{equation}
The fact that the relation $A$ is selfadjoint, follows from \cite[Proposition~3.16]{Derkach-Hassi-Malamud-deSnoo-2009}, since the auxiliary space $B$ is finite-dimensional.

\noindent
For later use, let us mention the following facts which follow from \cite[Proposition~3.2]{Derkach-Hassi-Malamud-deSnoo-2009}.

\begin{remark}\label{K56}
	Let $\dim B<\infty$ and $\Gamma\subseteq H^2\times B^2$ be a boundary relation for $S^*$.
	\begin{enumerate}[$(i)$]
		\item The relation $S$ has finite and equal defect indices $n_\pm(S)$. Moreover,
			\[
				\dim B=n_\pm(S)+\dim\mul\Gamma
				\,.
			\]
		\item Assume that $\dim B=1$. Then either $\mul\Gamma=\{0\}$ and $S$ has defect index $(1,1)$, or
			$\dim\mul\Gamma=1$ and $S$ is selfadjoint.
		\item Assume that $\dim B=1$, $\mul\Gamma\neq\{0\}$ and that $S$ is simple. Then either $\Gamma=\{0\}^2\times(\{0\}\times\mathbb C)$ or $\Gamma=\{0\}^2\times\{(w;mw),w\in\mathbb C\},m\in\mathbb R$, and the Weyl function is equal to $m$, a real constant\footnote{The case that $\mul\Gamma=\{0\}\times\bb C$ informally corresponds to the ``Weyl function'' $m\equiv\infty$ (formally to the Weyl family $m\equiv\{0\}\times\bb C$).}.
	\end{enumerate}
\end{remark}

\noindent
Next, we recall four methods to construct new boundary relations from given ones.
The first one is just taking orthogonal sums, the second is making a change of basis.
Both are easy to verify (and common knowledge); we skip the details.

\begin{lemma}\label{K61}
	Let $n\in\bb N$, and let for each $l\in\{1,\dots,n\}$ a boundary relation $\Gamma_l\subseteq H_l^2\times B_l^2$ for $S_l^*$ be given. Define
	\begin{multline*}
		\prod_{l=1}^n\Gamma_l:=\Bigg\{
		\Bigg(\bigg(\begin{pmatrix}f_1\\ \vdots\\ f_n\end{pmatrix};\begin{pmatrix}g_1\\ \vdots\\ g_n\end{pmatrix}\bigg);
		\bigg(\begin{pmatrix}\alpha_1\\ \vdots\\ \alpha_n\end{pmatrix};\begin{pmatrix}\beta_1\\ \vdots\\ \beta_n\end{pmatrix}\bigg)\Bigg):
		\\
		\big((f_l;g_l);(\alpha_l;\beta_l)\big)\in\Gamma_l,\ l=1,\dots,n\Bigg\}
		\subseteq\Big(\prod_{l=1}^n H_l\Big)^2\times\Big(\prod_{l=1}^nB_l\Big)^2
		\,.
	\end{multline*}
	Then $\prod_{l=1}^n\Gamma_l$ is a boundary relation for $\prod_{l=1}^nS_l^*$. Its Weyl family is given as
	\[
		\prod_{l=1}^nM_l(z):=\Bigg\{
		\bigg(\!\begin{pmatrix}\alpha_1\\ \vdots\\ \alpha_n\end{pmatrix}\!;\!\begin{pmatrix}\beta_1\\ \vdots\\ \beta_n\end{pmatrix}\!\bigg)\!:
		(\alpha_l;\beta_l)\in M_l(z),\ l\!=\!1,\dots,n\Bigg\}
		\!\subseteq\!\Big(\!\prod_{l=1}^nB_l\!\Big)^2
		.
	\]
	The relation $\prod_{l=1}^n\Gamma_l$ of function type (a boundary function) if and only if all relations $\Gamma_l$ are.
\end{lemma}

\begin{lemma}\label{K66}
	Let $\Gamma\subseteq H^2\times B^2$ be a boundary relation for $S^*$ of function type with Weyl function $M$. Moreover, let $U:B\to B$ be unitary, and define
	\[
		\Gamma_1:=\bigg\{\big((f;g);(U\alpha;U\beta)\big):\ \big((f;g);(\alpha;\beta)\big)\in\Gamma\bigg\}
		\subseteq H^2\times B^2
		\,.
	\]
	Then $\Gamma_1$ is a boundary relation for $S^*$ which is of function type, and the Weyl function $M_1$ is given as
	\[
		M_1(z)=UM(z)U^{-1}
		\,.
	\]
	The relation $\Gamma_1$ is a boundary function, if and only if $\Gamma$ is.
\end{lemma}

\noindent
The third construction shows how to realize the sum of two Weyl functions as a Weyl function,
cf.\ \cite[Corollary~4.5]{Derkach-Hassi-Malamud-deSnoo-2009}.

\begin{theorem}[\cite{Derkach-Hassi-Malamud-deSnoo-2009}]\label{K63}
Assume that $\dim B<\infty$. For $j=1,2$, let $\Gamma_j\subseteq H_j^2\times B^2$ be boundary relations for $S_j^*$ of function type with corresponding
	Weyl functions $M_j$. Define
	\begin{multline*}
		\Gamma:=\bigg\{\bigg(
		\Big(\begin{pmatrix}f_1\\ f_2\end{pmatrix};\begin{pmatrix}g_1\\ g_2\end{pmatrix}\Big);
		\big(\alpha;\beta_1+\beta_2\big)\bigg):
		\\
		\big((f_j;g_j);(\alpha;\beta_j)\big)\in\Gamma_j,j=1,2\bigg\}
		\subseteq\big(H_1\!\times\! H_2\big)^2\times B^2
		\,,
	\end{multline*}
	and
	\begin{multline*}
		S:=\bigg\{
		\Big(\begin{pmatrix}f_1\\ f_2\end{pmatrix};\begin{pmatrix}g_1\\ g_2\end{pmatrix}\Big):\quad
		\exists\beta\in B\text{ with}
		\\
\big((f_1;g_1);(0;\beta)\big)\in\Gamma_1,\big((f_2;g_2);(0;-\beta)\big)\in\Gamma_2\bigg\}
		\subseteq\big(H_1\!\times\! H_2\big)^2
		\,.
	\end{multline*}
	Then $S$ is a closed symmetric relation in $H_1\!\times\! H_2$, and $\Gamma$ is a boundary relation for $S^*$.
	The relation $\Gamma$ is of function type, and its Weyl function is given as
	\[
		M(z)=M_1(z)+M_2(z)
		\,.
	\]
\end{theorem}

\noindent
With the fourth procedure we construct a new boundary relation via a fractional linear transform. A proof can be found in \cite[Proposition~3.11]{Derkach-Hassi-Malamud-deSnoo-2009}. Before we can formulate this, let us introduce one more notation. We denote by $J_{\bb C^n}$ the $2n\!\times\!2n$-matrix
\[
	J_{\bb C^n}:=i
	\begin{pmatrix}
		0 & I_{\mathbb C^n}
		\\
		-I_{\mathbb C^n} & 0
	\end{pmatrix}
    .
\]
Then $J_{\bb C^n}$ defines a non-degenerated inner product on $\bb C^{2n}$. Let $w$ be a $2n\!\times\!2n$-matrix. Then $w$ is $J_{\bb C^n}$-unitary
(i.e., unitary with respect to the inner product induced by $J_{\bb C^n}$) if and only if
\begin{equation}\label{w}
	w^*J_{\bb C^n}w=J_{\bb C^n}
	\,.
\end{equation}

\begin{theorem}[\cite{Derkach-Hassi-Malamud-deSnoo-2009}]\label{prop frac-lin}
	Let $\Gamma\subseteq H^2\times(\bb C^n)^2$ be a boundary relation for $S^*$. Let $w$ be a $J_{\bb C^n}$-unitary
	$2n\!\times\!2n$-matrix, and write $w$ in block form as $w=(w_{ij})_{i,j=1}^2$ with $n\!\times\!n$-blocks $w_{ij}$, $i,j=1,2$.
	Then the composition
	\[
		\Gamma_1:=w\circ\Gamma=
		 \Big\{\big((f;g);(w_{11}\alpha+w_{12}\beta;w_{21}\alpha+w_{22}\beta)\big):\,\big((f;g);(\alpha;\beta)\big)\in\Gamma\Big\}
		\,.
	\]
is a boundary relation for $S^*$, and its Weyl family $M_1(z)$ is given as
	\begin{equation}\label{M-function}
		 M_1(z)=\big\{\big(w_{11}\alpha+w_{12}\beta;w_{21}\alpha+w_{22}\beta\big):\, (\alpha;\beta)\in M(z)\big\},\quad
		z\in\bb C\setminus\bb R
		\,.
	\end{equation}
\end{theorem}

\subsection{The Titchmarsh-Kodaira formula}

Let $\Gamma$ be a boundary relation of function type, and let $A$ be the selfadjoint relation \eqref{K64}. Then the data $a,b,\Omega$ in the
integral representation \eqref{K51} of the Weyl function $M$ associated with $\Gamma$ can be used to construct a functional model
for $A=\ker[\pi_1\circ\Gamma]$ acting as the multiplication operator in an $L^2$-space
(to be exact, the relational analogue of the multiplication operator).
For a Schr\"odinger operator on the half line (meaning limit-circle on one end and limit-point on the other)
or on the whole line (limit-point at both ends) this is a classical fact.
In the first case (where the spectral measure is scalar) this goes back to the initial considerations of H.Weyl, cf.\
\cite{Weyl-1910}, in the second case (where the spectral measure is $2\!\times\!2$-matrix valued) to the independent works of E.C.Titchmarsh and
K.Kodaira, cf.\ \cite{Titchmarsh-1962} and \cite{Kodaira-1949}. See also \cite[XII.5.Theorems 13 and 14]{Dunford-Schwartz-1963}, where
differential expressions of arbitrary order are studied.

In the present context this functional model plays an important role, since it allows us to compute the spectral multiplicity function
by computing the rank of a certain matrix; the precise statement being Proposition~\ref{K68} below.

First, let us recall the appropriate notion of a ``multiplication operator''.
Let $\Omega=(\Omega_{ij})_{i,j=1}^n$ be a positive $n\!\times\!n$-matrix valued measure, and denote by $\rho$ the (scalar) trace-measure $\rho:=\sum_{i=1}^n\Omega_{ii}$. Since, for each Borel set $\Delta$, the matrix $\Omega(\Delta)$ is positive semidefinite, we have
$|\Omega_{ij}(\Delta)|\leq\sqrt{\Omega_{ii}(\Delta)\Omega_{jj}(\Delta)}$, $i,j=1,\dots,n$, and this yields that $\Omega\sim\rho$. Hence, the Radon-Nikodym derivative $D_\rho\Omega=(D_\rho\Omega_{ij})_{i,j=1}^n$ is well-defined and $\rho$-a.e.\ positive.

Consider now the set of all $\rho$-a.e.\ finite functions $f:\mathbb R\mapsto\mathbb C^n$, such that each component is $\rho$-measurable
and such that
\[
	\int_{\mathbb R}\big(f(x),D_{\rho}\Omega(x)f(x)\big)_{\mathbb C^n}\,d\rho(x)<\infty
	\,.
\]
The space $L_2(\Omega)$ is the space of equivalence classes of such functions under the equivalence relation
\[
	f\sim g\text{ if and only if }\int_{\mathbb R}\big(f(x)-g(x),D_{\rho}\Omega(x)(f(x)-g(x))\big)_{\mathbb C^n}d\rho(x)=0
	\,.
\]
When endowed with the inner product\footnote{The right-hand side of this formula does not depend on the choice of the representative of the
equivalence class; as usual we slightly abuse notation.}
\[
	(f,g)_\Omega:=\int_{\mathbb R}\big(f(x),D_{\rho}\Omega(x)g(x)\big)_{\mathbb C^n}d\rho(x),\quad f,g\in L^2(\Omega)
	\,,
\]
this space becomes a Hilbert space. The operator of multiplication $A_t$ by the independent variable $t$ in this space is selfadjoint,
see \cite{Kac-1950} or \cite[XIII.5.Theorem 10]{Dunford-Schwartz-1963}.

Moreover, for a positive semidefinite $n\!\times\!n$-matrix $b$, denote by $G_b$ the space $\ran b$ endowed with the inner product defined as
\[
	(bx,by)_{G_b}:=(bx,y)_{\bb C^n},\quad x,y\in\bb C^n
	\,.
\]

\begin{definition}\label{K67}
	Let $\Omega$ be a positive $n\!\times\!n$-matrix valued measure, and let $b$ be a positive semidefinite $n\!\times\!n$-matrix. Then we
	set
	\[
		H_{\Omega,b}:=L^2(\Omega)\oplus G_b,\quad
		A_{\Omega,b}:=A_t\oplus\big(\{0\}\times G_b\big)
		\,.
	\]
\end{definition}

\noindent
Clearly, $H_{\Omega,b}$ is a Hilbert space and $A_{\Omega,b}$ is a selfadjoint linear relation in $H_{\Omega,b}$.

Let us now provide the afore mentioned functional model for $A=\ker[\pi_1\circ\Gamma]$.
The essence of this result is the model constructed in \cite[Proposition~5.2]{Derkach-Malamud-1995}, and it is easily deduced from this.

\begin{proposition}\label{K65}
	Let $\Gamma\subseteq H^2\times \mathbb C^{2n}$ be a boundary relation of function type with Weyl function $M$, let $a,b,\Omega$ be the data in the integral representation \eqref{K51} of $M$, and set $\tilde\Omega:=(1+t^2)\cdot\Omega$. Then the selfadjoint relation $A:=\ker[\pi\circ\Gamma]$ is unitarily equivalent to the relation $A_{\tilde\Omega,b}$.
\end{proposition}

\begin{proof}
	The following facts are well-known (for an explicit proof see, e.g., \cite[Proof of Theorem~3.2]{Behrndt-2009}): The kernel of $\Im M(z)$ for $z\in\mathbb C\setminus\mathbb R$ is independent of $z$. Denote its codimension by $n_0$. There exists a constant unitary matrix $U$ such that $U^{-1}\ker\Im M(i)=\{0\}\oplus\mathbb C^{n-n_0}$ and $U^{-1}(M(z)-a)U=M_0(z)\oplus0$ (the block form with respect to the decomposition $\mathbb C^n=\mathbb C^{n_0}\oplus\mathbb C^{n-n_0}$). At the same time $U^{-1}\tilde\Omega U=\tilde\Omega_0\oplus0$, $U^{-1}bU=b_0\oplus0$ and
\begin{equation*}
M_0(z)=b_0z+\int_{\bb R}\Big(\frac 1{x-z}-\frac x{1+x^2}\Big)\,d\tilde\Omega_0(x)
\end{equation*}
is the Herglotz-integral representation for the function $M_0$. The latter is such that $\Im M_0(z)$ is invertible for each $z\in\bb C\setminus\bb R$. We apply \cite[Proposition~5.2]{Derkach-Malamud-1995} to obtain a functional model for $M_0$: The relation
	\[
		S:=\bigg\{(f(t)\oplus 0;tf(t)\oplus y)\in H_{\tilde\Omega_0,b_0}^2:\ \int_{\bb R}d\tilde\Omega_0(t)f(t)+y=0\bigg\}
	\]
	is closed symmetric and simple, its adjoint is given by
	\[
		S^*:=\bigg\{(f(t)\oplus x;g(t)\oplus y)\in H_{\tilde\Omega_0,b_0}^2:\
		\exists h\in \mathbb C^{n_0}:\ g(t)-tf(t)=-h,x=b_0h\bigg\}
		\,.
	\]
	Obviously, the element $h$ in this formula is uniquely determined by the element $(f(t)\oplus x;g(t)\oplus y)$ of $S^*$. Hence,
	we may define
	\begin{align*}
		\Gamma_0\big((f(t)\oplus x;g(t)\oplus y)\big):= & h,
			\\
		\Gamma_1\big((f(t)\oplus x;g(t)\oplus y)\big):= & y+\int_{\bb R}d\tilde\Omega_0(t)\frac{tg(t)+f(t)}{1+t^2}
			\,.
	\end{align*}
	Then it follows from \cite{Derkach-Malamud-1995} that the relation
	\begin{multline*}
		\mr\Gamma:=\bigg\{\Big(\big(f(t)\oplus x;g(t)\oplus y\big);\big(\Gamma_0(f(t)\oplus x;g(t)\oplus y);
		\Gamma_1(f(t)\oplus x;g(t)\oplus y)\big)\Big):
		\\
		\big(f(t)\oplus x;g(t)\oplus y\big)\in S^*\bigg\}
	\end{multline*}
	is a boundary function for $S^*$. The Weyl function of $\mr\Gamma$ is equal to $M_0$, and $\ker\Gamma_0=A_{\tilde\Omega_0,b_0}$. Obviously $\Gamma_U:=\mr\Gamma\oplus(\{0\}^2\times\{(w;0),w\in\mathbb C^{n-n_0}\}$ is a boundary relation of function type for $S^*$ with the Weyl function $M_0\oplus0$ and $\ker[\pi_1\circ\Gamma_U]=A_{\tilde\Omega_0,b_0}$. Next, $\widehat\Gamma:=\{((f;g);(U\alpha;U\beta)):((f;g);(\alpha;\beta))\in\Gamma_U\}$ is another boundary relation of function type for $S^*$ with the Weyl function $M-a$ and $\ker[\pi_1\circ\widehat\Gamma]=A_{\tilde\Omega_0,b_0}$. Finally, the selfadjoint constant $a$ is the Weyl function of the boundary relation $\{0\}^2\times\{(w;aw),w\in \mathbb C^n\}$ acting in $\{0\}^2\times \mathbb C^{2n}$. Using Theorem~\ref{K63}, we obtain a boundary relation $\widetilde\Gamma$ having $M$ as its Weyl function. Explicitly computing $\widetilde\Gamma$ shows that the relation $\ker[\pi_1\circ\widetilde\Gamma]$ coincides with $A_{\tilde\Omega_0,b_0}$. Uniqueness part of Theorem \ref{K60} ensures that $\ker[\pi_1\circ\Gamma]$ is unitary equivalent to $\ker[\pi_1\circ\widetilde\Gamma]=A_{\tilde\Omega_0,b_0}$. Obviously, the relation $A_{\tilde\Omega_0,b_0}$ can be identified with the relation $A_{\tilde\Omega_0\oplus0,b_0\oplus0}=A_{U^{-1}\tilde\Omega U,U^{-1}b U}$, which is unitarily equivalent to $A_{\tilde\Omega,b}$.
\end{proof}

\noindent
Now we come to the promised way to compute the spectral multiplicity function.
For the case ``$n=2$'' this fact is proved and used in \cite{Kac-1962}, see also \cite{Gilbert-1998}. It is of course not hard to believe
that it holds for arbitrary $n\geq 2$, however, we are not aware of an explicit reference, and therefore provide a complete proof.

\begin{proposition}\label{K68}
	Let $\Gamma\subseteq H^2\times B^2$ be a boundary relation of function type. Denote by $M$ its Weyl function, and
	set $A:=\ker[\pi_1\circ\Gamma]$. Let $a,b,\Omega$ be the data in the integral representation \eqref{K51} of $M$,
	let $\rho$ be the trace measure of $\Omega$, and let $\omega$ be the symmetric derivative
	\begin{equation}\label{K77}
		\omega:=\frac{d\Omega}{d\rho}
		\,.
	\end{equation}
	Then ($N_A$ is the spectral multiplicity function of $A$)
	\begin{equation}\label{K78}
		N_A=\rank\omega,\quad \rho\text{-a.e.}
	\end{equation}
\end{proposition}

\begin{proof}
	From the readily established by Proposition~\ref{K65} unitary equivalence, we see that it is enough to compute the spectral multiplicity function of the multiplication operator $A_t$ in the space $L^2(\Omega)$. To do this, the idea is to construct a measurable ($\rho$-measurable, or Borel measurable on a compliment of some Borel $\rho$-zero set) diagonalization of $\omega(x)$. Once this is done, it is easy to give a unitarily equivalent form of $A$ (and a particular generating basis) from which $N_A(x)$
can be read off. The essential tool in the proof is Hammersley's theorem on the measurability of the zeros of a random polynomial, cf.\ \cite[Theorem~4.1]{Hammersley-1956}\footnote{See also \cite[Theorem~2.2]{BharuchaReid-Sambandham-1986}, where a short proof based on von Neumann's measurable selection theorem is given.}.

By Hammersley's theorem there exist measurable functions $\xi_1,\dots,\xi_n$ such that
	\[
		\det\big[\omega(x)-t\big]=(-1)^n\prod_{j=1}^n\big(t-\xi_j(x)\big),\quad x\in\bb R
		\,.
	\]
	Since $\omega(x)$ is nonnegative, we have $\xi_j(x)\geq 0$, $j=1,\dots,n$. By pointwise rearranging (which can be done in a measurable way)
	we can redefine the functions $\xi_j$, such that in addition
	\[
		0\leq\xi_n(x)\leq\xi_{n-1}(x)\leq\dots\leq\xi_1(x),\quad x\in\bb R
		\,.
	\]
	What follows is basic linear algebra. For $j,k\in\{1,\dots,n\}$ and $\{i_1,\dots,i_k\}\subseteq\{1,\dots,n\}$, $i_1<\dots<i_k$, set
	\begin{multline*}
		M^{j,k}_{i_1,\dots,i_k}:=\Big\{x\in\bb R:\ \rank\big[\omega(x)-\xi_j(x)\big]=k,
		\\
		\det\big(w_{i_li_m}(x)-\xi_j(x)\delta_{i_li_m}\big)_{l,m=1}^k\neq 0\Big\}
		\,.
	\end{multline*}
	The determinant of a matrix is a polynomial of the entries, and hence is measurable.	 The rank of a matrix depends, as the maximal order of an invertible square minor, measurably on the entries of the matrix.
	It follows that $M^{j,k}_{i_1,\dots,i_k}$ is a Borel set. Also the set
	\[
		M^{j,0}_\emptyset:=\big\{x\in\bb R:\ \omega(x)=\xi_j(x)\big\}
	\]
	is a Borel set.

	Let $j\in\{1,\dots,n\}$ be fixed. For each $x\in M^{j,k}_{i_1,\dots,i_k}$ the submatrix $(w_{i_li_m}(x)-\xi_j(x)\delta_{i_li_m})_{l,m=1}^k$ of $[\omega(x)-\xi_j(x)]$
is invertible. Applying Cramer's rule, we find a basis of the eigenspace $\ker[\omega(x)-\xi_j(x)]$ which depends measurably on $x\in M^{j,k}_{i_1,\dots,i_k}$. Applying the Gram-Schmidt orthogonalization procedure, we obtain an	orthonormal basis which also depends measurably on $x\in M^{j,k}_{i_1,\dots,i_k}$. For $x\in M^{j,0}_\emptyset$ the canonical basis of $\bb C^n$ is an orthonormal basis of $\ker[\omega(x)-\xi_j(x)]$.

Clearly, for each $j$, the sets $M^{j,k}_{i_1,\dots,i_k}$, $k\in\{0,\dots,n\}$, $\{i_1,\dots,i_k\}\subseteq\{1,\dots,n\}$ together cover the whole line. Hence, we can produce a disjoint covering of $\bb R$ with each set of the covering being a Borel subset of some intersection
	\[
		\bigcap_{j=1}^n M^{j,k_j}_{i_{j,1},\dots,i_{j,k_j}}
		\,.
	\]
	By the above paragraph, we can thus find a measurable orthonormal basis in $\bb C^n$ which consists of eigenvectors of $\omega$. The
	corresponding basis transform $U(x)$ is a measurable function and diagonalizes $\omega(x)$:
	\[
		U(x)^{-1}\omega(x)U(x)=D(x)
	\]
	with
	\[
		D(x):=
		\begin{pmatrix}
			\xi_1(x) & & \\
			& \ddots & \\
			& & \xi_n(x)
		\end{pmatrix}
		\,.
	\]
	The map $f\mapsto U^{-1}f$ is an isometric isomorphism of $L^2(\Omega)$ onto $L^2(D\cdot\rho)$, and establishes a unitary equivalence between the respective multiplication operators. One can regard $\xi_l\cdot\rho$ as $\nu_l$ from \eqref{measures}, so $D_{\nu_1}\nu_l(x)=\xi_l(x)$. Therefore the spectral multiplicity function computes as
	\[
		\#\big\{l\in\{1,\dots,n\}:\,\xi_l(x)>0\big\}=\rank D(x)=\rank\omega(x)
		\,.
	\]
\end{proof}

\subsection{Pasting of boundary relations with standard interface conditions}

Let $n\geq 2$, and let for each $l\in\{1,\dots,n\}$ a closed symmetric relation $S_l$ in a Hilbert space $H_l$ and a boundary
relation $\Gamma_l\subseteq H_l^2\times\bb C^2$ for $S_l^*$ be given. Moreover, denote
\begin{equation}\label{K34}
	A_l:=\ker\big[\pi_1\circ\Gamma_l\big]
	\,.
\end{equation}
Consider the Hilbert space $H:=\prod_{l=1}^n H_l$, the linear relation $S:=\prod_{l=1}^n S_l$ acting in this space, and the orthogonal sum $\widetilde\Gamma=\prod_{l=1}^n\Gamma_l$, cf. Lemma \ref{K61}.

Now we define another boundary relation $\Gamma$ by
using in Theorem~\ref{prop frac-lin} the $J_{\bb C^n}$-unitary matrix $w=(w_{ij})_{i,j=1}^2$ whose blocks $w_{ij}$ are given as
\begin{equation}\label{K40}
	w_{11}:=
	\begin{pmatrix}
		-1 & 0 & \cdots & 0 & 1
		\\
		0 & -1 & \cdots & 0 & 1
		\\
		\vdots & \vdots & \ddots & \vdots & \vdots
		\\
		0 & 0 & \cdots & -1 & 1
		\\
		0 & 0 & \cdots & 0 & 0
	\end{pmatrix}
	,\
	w_{12}:=	
	\begin{pmatrix}
		0 & 0 & \cdots & 0 & 0
		\\
		0 & 0 & \cdots & 0 & 0
		\\
		\vdots & \vdots & \ddots & \vdots & \vdots
		\\
		0 & 0 & \cdots & 0 & 0
		\\
		1 & 1 & \cdots & 1 & 1
	\end{pmatrix}
	\,,
\end{equation}
\begin{equation}\label{K41}
	w_{21}:=
	\begin{pmatrix}
		0 & 0 & \cdots & 0 & 0
		\\
		0 & 0 & \cdots & 0 & 0
		\\
		\vdots & \vdots & \ddots & \vdots & \vdots
		\\
		0 & 0 & \cdots & 0 & 0
		\\
		0 & 0 & \cdots & 0 & -1
	\end{pmatrix}
	,\
	w_{22}:=	
	\begin{pmatrix}
		-1 & 0 & \cdots & 0 & 0
		\\
		0 & -1 & \cdots & 0 & 0
		\\
		\vdots & \vdots & \ddots & \vdots & \vdots
		\\
		0 & 0 & \cdots & -1 & 0
		\\
		0 & 0 & \cdots & 0 & 0
	\end{pmatrix}
	\,.
\end{equation}
A straightforward computation shows that this matrix $w$ indeed satisfies \eqref{w}.
Explicitly, the relation $\Gamma$ is given as
\begin{multline}\label{K43}
	\Gamma:=w\circ\tilde\Gamma=\Bigg\{
	\Bigg(\bigg(\begin{pmatrix}f_1\\ \vdots\\ f_n\end{pmatrix};\begin{pmatrix}g_1\\ \vdots\\ g_n\end{pmatrix}\bigg);
	\bigg(\begin{pmatrix}-\alpha_1+\alpha_n\\ \vdots\\ -\alpha_{n-1}+\alpha_n\\ \beta_1+\dots+\beta_n\end{pmatrix};
	\begin{pmatrix}-\beta_1\\ \vdots\\ -\beta_{n-1}\\ -\alpha_n\end{pmatrix}\bigg)\Bigg):
		\\
	\big((f_l;g_l);(\alpha_l;\beta_l)\big)\in\Gamma_l,\ l=1,\dots,n\Bigg\}
	\,.
\end{multline}

\begin{definition}\label{K42}
	Let $n\geq 2$, and let for each $l\in\{1,\dots,n\}$ a closed symmetric relation $S_l$ in a Hilbert space $H_l$ and a boundary	relation $\Gamma_l\subseteq H_l^2\times\bb C^2$ for $S_l^*$ be given. Assume that
	\begin{axioms}{15mm}
	\item[Hyp1] Each relation $S_l$ is simple.
	\item[Hyp2] Each boundary relation $\Gamma_l$ is of function type.
	\item[Hyp3] There exists $l\in\{1,\dots,n\}$, such that $\Gamma_l$ is a boundary function.
	\end{axioms}
Obviously, the relation $\Gamma$ depends on the order in which the relations $\Gamma_1,\Gamma_2,...,\Gamma_n$ are taken, but the selfadjoint relation $A:=\ker[\pi_1\circ\Gamma]$ does not. Thus we call $A$ the relation constructed	 by \emph{pasting the family $\{\Gamma_l:\,l=1,\dots,n\}$, with standard interface conditions}.
\end{definition}

\noindent
The hypothesis {\rm(Hyp1)} ensures that knowledge about the spectrum of $A$ can be deduced from the associated Weyl family (in fact, the Weyl function, see below).
The hypotheses {\rm(Hyp2)}, {\rm(Hyp3)}, are required in order to avoid trivial cases (remember Remark~\ref{K56}).

To justify our choice of terminology, let us return to our model example.

\begin{example}\label{K36}
	Let a Schr\"odinger operator on a star-shaped graph be given by the data $(1)$--$(3)$. Let $S_l$ be the minimal operator
	on the $l$-th edge, i.e.\
	\begin{align*}
		\dom S_l:=\Big\{u\in L^2&(0,e_l):\,u,u'\text{ absolutely continuous},
		\\
		& -u''+q_lu\in L_2(0,e_l),
		\\
		& u(0)=u'(0)=0,\ u\text{ satisfies b.c.\ at }e_l,\text{ if present}\Big\}
		\,.
	\end{align*}
	\[
		S_lu:=-u''+q_lu,\quad u\in\dom S_l
		\,.\rule{58mm}{0pt}
	\]
	Moreover, define
	\[
		\Gamma_l:=\big\{\big((u;v);(u(0);u'(0)\big):\, (u;v)\in S_l^*\big\}
		\,.
	\]
	Then $\Gamma_l$ is a boundary relation for $S_l^*$.
	The selfadjoint extension $A_l=\ker[\pi_1\circ\Gamma_l]$ is just the Schr\"odinger operator given by the potential $q_l$ with Dirichlet boundary conditions at $0$.
	
	Now consider the boundary relation $\Gamma$ defined by \eqref{K43}.
	The operator $A:=\ker[\pi_1\circ\Gamma]$, is nothing but the operator defined by
	\eqref{K10}, \eqref{K11}, using the standard interface condition \eqref{K1}.
\end{example}

\noindent
In order to understand the spectrum of a pasting with standard interface conditions, we will analyze the Weyl function of the boundary relation
$\Gamma$. Using Theorem~\ref{prop frac-lin}, this Weyl function can be
computed explicitly in terms of the Weyl functions of the boundary relations $\Gamma_l$.

\begin{proposition}\label{K37}
	Let $\Gamma_l$ be as in Definition~\ref{K42}, and let $\Gamma$ be the boundary relation given by \eqref{K43}. Denote by $m_l$ the Weyl function of $\Gamma_l$, and set $m:=\sum_{l=1}^n m_l$. Then we have $\mul\Gamma\cap(\{0\}\times\bb C^n)=\{0\}$, and the Weyl function $M$ of $\Gamma$ is given as
	\begin{equation}\label{M}
		M=\frac1{m}\!
		\begin{pmatrix}
			m_1(m\!-\!m_1)&-m_2m_1&\cdots&-m_{n-1}m_1&-m_1\\
			-m_1m_2&m_2(m\!-\!m_2)\!&\cdots&-m_{n-1}m_2&-m_2\\
			\vdots&\vdots&\ddots&\vdots&\vdots\\
			-m_1m_{n-1}&-m_2m_{n-1}&\cdots&\!m_{n-1}(m\!-\!m_{n-1})&-m_{n-1}\\
			-m_1&-m_2&\cdots&-m_{n-1}&-1
		\end{pmatrix}
        .
	\end{equation}
\end{proposition}
\begin{proof}
	Consider an element of $\mul\Gamma\cap(\{0\}\times\bb C^n)$. By the definition of $\Gamma$ there exist
	$(\alpha_l;\beta_l)\in\mul\Gamma_l$, $l=1,\dots,n$, such that this element is equal to
	\[
		\Bigg(\begin{pmatrix}-\alpha_1+\alpha_n\\ \vdots\\ -\alpha_{n-1}+\alpha_n\\ \beta_1+\dots+\beta_n\end{pmatrix};
		\begin{pmatrix}-\beta_1\\ \vdots\\ -\beta_{n-1}\\ -\alpha_n\end{pmatrix}\Bigg)
=
		\Bigg(\begin{pmatrix}0\\ \vdots\\0\\0\end{pmatrix};
		\begin{pmatrix}-\beta_1\\ \vdots\\ -\beta_{n-1}\\ -\alpha_n\end{pmatrix}\Bigg)
		\,.
	\]
	By {\rm(Hyp3)}, there exists an index $l_0$ with $\alpha_{l_0}=\beta_{l_0}=0$. Now it follows that $\alpha_l=0$ for all $l$, and {\rm(Hyp2)} implies that also $\beta_l=0$ for all $l$. This shows that $\mul\Gamma\cap(\{0\}\times\bb C^n)=\{0\}$.
	
	The Weyl function of the boundary relation $\widetilde\Gamma$ is
	\[
		\tilde M=
		\begin{pmatrix}
			m_1 & 0 & \cdots & 0\\
			0 & m_2 & \cdots & 0\\
			\vdots & \vdots & \ddots & \vdots\\
			0 & 0 & \cdots & m_n
		\end{pmatrix}
		\ .
	\]
	Computation gives:
	\[
		w_{11}+w_{12}\tilde M=
		\begin{pmatrix}
			-1 & \cdots & 0 & 1\\
			\vdots & \ddots & \vdots & \vdots\\
			0 & \cdots & -1 & 1\\
			m_1 & \cdots & m_{n-1} & m_n
		\end{pmatrix}
		\ ,
	\]
	\[
		w_{21}+w_{22}\tilde M=
		\begin{pmatrix}
			m_1 & \cdots & 0 & 0\\
			\vdots & \ddots & \vdots & \vdots\\
			0 & \cdots & -m_{n-1} & 0\\
			0 & \cdots & 0 & -1
		\end{pmatrix}
		\ ,
	\]
	and
	\[
		\det\big(w_{11}+w_{12}\tilde M\big)=(-1)^{n-1}\sum_{l=1}^nm_l
		\,.
	\]
	Since there exists at least one index $l$ such that $m_l$ is not a real constant, this determinant does not vanish throughout $\bb C\setminus\bb R$.

	Next, let $M$ be the matrix defined by \eqref{M}. It is easy to check that $M$ satisfies
$M(w_{11}+w_{12}\tilde M)=w_{21}+w_{22}\tilde M$, and this implies that
	\[
		M=(w_{21}+w_{22}\tilde M)(w_{11}+w_{12}\tilde M)^{-1}
		\,.
	\]
	Theorem~\ref{prop frac-lin} now yields that $M$ is indeed the Weyl function of $\Gamma$.
\end{proof}

\subsection{The point spectrum}

It is elementary to locate the point spectrum of a pasting.

\begin{theorem}\label{K35}
	Let $n\geq 2$, and let for each $l\in\{1,\dots,n\}$ a closed symmetric relation $S_l$ in a Hilbert space $H_l$ and a boundary	relation $\Gamma_l\subseteq H_l^2\times\bb C^2$ for $S_l^*$ be given. Assume that these data are subject to	{\rm(Hyp1)}--{\rm(Hyp3)}, and consider the selfadjoint operator $A$ constructed by pasting $\{\Gamma_l:\,l=1,\dots,n\}$	 with standard interface conditions.

Let $x\in\bb R$. Then $x\in\sigma_p(A)$ if and only if one of the following two alternatives takes place.
	\begin{itemize}
		\item[{\rm(I/II)}] The point $x$ belongs to at least two of the point spectra $\sigma_p(A_l)$. In this case its multiplicity $N_A$ of an eigenvalue is equal to
			\begin{align*}
				N_A(x)= &\ \#\big\{l\in\{1,\dots,n\}:\,x\in\sigma_p(A_l)\big\}-1
				\\
				= &\ \#\big\{l\in\{1,\dots,n\}:\,\lim_{\varepsilon\downarrow 0}\varepsilon\Im m_l(x+i\varepsilon)>0\big\}-1
				\,.
			\end{align*}
		\item[{\rm(III)}] The limits $m_l(x):=\lim_{\varepsilon\downarrow 0}m_l(x+i\varepsilon)$ all exist, are real,
			we have $\lim_{\varepsilon\downarrow 0}\frac 1{i\varepsilon}\big(m_l(x+i\varepsilon)-m_l(x)\big)\in[0,\infty)$, and
			$\sum_{l=1}^n m_l(x)=0$. In this case $x$ is a simple eigenvalue.
	\end{itemize}
\end{theorem}

\noindent
First we prove a technical statement which is an immediate consequence of {\rm(Hyp1)}, {\rm(Hyp2)}.

\begin{lemma}\label{K54}
	Let $S$ be a closed symmetric simple relation in a Hilbert space $H$, let $\Gamma\subseteq H^2\times\bb C^2$ be a boundary relation for $S^*$ of function type. Set $A:=\ker[\pi_1\circ\Gamma]$.
	\begin{enumerate}[$(i)$]
		\item For each $(f;g)\in A$ there exists unique $\beta\in\bb C$ such that $((f;g);(0;\beta))\in\Gamma$.
		\item Let $\Upsilon:A\to\bb C$ be defined as $\Upsilon(f;g)=\beta$ where $((f;g);(0;\beta))\in\Gamma$, and let
			$x\in\bb R$. Then the restriction of $\Upsilon$ to the set $\{(f;xf):\,f\in\ker(A-x)\}$ is injective.
		\item Let $x\in\bb R$ and $\alpha\in\bb C\setminus\{0\}$. Then there exists at most one element of the form
			$((f;xf);(\alpha;\beta))$ which belongs to $\Gamma$.
	\end{enumerate}
\end{lemma}

\begin{proof}
	Existence of $\beta$ is the definition of $A$. Uniqueness follows since $\mul\Gamma\cap(\{0\}\times\bb C)=\{0\}$. This shows $(i)$
	and that the map $\Upsilon$ in $(ii)$ is well-defined. Since $\Gamma$ is minimal, $S$ is completely non-selfadjoint.
	We clearly have $\ker\Upsilon\subseteq S$, and hence $\ker\Upsilon\cap\{(f;xf):\,f\in\ker(A-x)\}=\{0\}$. This shows $(ii)$.
	
	To show $(iii)$, assume that $((f;xf);(\alpha;\beta)),((f';xf');(\alpha;\beta'))\in\Gamma$ and that $f\neq f'$. Without loss
	of generality assume that $f\neq 0$. By minimality of $\Gamma$ this implies that (remember Remark~\ref{K56})
	\[
		\mul\Gamma=\{0\},\quad \dim S^*/S=2,\quad \dim\mc N_x\leq 1,\ x\in\bb R
		\,.
	\]
	Let $\lambda$ be such that $f'=\lambda f$, then $((0;0);((\lambda-1)\alpha;\lambda\beta-\beta'))\in\Gamma$. Since $\lambda\neq 1$ and
	$\alpha\neq 0$, we obtain $\mul\Gamma\neq\{0\}$ a contradiction. We conclude that $f=f'$. Since $\mul\Gamma\cap(\{0\}\times\bb C)=\{0\}$,
	it follows that also $\beta=\beta'$.
\end{proof}

\begin{proof}[Proof of Theorem \ref{K35}]
	\hspace*{0pt}\\[-0mm]\textit{Step 1; A preliminary observation:}
	In this step, we show that
	\[
		\bigg(\begin{pmatrix}f_1\\ \vdots\\ f_n\end{pmatrix};\begin{pmatrix}g_1\\ \vdots\\ g_n\end{pmatrix}\bigg)\in A
		\quad\Longrightarrow\quad
		\exists!\,\alpha,\beta_l:\ \big((f_l;g_l);(\alpha;\beta_l)\big)\in\Gamma_l
	\]
	Moreover, for these numbers $\beta_l$, it holds that $\beta_1+\dots+\beta_n=0$.

	By the definition of $A$ and $\Gamma$, cf.\ \eqref{K43}, there exist $\alpha_1,\dots,\alpha_n$ and $\beta_1,\dots,\beta_n$ such that
	\[
		\big((f_l;g_l);(\alpha_l;\beta_l)\big)\in\Gamma_l,\quad -\alpha_1+\alpha_n=\dots=-\alpha_{n-1}+\alpha_n=\beta_1+\dots+\beta_n=0
		\,.
	\]
	This proves the existence part (set $\alpha:=\alpha_1$). For uniqueness, assume that $\alpha'$ and $\beta_1',\dots,\beta_n'$ are such that
$((f_l;g_l);(\alpha';\beta_l'))\in\Gamma_l$. By {\rm(Hyp3)} there exists an index $l_0\in\{1,\dots,n\}$ with $\mul\Gamma_{l_0}=\{0\}$.
	It follows that $\alpha'=\alpha$ and $\beta_{l_0}'=\beta_{l_0}$. Due to {\rm(Hyp2)} it follows that for all indices $\beta_l'=\beta_l$.
	
	\hspace*{0pt}\\[-2mm]\textit{Step 2; Two examples of eigenvectors:}
	Let $x\in\bb R$. First, consider the space
	\[
		\mc L_x:=\bigg\{\begin{pmatrix}f_1\\ \vdots\\ f_n\end{pmatrix}\in\prod_{l=1}^n\ker(A_l-x):\
		\sum_{l=1}^n\Upsilon_l(f_l;xf_l)=0\bigg\}
		\,.
	\]
	Since each map $\Upsilon_l$ is injective, we have
	\[
		\dim\mc L_x=
		\begin{cases}
			0 &\hspace*{-3mm},\text{ if for every }l\,\ker(A_l-x)=\{0\},\\
			\sum_{l=1}^n\dim\ker(A_l-x)-1 &\hspace*{-3mm},\text{ if for some }l\,\ker(A_l-x)\neq\{0\}.\\
		\end{cases}
	\]
	It is clear that $\mc L_x\subseteq\ker(A-x)$.

	Second, assume that there exist elements $f_l\in\ker(S_l^*-x)$ and $\beta_l\in\bb C$ such that
	$((f_l;xf_l);(1;\beta_l))\in\Gamma_l$ and $\beta_1+\dots+\beta_n=0$. Then, clearly,
	$(f_1,\dots,f_n)\in\ker(A-x)$.

	\hspace*{0pt}\\[-2mm]\textit{Step 3; Determining the eigenspace:}
	Let $x\in\bb R$ and $(f_1,\dots,f_n)\in\ker(A-x)\setminus\{0\}$ be given, and let $\alpha$ and $\beta_1,\dots,\beta_n$ be the
	unique numbers with $((f_l;xf_l);(\alpha;\beta_l))\in\Gamma_l$, $l=1,\dots,n$. We distinguish the two cases that $\alpha=0$ and
	$\alpha\neq 0$.

	Assume that $\alpha=0$. Then $f_l\in\ker(A_l-x)$ and $\beta_l=\Upsilon_l(f_l;xf_l)$. Hence, in this case,
	$(f_1,\dots,f_n)\in\mc L_x$.

	Assume that $\alpha\neq 0$. Clearly, $f_l\in\mc N_{l,x}$ for all $l$. Moreover, whenever $\mul\Gamma_l=\{0\}$, we
	must have $f_l\neq 0$. Let us show that
	\[
		\ker(A_l-x)=\{0\},\quad l=1,\dots,n
		\,.
	\]
	If $f_l=0$, then $\mul\Gamma_l\neq\{0\}$, and
	hence $H_l=\{0\}$. If $f_l\neq 0$ and $\ker(A_l-x)\neq\{0\}$, then $f_l\in\ker(A_l-x)$ since $\dim\mc N_{l,x}\leq 1$. Thus there
	exists $\beta'_l$ with $((f_l;xf_l);(0;\beta'_l))\in\Gamma_l$, and it follows that $\mul\Gamma_l\neq\{0\}$.
	This contradicts the fact that $f_l\neq 0$.

	Next we show (still assuming $\alpha\neq 0$) that
	\[
		\ker(A-x)=\spn\big\{(f_1,\dots,f_n)\big\}
		\,.
	\]
	Let $g=(g_1,\dots,g_n)\in\ker(A-x)\setminus\{0\}$
	be given, and let $\alpha'$ and $\beta_1',\dots,\beta_n'$ be the unique numbers with $((g_l;xg_l);(\alpha';\beta_l'))\in\Gamma_l$,
	$l=1,\dots,n$. If $\alpha'=0$, we would have $\ker(A_l-x)\neq\{0\}$ for at least one index $l$. This contradicts what we showed
	in the previous paragraph, and we conclude that $\alpha'\neq 0$. Set $\lambda:=\frac{\alpha'}{\alpha}$. Then
	$((\lambda f_l;x\lambda f_l);(\alpha';\lambda\beta_l))\in\Gamma_l$, and it follows from Lemma~\ref{K54}, $(iii)$, that $g_l=\lambda f_l$.

	Putting together these facts with what we showed in Step 2, we obtain that for each real point $x$ one of the following three
	alternatives holds:
	\begin{enumerate}[$(i)$]
		\item $\ker(A-x)=\{0\}$.
		\item There exist at least two indices $l$ with $\ker(A_l-x)\neq\{0\}$.
		\item We have $\ker(A_l-x)=\{0\}$, $l=1,\dots,n$, and $\mc N_{l,x}\neq\{0\}$ whenever $\mul\Gamma_l=\{0\}$.
	\end{enumerate}
	If the alternative $(ii)$ takes place, then $\ker(A-x)=\mc L_x$. If $(iii)$ takes place, then $\dim\ker(A-x)=1$.

	\hspace*{0pt}\\[-2mm]\textit{Step 4; Asymptotics of $m_l$:}
	The characterizations stated in the theorem now follow from standard Weyl function theory. First, a point $x\in\bb R$ is an eigenvalue of
	$A_l$ if and only
	\[
		\lim_{\varepsilon\downarrow 0}\varepsilon\Im m_l(x+i\varepsilon)>0
		\,.
	\]
	Next, assume that $\mul\Gamma_l=\{0\}$. Then we have $\mc N_{l,x}\setminus\ker(A_l-x)\neq\emptyset$ if and only if
	\[
		m_l(x):=\lim_{\varepsilon\downarrow 0}m_l(x+i\varepsilon)\in\bb R,\quad
		\lim_{\varepsilon\downarrow 0}\frac 1{i\varepsilon}\big(m_l(x+i\varepsilon)-m_l(x)\big)\in(0,\infty)
		\,.
	\]
	If $\mul\Gamma_l\neq\{0\}$, then $m_l$ is identically equal to a real constant, and hence trivially
	$\lim_{\varepsilon\downarrow 0}m_l(x+i\varepsilon)$ exists in $\bb R$ and
	$\lim_{\varepsilon\downarrow 0}\frac 1{i\varepsilon}\big(m_l(x+i\varepsilon)-m_l(x)\big)=0$. Conversely, if these two
	relations hold, the function $m_l$ must be a real constant.

	Finally, we need to relate the limit $m_l(x)$ with the boundary relation $\Gamma_l$ under the assumption that this limit at all exists, and that $\lim_{\varepsilon\downarrow 0}\frac 1{i\varepsilon}\big(m_l(x+i\varepsilon)-m_l(x)\big)\in[0,\infty)$. For $\varepsilon>0$, let $f_{l,\varepsilon}$ and $\beta_{l,\varepsilon}$ be the unique elements with
	\[
		 \big((f_{l,\varepsilon};(x+i\varepsilon)f_{l,\varepsilon});(1;\beta_{l,\varepsilon})\big)\in\Gamma_l
		\,.
	\]
	Then, by the definition of $m_l$, we have $m_l(x+i\varepsilon)=\beta_{l,\varepsilon}$. The abstract Green's identity gives
	\[
		\varepsilon\|f_{l,\varepsilon}\|^2=\Im m_l(x+i\varepsilon)
		\,,
	\]
	and hence $\|f_{l,\varepsilon}\|$ remains bounded when $\varepsilon$ approaches $0$. Let $f_l$ be the weak limit of $f_{l,\varepsilon}$ for $\varepsilon\downarrow 0$. Since $\Gamma_l$ is a closed linear relation, it is weakly closed, and we obtain
	\[
		\big((f_l;xf_l);(1;m_l(x))\big)\in\Gamma_l
		\,.
	\]
	This finishes the proof.
\end{proof}

\section{Computation of rank for singular spectrum}

The following theorem is our main result, and this section is entirely devoted to its proof. Concerning terminology for boundary relations,
remember Definition~\ref{K58}.

\begin{theorem}\label{K38}
	Let $n\geq 2$, and let for each $l\in\{1,\dots,n\}$ a closed symmetric simple relation $S_l$ in a Hilbert space $H_l$ and a boundary relation $\Gamma_l\subseteq H_l^2\times\bb C^2$ for $S_l^*$ be given. Assume that each $\Gamma_j$ is of function type, and that at least one $\Gamma_l$ is a boundary function. Let $\mu_l$ be the measure in the Herglotz-integral representation of the Weyl function of $\Gamma_l$, set	 $\mu:=\sum_{l=1}^n\mu_l$, and let $\mu_s$ be the singular part of $\mu$ with respect to the Lebesgue measure.

Consider the selfadjoint operator $A$ constructed by pasting $\{\Gamma_l:\,l=1,\dots,n\}$ with standard interface conditions. Denote by $E$ the projection valued spectral measure of $A$, let $E_s$ be its singular part with respect to the Lebesgue measure, and let $E_{s,ac}$ and $E_{s,s}$ be the absolutely continuous and singular parts of $E_s$ with respect to $\mu$. Moreover, let $N_A$ be the spectral multiplicity function of $A$ and
	\[
		r(x):=\#\big\{l\in\{1,\dots,n\}:\,D_{\mu}{\mu_l}(x)>0\big\}
		\,.
	\]
	Then the following hold:
	\begin{enumerate}[{\rm(I)}]
		\item $E_{s,ac}\sim\mathds{1}_{X_{>1}}\cdot\mu_s$ where $X_{>1}:=r^{-1}(\{2,\dots,n\})$ .
		\item $N_A(x)=r(x)-1$ for $E_{s,ac}$-a.a.\ points $x\in\bb R$.
		\item $N_A(x)=1$ for $E_{s,s}$-a.a.\ points $x\in\bb R$.
	\end{enumerate}
\end{theorem}

\begin{remark}\label{K69}
	We may assume without loss of generality that all boundary relations $\Gamma_l$ with possible exception of $\Gamma_n$ are boundary functions: First, reordering the boundary relations $\Gamma_l$ obviously does not change the relation $A$. Second, it is easy to see that the pasting with standard interface conditions of a collection as given in the theorem is always unitary equivalent to the pasting with standard interface conditions of a collection with at most one boundary relation being a proper relation (take instead of $k$ pure relations $\{((0;0);(w;\beta_lw)),w\in\mathbb C\}$, $l=l_1,...,l_k$, one relation $\{((0;0);(w;\sum_{j=1}^k\beta_{l_j}w)),w\in\mathbb C\}$).
\end{remark}

\noindent
Let us recall some notation:
The Weyl functions of the boundary relations $\Gamma_l$ are denoted as $m_l$, and we set $m:=\sum_{l=1}^nm_l$. The boundary relation $\Gamma$ is as in \ref{K43}, and we denote by $M(z)=(M_{ij}(z))_{i,j=1}^n$ its Weyl function. Explicitly, the function $M$ is given by \eqref{M}. Let $\Omega=(\Omega_{ij})_{i,j=1}^n$ be the $n\!\times\!n$-matrix valued measure in the Herglotz-integral representation \eqref{K25} of $M$, let $\rho$ be the trace measure $\rho:=\tr\Omega$, and let $\omega=(\omega_{ij})_{i,j=1}^n$ be the symmetric derivative of $\Omega$ with respect to $\rho$, i.e.\
\[
	\omega_{ij}(x):=\frac{d\Omega_{ij}}{d\rho}(x),\quad i,j=1,\dots,n
	\,,
\]
whenever these derivatives exist.

Moreover, remember \ref{K68} which says that the actual task is to compute $\rank\omega(x)$.

\begin{flushleft}
	\textbf{Stage 1: $\bm\mu$-singular part.}
\end{flushleft}
\vspace*{-2mm}
In this part, we prove the following statement.

\begin{proposition}\label{lemma mu-zero}
	If the set $X\subseteq\mathbb R$ is $\mu$-zero, then for $\rho$-a.a.\ points $x\in X$ the symmetric derivative $\omega(x)$ exists and
	$\rank\omega(x)=1$.
\end{proposition}

\noindent
The proof is split into two parts. First, the case when the right lower entry $M_{nn}$ of $M$ dominates.

\begin{lemma}\label{lemma d-nn>0}
	Let $x\in\mathbb R$, and assume that
	\begin{enumerate}[$(i)$]
		\item The symmetric derivative $\omega(x)$ exists.
		\item $\frac{d\rho}{d\lambda}(x)=\infty$.
		\item $\omega_{nn}(x)>0$.
	\end{enumerate}
	Then the limits $m_l(x):=\lim\limits_{\varepsilon\downarrow 0}m_l(x+i\varepsilon)$, $l=1,\dots,n$, exist, are real, and
	\[
		m(x):=\lim_{\varepsilon\downarrow 0}m(x+i\varepsilon)=0
		\,.
	\]
	The rank of the matrix $\omega(x)$ is equal to one.
\end{lemma}
\begin{proof}
	\hspace*{0pt}\\[-0mm]\textit{Step 1. Existence of limits.}
	The present hypotheses imply that
	\begin{align}
		\frac{d\Omega_{nn}}{d\lambda}(x) & =\frac{d\Omega_{nn}}{d\rho}(x)\cdot\frac{d\rho}{d\lambda}(x)=\infty
			\,,\label{K13}
			\\
		\frac{d\Omega_{ij}}{d\Omega_{nn}}(x) & =\frac{d\Omega_{ij}}{d\rho}(x)\cdot\frac{d\rho}{d\Omega_{nn}}(x)=
			\frac{\omega_{ij}(x)}{\omega_{nn}(x)}
			\,.\label{K14}
	\end{align}
	Applying Theorem~\ref{prop Herglotz} with $\Omega_{nn}$ gives
	$\lim_{\varepsilon\downarrow 0}\Im M_{nn}(x+i\varepsilon)=\infty$.
	However, $M_{nn}=-\frac 1m$, and we thus have
	\[
		\lim_{\varepsilon\downarrow 0}\Big|\Im\frac 1{m(x+i\varepsilon)}\Big|=\infty
		\,.
	\]
	In particular, $\lim_{\varepsilon\downarrow 0}m(x+i\varepsilon)=0$.
	Since $\Im m_l$ is nonnegative throughout the upper half-plane, this implies that also
	$\lim_{\varepsilon\downarrow 0}\Im m_l(x+i\varepsilon)=0$, $l=1,\dots,n$.
	
	In order to capture behavior of real parts, we apply Theorem~\ref{Kac lemma} with the measures $\Omega_{ij}$ and $\Omega_{nn}$. This gives
	\[
		\lim_{\varepsilon\downarrow 0}\frac{\Im M_{ij}(x+i\varepsilon)}{\Im M_{nn}(x+i\varepsilon)}=\frac{\omega_{ij}(x)}{\omega_{nn}(x)}
		\,.
	\]
	Let $l\in\{1,\dots,n-1\}$. For each $z\in\bb C^+$, we have
	\begin{equation}\label{M-ln/M-nn}
			\frac{\Im M_{ln}(z)}{\Im M_{nn}(z)}=\frac{\Im\frac{m_l(z)}{m(z)}}{\Im\frac{1}{m(z)}}=
			\Re m_l(z)+\frac{\Im m_l(z)\Re\frac1{m(z)}}{\Im\frac1{m(z)}}
			\,,
	\end{equation}
	and
	\[
		\left|\Im m_l(z)\Re\frac1{m(z)}\right| \leq \Im m_l(z)\Big|\frac 1{m(z)}\Big| \leq \frac{\Im m_l(z)}{\Im m(z)}\leq 1
		\,.
	\]
	Hence, the limit of $\Re m_l$ exists, in fact,
	\[
		\lim_{\varepsilon\downarrow 0}\Re m_l(x+i\varepsilon)=\frac{\omega_{ln}(x)}{\omega_{nn}(x)}
		\,.
	\]
	Since we already know that imaginary parts tend to zero, thus
	\[
		\lim_{\varepsilon\downarrow 0}m_l(x+i\varepsilon)=\frac{\omega_{ln}(x)}{\omega_{nn}(x)},\quad l=1,\dots,n-1.
	\]
	Since $m$ tends to zero, it follows that
	$\lim_{\varepsilon\downarrow 0}m_n(x+i\varepsilon)=-\sum_{l=1}^{n-1}\frac{\omega_{ln}(x)}{\omega_{nn}(x)}$.
	
	\hspace*{0pt}\\[-2mm]\textit{Step 2. Computing rank.}
	Let $l\in\{1,\dots,n-1\}$. For each $z\in\bb C^+$, we have
	\begin{equation}\label{Im M-ll}
	\begin{aligned}
		\Im M_{ll}(z)= & \Im\frac{m_l(z)[m(z)-m_l(z)]}{m(z)}=
			\\
		= & \Im\frac1{m(z)}\Re\big(m_l(z)[m(z)-m_l(z)]\big)+
			\\
		& +\Re\frac1{m(z)}\Im\big(m_l(z)[m(z)-m_l(z)]\big)
			\,,
	\end{aligned}
	\end{equation}
	and
	\begin{align*}
		\bigg|\Re\frac1{m(z)}\Im\big(m_l(z)[m(z)- & m_l(z)]\big)\bigg|\leq
			\\
		= & \left|\Re\frac1{m(z)}\cdot\Im m_l(z)\cdot\Re\big(m(z)-m_l(z)\big)\right|+
			\\
		& +\left|\Re\frac1{m(z)}\cdot\Re m_l(z)\cdot\Im\big(m(z)-m_l(z)\big)\right|\leq
			\\
		\leq & \left|\Re\frac1{m(z)}\right|\cdot\Im m(z)\cdot\sum_{j=1}^{n}|\Re m_j(z)|\leq
			\\
		\leq & \frac{\Im m(z)}{|m(z)|}\cdot\sum_{j=1}^{n}|\Re m_j(z)| \leq \sum_{j=1}^{n}|\Re m_j(z)|
			\,.
	\end{align*}
	Since $m_j(x+i\varepsilon)$ approaches a real limit when $\varepsilon\downarrow 0$, and $m(x+i\varepsilon)$ tends to zero, and
	$|\Im\frac 1{m(x+i\varepsilon)}|$ tends to infinity, we obtain
	(the ``$o(1)$'' understands for $\varepsilon\downarrow 0$)
	\[
		\Im M_{ll}(x+i\varepsilon)=-\Im\frac1{m(x+i\varepsilon)}\cdot\big(m_l^2(x)+o(1)\big),\quad l=1,\dots,n-1
		\,.
	\]
	Arguing analogously, we obtain
	\begin{align*}
		\Im M_{lk}(x+i\varepsilon) & = -\Im\frac1{m(x+i\varepsilon)}\cdot\big(m_l(x)m_k(x)+o(1)\big),
			\\
		& \hspace*{45mm}l,k=1,\dots,n-1,\ l\neq k\,,
			\\
		\Im M_{ln}(x+i\varepsilon) & = -\Im\frac1{m(x+i\varepsilon)}\cdot\big(m_l(x)+o(1)\big),\quad l=1,\dots,n-1
			\,.
	\end{align*}
	Therefore
	\[
		\Im\tr M(x+i\varepsilon)=-\Im\frac1{m(x+i\varepsilon)}\Big(1+\sum_{l=1}^{n-1}m_l^2(x)+o(1)\Big)
		\,,\hspace*{16.5mm}
	\]
	and hence, referring again to Theorem~\ref{Kac lemma},
	\begin{multline*}
		\omega(x)=\frac1{1+\sum\limits_{l=1}^{n-1}m_l^2(x)}
			\\[2mm]
		\times
			\left(
			\begin{array}{ccccc}
				m_1^2(x) & m_2(x)m_1(x) & \cdots & m_{n-1}(x)m_1(x) & m_1(x)\\
				m_1(x)m_2(x) & m_2^2(x) & \cdots & m_{n-1}(x)m_2(x) & m_2(x)\\
				\vdots & \vdots & \ddots & \vdots & \vdots\\
				m_1(x)m_{n-1}(x) & m_2(x)m_{n-1}(x) & \cdots & m_{n-1}^2(x) & m_{n-1}(x)\\
				m_1(x) & m_2(x) & \cdots & m_{n-1}(x) & 1
			\end{array}
			\right)
	\end{multline*}
	Obviously, the rank of this matrix is $1$.
\end{proof}

\noindent
Second, the case that the right lower entry of $M$ does not dominate. In this case, a more refined argument is necessary.
First, two technical observations.

\begin{lemma}\label{lemma tech 1}
	Let $D\subseteq\bb C$ be a connected set with $x\in\overline{D}$. Moreover, let $f,g:D\to\bb C^+$ be continuous functions.
	If $\lim_{t\to x}\frac{f(t)}{g(t)}=-1$, then
	\begin{equation}\label{K15}
		\lim_{t\to x}\frac{\Im f(t)}{\Re f(t)}=\lim_{t\to x}\frac{\Im g(t)}{\Re g(t)}=0
		\,.
	\end{equation}
\end{lemma}

\begin{proof}
	For $z\in\bb C^+$, let $\arg z$ denote the branch of the argument of $z$ in $[0,\pi]$. Then $\arg f$ and $\arg g$ are continuous functions.
	We have
	\[
		\lim_{t\to x}\big[\arg f(t)-\arg g(t)\big]=\pi\text{ mod }2\pi
		\,,
	\]
	and hence either
	\[
		\lim_{t\to x}\arg f(t)=\pi\quad\text{and}\quad\lim_{t\to x}\arg g(t)=0
		\,,
	\]
	or
	\[
		\lim_{t\to x}\arg f(t)=0\quad\text{and}\quad\lim_{t\to x}\arg g(t)=\pi
		\,.
	\]
	In both	cases, \eqref{K15} follows.
\end{proof}

\begin{lemma}\label{lemma tech 2}
	Let $\alpha\in\mathbb R$, and let $\{f_j\}_{j\in\bb N}$ and $\{g_j\}_{j\in\bb N}$ be sequences of complex numbers. Assume that
	\begin{enumerate}[$(i)$]
		\item For each $j\in\bb N$ we have $\Im g_j\neq 0$.
		\item $\lim\limits_{j\to\infty}\frac{f_j}{g_j}=\alpha$.
		\item $\frac{\Re g_j}{\Im g_j}=O(1)$ as $j\to\infty$.
	\end{enumerate}
	Then $\lim\limits_{j\to\infty}\frac{\Im f_j}{\Im g_j}=\alpha$.
\end{lemma}
\begin{proof}
	Let $\varepsilon>0$. Then, for sufficiently large indices $j$, we have
	\[
		\left|\frac{f_j}{g_j}-\alpha\right|<\varepsilon
		\,.
	\]
	This implies that $\big|\Im f_j-\alpha\Im g_j\big|\leq|f_j-\alpha g_j|<\varepsilon|g_j|$, and hence
	\[
		\left|\frac{\Im f_j}{\Im g_j}-\alpha\right|<\varepsilon\frac{|g_j|}{|\Im g_j|}
		\,.
	\]
	Due to our assumption $(iii)$, the right-hand side of this estimate can be made arbitrarily small for large indices $j$.
\end{proof}

\noindent
Now we are ready to settle the case that $M_{nn}$ does not dominate.

\begin{lemma}\label{lemma mu-zero-tech}
	Let $x\in\mathbb R$, and assume that
	\begin{enumerate}[$(i)$]
		\item The symmetric derivative $\omega(x)$ exists.
		\item $\frac{d\rho}{d\lambda}(x)=\infty$.
		\item $\omega_{nn}(x)=0$.
		\item $\frac{d\mu}{d\rho}(x)=0$.
		\item There exists no $k\in\{1,\dots,n\}$ with $\lim\limits_{\varepsilon\downarrow 0}
			\frac{|\Re M_{kk}(x+i\varepsilon)|}{\Im\tr M(x+i\varepsilon)}=\infty$.
	\end{enumerate}
	Then the rank of the matrix $\omega(x)$ is equal to one.
\end{lemma}
\begin{proof}
	Since $w_{nn}(x)=0$, there exists an index $k\in\{1,\ldots ,n-1\}$ such that $\omega_{kk}(x)>0$. From this, the condition $(ii)$, and
	Theorem~\ref{Kac lemma}, it follows that
	\begin{equation}\label{A1}
		\lim_{\varepsilon\downarrow 0}\frac{\Im M_{kk}(x+i\varepsilon)}{\Im\tr M(x+i\varepsilon)}=\omega_{kk}(x)>0.
	\end{equation}
	Throughout the proof we fix an index $k$ with this property.
	
	\hspace*{0pt}\\[-2mm]\textit{Step 1:}
	In this step we deduce that
	\begin{equation}\label{K27}
		\lim_{\varepsilon\downarrow 0}\frac{m(x+i\varepsilon)}{m_k(x+i\varepsilon)}=0
		\,.
	\end{equation}
	By Theorem~\ref{Kac lemma}, the present hypotheses $(iv)$ and $(ii)$ imply that
	\[
		\lim_{\varepsilon\downarrow 0}\frac{\Im m(x+i\varepsilon)}{\Im\tr M(x+i\varepsilon)}=0
		\,.
	\]
	Therefore, using \eqref{A1}, one has $\lim_{\varepsilon\downarrow 0}\frac{\Im m(x+i\varepsilon)}{\Im M_{kk}(x+i\varepsilon)}=0$.
	We compute
	\begin{align}
		\Im M_{kk}= & \Im\frac{m_k(m-m_k)}m=\frac{\Im[m_k(m-m_k)\overline m]}{|m|^2}=
		\nonumber \\
		= & \frac{\Im[m_k(m-m_k)(\overline m-\overline{m_k}+\overline{m_k})]}{|m|^2}=
		\nonumber \\
		= & \Big|\frac{m-m_k}m\Big|^2\Im m_k+\Big|\frac{m_k}m\Big|^2\Im(m-m_k)
		\,.\label{K44}
	\end{align}
	From this we have
	\[
		\frac{\Im m}{\Im M_{kk}}=
		\frac1{\left|\frac{m_k}m\right|^2\frac{\Im(m-m_k)}{\Im m}+\left|\frac{m-m_k}m\right|^2\frac{\Im m_k}{\Im m}}.
	\]
	Using the estimate
	\begin{multline*}
		\left|\frac{m_k}m\right|^2\frac{\Im(m-m_k)}{\Im m}+\left|\frac{m-m_k}m\right|^2\frac{\Im m_k}{\Im m}\leq
		\left|\frac{m_k}m\right|^2+\left|\frac{m-m_k}m\right|^2\leq
		\\
		\leq 2\left|\frac{m_k}m\right|^2+2\left|\frac{m_k}m\right|+1
	\end{multline*}
	we see that the limit relation
	$\lim_{\varepsilon\downarrow 0}\frac{\Im m(x+i\varepsilon)}{\Im M_{kk}(x+i\varepsilon)}=0$ implies \eqref{K27}.
	
	In addition, further rewriting \eqref{K27} as
	$\lim_{\varepsilon\downarrow 0}\frac{m(x+i\varepsilon)-m_k(x+i\varepsilon)}{m_k(x+i\varepsilon)}=-1$, we get from
	Lemma~\ref{lemma tech 1} that
	\[
		\lim_{\varepsilon\downarrow 0}\frac{\Im m_k(x+i\varepsilon)}{\Re m_k(x+i\varepsilon)}=0,\quad
		\lim_{\varepsilon\downarrow 0}\frac{\Im[m(x+i\varepsilon)-m_k(x+i\varepsilon)]}{\Re[m(x+i\varepsilon)-m_k(x+i\varepsilon)]}=0
		\,.
	\]

	\hspace*{0pt}\\[-2mm]\textit{Step 2:}
	In this step we show that for each $l\in\{1,\dots,n-1\}$
	\[
		a_l(x):=\lim_{\varepsilon\downarrow 0}\frac{m_l(x+i\varepsilon)}{m_k(x+i\varepsilon)}\quad\text{exists in }\bb R
		\,.
	\]
	Let $l\in\{1,\ldots,n-1\}$ be given. By $(i)$, $(ii)$, and Theorem~\ref{Kac lemma}, we have
	$\lim_{\varepsilon\downarrow 0}\frac{\Im M_{ll}(x+i\varepsilon)}{\Im\tr M(x+i\varepsilon)}=\omega_{ll}(x)$. Using \eqref{A1}, thus
	$\lim_{\varepsilon\downarrow 0}\frac{\Im M_{ll}(x+i\varepsilon)}{\Im M_{kk}(x+i\varepsilon)}=\frac{\omega_{ll}(x)}{\omega_{kk}(x)}$.
	From the computation \eqref{K44}, it follows that
	\[
		\frac{\Im M_{ll}(z)}{\Im M_{kk}(z)}
		=
		\frac{\big|\frac{m_l(z)}{m_k(z)}\big|^2\frac{\Im(m(z)-m_l(z))}{\Im m(z)}+\big|\frac{m_l(z)}{m_k(z)}-\frac{m(z)}{m_k(z)}\big|^2
		\frac{\Im m_l(z)}{\Im m(z)}}
		{\frac{\Im(m(z)-m_k(z))}{\Im m(z)}+\big|1-\frac{m(z)}{m_k(z)}\big|^2\frac{\Im m_k(z)}{\Im m(z)}}
		\,.
	\]
	Due to \eqref{K27}, the denominator tends to $1$ when $z=x+i\varepsilon$ and $\varepsilon\downarrow 0$. Therefore, the numerator has the limit
	$\frac{\omega_{ll}(x)}{\omega_{kk}(x)}$, i.e.\
	\begin{multline}\label{star}
		\lim_{\varepsilon\downarrow 0}\bigg[
		 \Big|\frac{m_l(x\!+\!i\varepsilon)}{m_k(x\!+\!i\varepsilon)}\Big|^2+\bigg(\Big|\frac{m_l(x\!+\!i\varepsilon)}{m_k(x\!+\!i\varepsilon)}-
		\frac{m(x\!+\!i\varepsilon)}{m_k(x\!+\!i\varepsilon)}\Big|^2-
		\\
		 -\Big|\frac{m_l(x\!+\!i\varepsilon)}{m_k(x\!+\!i\varepsilon)}\Big|^2\bigg)
		\frac{\Im m_l(x\!+\!i\varepsilon)}{\Im m(x\!+\!i\varepsilon)}\bigg]
		=\frac{\omega_{ll}(x)}{\omega_{kk}(x)}
		\,.
	\end{multline}
	Let us show that $\lim_{\varepsilon\downarrow 0}\big|\frac{m_l(x+i\varepsilon)}{m_k(x+i\varepsilon)}\big|^2=
	\frac{\omega_{ll}(x)}{\omega_{kk}(x)}$. First, if $\frac{m_l(x+i\varepsilon)}{m_k(x+i\varepsilon)}$ were not bounded as
	$\varepsilon\downarrow 0$, then there would exists a sequence $\{\varepsilon_j\}_{j\in\bb N}$ with $\varepsilon_j\downarrow0$, such that $\lim_{j\to\infty}\big|\frac{m_l(x+i\varepsilon_j)}{m_k(x+i\varepsilon_j)}\big|=\infty$.
	We have
	\[
    \Big|\frac{m_l}{m_k}-\frac{m}{m_k}\Big|^2-\Big|\frac{m_l}{m_k}\Big|^2=
		\Big|\frac{m_l}{m_k}\Big|^2\cdot\bigg(\bigg|1-\frac{\frac m{m_k}}{\frac{m_l}{m_k}}\bigg|-1\bigg)
		\,,
	\]
	and it follows that
	\begin{multline}\label{estimate} \Big|\frac{m_l(x\!+\!i\varepsilon_j)}{m_k(x\!+\!i\varepsilon_j)}-\frac{m(x\!+\!i\varepsilon_j)}{m_k(x\!+\!i\varepsilon_j)}\Big|^2-
		 \Big|\frac{m_l(x\!+\!i\varepsilon_j)}{m_k(x\!+\!i\varepsilon_j)}\Big|^2=
		\\
=o\bigg(\Big|\frac{m_l(x\!+\!i\varepsilon_j)}{m_k(x\!+\!i\varepsilon_j)}\Big|^2\bigg)
		\quad\text{as }j\to\infty
		\,.
	\end{multline}
	Since $0<\frac{\Im m_l(x+i\varepsilon)}{\Im m(x+i\varepsilon)}\leq 1$, it follows that the expression on the left side of \eqref{star}
	would also be unbounded, a contradiction.
	This means that $\frac{m_l(x+i\varepsilon)}{m_k(x+i\varepsilon)}$ remains bounded when $\varepsilon\downarrow 0$.
	We have
	\begin{multline*}
\left|\Big|\frac{m_l}{m_k}-\frac{m}{m_k}\Big|^2-\Big|\frac{m_l}{m_k}\Big|^2\right|=
\left|\Big|\frac{m_l}{m_k}-\frac{m}{m_k}\Big|-\Big|\frac{m_l}{m_k}\Big|\right|
\bigg(\Big|\frac{m_l}{m_k}-\frac{m}{m_k}\Big|+\Big|\frac{m_l}{m_k}\Big|\bigg)\leq
		\\
		\leq \Big|\frac{m}{m_k}\Big|\bigg(\Big|\frac{m_l}{m_k}-\frac{m}{m_k}\Big|+\Big|\frac{m_l}{m_k}\Big|\bigg)
		\,,
	\end{multline*}
	and it follows that
	\[
		 \Big|\frac{m_l(x\!+\!i\varepsilon)}{m_k(x\!+\!i\varepsilon)}-\frac{m(x\!+\!i\varepsilon)}{m_k(x\!+\!i\varepsilon)}\Big|^2-
		 \Big|\frac{m_l(x\!+\!i\varepsilon)}{m_k(x\!+\!i\varepsilon)}\Big|^2=o(1)
		\quad\text{as }\varepsilon\downarrow 0
		\,.
	\]
	Now we get from \eqref{star} that
	\[
		\lim_{\varepsilon\downarrow 0}\left|\frac{m_l(x+i\varepsilon)}{m_k(x+i\varepsilon)}\right|^2=\frac{\omega_{ll}(x)}{\omega_{kk}(x)}
		\,.
	\]
	Next, rewrite
	\[
		\Im\Big(\frac{m_l}{m_k}\Big)=\frac{\Im m_l}{|m_k|}\cdot\frac{\Re m_k}{|m_k|}-\frac{\Re m_l}{|m_k|}\cdot\frac{\Im m_k}{|m_k|}.
	\]
	Since
	\[
		\frac{|\Re m_k|}{|m_k|}\le1,\quad \frac{|\Re m_l|}{|m_k|}\leq\Big|\frac{m_l}{m_k}\Big|=O(1)
		\,,
	\]
	\[
		\frac{\Im m_l}{|m_k|},\frac{\Im m_k}{|m_k|}\leq\frac{\Im m}{|m_k|}\leq\Big|\frac{m}{m_k}\Big|=o(1)
		\,,
	\]
	we have $\lim_{\varepsilon\downarrow 0}\Im\big(\frac{m_l(x+i\varepsilon)}{m_k(x+i\varepsilon)}\big)=0$. Therefore
	\[
		\text{either}\quad\lim_{\varepsilon\downarrow 0}\frac{m_l(x+i\varepsilon)}{m_k(x+i\varepsilon)}=
		\sqrt{\frac{\omega_{ll}(x)}{\omega_{kk}(x)}}
		\quad\text{or}\quad\lim_{\varepsilon\downarrow 0}\frac{m_l(x+i\varepsilon)}{m_k(x+i\varepsilon)}=
		-\sqrt{\frac{\omega_{ll}(x)}{\omega_{kk}(x)}}
		\,.
	\]
	In both cases the limit $\lim_{\varepsilon\downarrow 0}\frac{m_l(x+i\varepsilon)}{m_k(x+i\varepsilon)}$ exists and is real.
	
	\hspace*{0pt}\\[-2mm]\textit{Step 3; Computing rank:}
	Let $l,p\in\{1,\ldots ,n-1\}$. If $l\neq p$, then
	\begin{equation}\label{M-lp}
		\lim_{\varepsilon\downarrow 0}\frac{M_{lp}(x+i\varepsilon)}{M_{kk}(x+i\varepsilon)}=
		\lim_{\varepsilon\downarrow 0}\frac{\frac{m_l(x+i\varepsilon)}{m_k(x+i\varepsilon)}\cdot
		 \frac{m_p(x+i\varepsilon)}{m_k(x+i\varepsilon)}}{1-\frac{m(x+i\varepsilon)}{m_k(x+i\varepsilon)}}=a_l(x)a_p(x)
		\,.
	\end{equation}
	If $l=p$, then
	\begin{equation}\label{M-ll}
		\lim_{\varepsilon\downarrow 0}\frac{M_{ll}(x+i\varepsilon)}{M_{kk}(x+i\varepsilon)}=
		\lim_{\varepsilon\downarrow 0}\frac{\frac{m_l(x+i\varepsilon)}{m_k(x+i\varepsilon)}\cdot
		 \big(\frac{m_l(x+i\varepsilon)}{m_k(x+i\varepsilon)}-\frac{m(x+i\varepsilon)}{m_k(x+i\varepsilon)}\big)}{
		1-\frac{m(x+i\varepsilon)}{m_k(x+i\varepsilon)}}=a_l^2(x)
		\,.
	\end{equation}
	From $(v)$ it follows that there exists a sequence $\{\varepsilon_j\}_{j\in\bb N}$ with $\varepsilon_j\downarrow 0$, such that
	$\lim_{j\to\infty}\frac{\Re M_{kk}(x+i\varepsilon_j)}{\Im M_{kk}(x+i\varepsilon_j)}=O(1)$ as $j\to\infty$.
	Applying Lemma~\ref{lemma tech 2}, we get from \eqref{M-lp} and \eqref{M-ll} that
	\[
		\frac{\Im M_{lp}(x+i\varepsilon_j)}{\Im M_{kk}(x+i\varepsilon_j)}=a_l(x)a_p(x),\quad l,p=1,\ldots,n-1
		\,.
	\]
	Since $\omega(x)$ is positive semidefinite, $(iii)$ implies that
	\[
		\omega_{ij}(x)=0,\quad i=n\text{ or }j=n
		\,.
	\]
	Altogether,
	\[
		\omega(x)=\frac1{\sum\limits_{l=1}^{n-1}a_l^2(x)}
		\begin{pmatrix}
			a_1^2(x)&a_2(x)a_1(x)&\cdots&a_{n-1}(x)a_1(x)&0\\
			a_1(x)a_2(x)&a_2^2(x)&\cdots&a_{n-1}(x)a_2(x)&0\\
			\vdots&\vdots&\ddots&\vdots&\vdots\\
			a_1(x)a_{n-1}(x)&a_2(x)a_{n-1}(x)&\cdots&a_{n-1}^2(x)&0\\
			0&0&\cdots&0&0
		\end{pmatrix}
		\ .
	\]
	The rank of this matrix obviously cannot exceed $1$. However, $a_k(x)=1$, and hence it is nonzero.
\end{proof}

\noindent
Having available Lemma~\ref{lemma d-nn>0} and Lemma~\ref{lemma mu-zero-tech}, it is not difficult to prove Proposition~\ref{lemma mu-zero}.

\begin{proof}[Proof of Proposition~\ref{lemma mu-zero}]
	Let a $\mu$-zero set $X\subseteq\bb R$ be given.
	It is enough to show that the conditions $(i)$, $(ii)$ appearing in Lemma~\ref{lemma d-nn>0} and Lemma~\ref{lemma mu-zero-tech},
	and the conditions $(iv)$, $(v)$ appearing in Lemma~\ref{lemma mu-zero-tech} are satisfied $\rho$-a.e.\ on $X$.

	The fact that the symmetric derivative $\omega(x)$ exists $\rho$-a.e., has been noted in Remark~\ref{K22}.
	Denote by $\mu_{ac}$ and $\rho_{ac}$ the absolutely continuous parts of $\mu$ and $\rho$ with respect to the Lebesgue measure.
	By Theorem~\ref{K2}, $\rho_{ac}\sim\mu_{ac}\ll\mu$. Thus the set $X$ is $\rho_{ac}$-zero.
	Corollary~\ref{cor Vallee-Poussin measures}, $(iv)$, says that for $\rho_s$-a.a. points $x\in\bb R$ one
	has $\frac{d\rho}{d\lambda}(x)=\infty$. Since $X$ is $\rho_{ac}$-zero, one has $\frac{d\rho}{d\lambda}(x)=\infty$ not only for $\rho_s$-a.a.,
	but even for $\rho$-a.a. $x\in X$.

	Corollary~\ref{cor Vallee-Poussin sets}, ($ii$), shows that for $\rho$-a.a. points $x\in X$ we have $\frac{d\mu}{d\rho}(x)=0$.
	Corollary~\ref{lemma Poltoratski} applied with the measure $\rho$ and the measures that correspond to the Herglotz functions
	$M_{ll}$ gives
	\[
	    \frac{|\Re M_{ll}(x+i\varepsilon)|}{\Im \tr  M(x+i\varepsilon)}\nrightarrow\infty\text{ as }\varepsilon\downarrow 0,\quad
	    l=1,\dots,n-1
	    \,,
	\]
	for $\rho$-a.a.\ $x\in\mathbb R$.
\end{proof}

%
\begin{flushleft}
	\textbf{Stage 2: $\mu$-absolutely continuous part.}
\end{flushleft}
\vspace*{-2mm}
Consider the sets
\begin{align*}
    X_{reg}:=\bigg\{&x\in\bb R:\
        \frac{d\mu}{d\lambda}(x)=\infty,\ \frac{d\rho}{d\mu}(x)\in[0,\infty),
        \\&
		\frac{|\Re m(x+i\varepsilon)|}{\Im m(x+i\varepsilon)}\nrightarrow\infty\text{ as }\varepsilon\downarrow 0,
        \\&
		\forall\,l=1,\ldots,n:\ \frac{d\mu_l}{d\mu}(x)\text{ exists, and }
		 \lim_{\varepsilon\downarrow0}\frac{m_l(x+i\varepsilon)}{m(x+i\varepsilon)}=\frac{d\mu_l}{d\mu}(x)
		\bigg\}.
\end{align*}
\begin{align*}
	& X_{reg}^{1}:=\Big\{x\in X_{reg}:\quad \exists\,l\in\{1,\dots,n\}:\,\frac{d\mu_l}{d\mu}(x)=1\Big\}
\,,\\
	& X_{reg}^{>1}:=\Big\{x\in X_{reg}:\quad \nexists\,l\in\{1,\dots,n\}:\,\frac{d\mu_l}{d\mu}(x)=1\Big\}
		\,.
\end{align*}
Note that, since $\sum_{l=1}^n\frac{d\mu_l}{d\mu}(x)=1$ and $\frac{d\mu_l}{d\mu}(x)\in[0,1]$, one of the following
alternatives takes place:
\begin{itemize}
	\item[$(1)$] There exists one index $k_0$ with $\frac{d\mu_{k_0}}{d\mu}(x)=1$, and for all other indices $k\neq k_0$ we have
		$\frac{d\mu_k}{d\mu}(x)=0$.
	\item[$(${\footnotesize\raisebox{0.8pt}{$>$}}$1)$] For all indices $k$ we have $\frac{d\mu_k}{d\mu}(x)<1$,
		and there exist at least two indices $k$ with $\frac{d\mu_k}{d\mu}(x)>0$.
\end{itemize}

\noindent
In this part we show the following statement.

\begin{proposition}\label{prop s-reg}
	The following hold:
	\begin{enumerate}[$(i)$]
		\item The set $X_{reg}^1$ is $\rho$-zero.
		\item For $\rho$-a.a.\ points $x\in X_{\reg}$ the symmetric derivative $\omega(x)$ exists and
			\[
				\rank\omega(x)=r(x)-1
				\,.
			\]
	\end{enumerate}
\end{proposition}

\noindent
First, an elementary fact which we use to compute rank.

\begin{lemma}\label{lemma rank}
	Let $n\in\bb N$, let $b_1,b_2,\dots,b_n,d\in\bb R\setminus\{0\}$, and consider the matrix
	\[
		M_d:=\left(
		\begin{array}{cccc}
			b_1(d-b_1) & -b_2b_1 & \cdots & -b_nb_1\\
			-b_1b_2 & b_2(d-b_2) & \cdots & -b_nb_2\\
			\vdots & \vdots & \ddots & \vdots \\
			-b_1b_n & -b_2b_n & \cdots & b_n(d-b_n)\\
		\end{array}
		\right).
	\]
	If $d=\sum\limits_{l=1}^nb_l$, then $\rank M_d=n-1$. Otherwise $\rank M_d=n$.
\end{lemma}
\begin{proof}
	We have:
	\begin{align*}
	\rank M_d & =\rank
		\left(
		\arraycolsep1mm
		\begin{array}{cccc}
			d-b_1 & -b_1 & \cdots & -b_1\\
			-b_2 & d-b_2 & \cdots & -b_2\\
			\vdots & \vdots & \ddots & \vdots\\
			-b_n & -b_n & \cdots & d-b_n\\
		\end{array}
		\right)=
	\\[2mm]
	& =\rank
		\left(
		\arraycolsep1mm
		\begin{array}{cccc}
			d & 0 & \cdots & -b_1\\
			0 & d & \cdots & -b_2\\
			\vdots & \vdots & \ddots & \vdots\\
			-d & -d & \cdots & d-b_n\\
		\end{array}
		\right)
		=\rank
		\left(
		\arraycolsep1mm
		\begin{array}{cccc}
			1 & 0 & \cdots & -b_1\\
			0 & 1 & \cdots & -b_2\\
			\vdots & \vdots & \ddots & \vdots\\
			-1 & -1 & \cdots & d-b_n\\
		\end{array}
		\right)=
		\\[2mm]
	&=\rank
		\left(
		\arraycolsep2mm
		\begin{array}{cccc}
			1 & 0 & \cdots & -b_1\\
			0 & 1 & \cdots & -b_2\\
			\vdots & \vdots & \ddots & \vdots\\
			0 & 0 & \cdots & d-\sum\limits_{l=1}^nb_l\\
		\end{array}
		\right)
		=
		\begin{cases}
			n-1 &\hspace*{-3mm},\ d-\sum\limits_{l=1}^nb_l=0,\\[4mm]
			n &\hspace*{-3mm},\ d-\sum\limits_{l=1}^nb_l\neq0\,.
		\end{cases}
	\end{align*}
\end{proof}

\noindent
The next two lemmata contain the essential arguments.

\begin{lemma}\label{K45}
	For each $x\in X_{reg}$ we have
	\[
		 \frac{d\rho}{d\mu}(x)=\sum_{l=1}^{n-1}\frac{d\mu_k}{d\mu}(x)\Big(1-\frac{d\mu_k}{d\mu}(x)\Big)
		\,.
	\]
\end{lemma}
\begin{proof}
	Since $\frac{d\mu}{d\lambda}(x)=\infty$, Theorem~\ref{prop Herglotz} gives $\lim_{\varepsilon\downarrow 0}\Im m(x+i\varepsilon)=\infty$.
	This implies that
	\[
		\lim_{\varepsilon\downarrow 0}\frac{\Im M_{nn}(x+i\varepsilon)}{\Im m(x+i\varepsilon)}=
		\lim_{\varepsilon\downarrow 0}\frac 1{|m(x+i\varepsilon)|^2}=0
		\,.
	\]
	Let $l\in\{1,\dots,n-1\}$. Then we have
	\[
		\lim_{\varepsilon\downarrow 0}\frac{M_{ll}(x+i\varepsilon)}{m(x+i\varepsilon)}=
		\frac{d\mu_l}{d\mu}(x)\Big(1-\frac{d\mu_l}{d\mu}(x)\Big)
		\,.
	\]
	Choose a sequence $\{\varepsilon_j\}_{j\in\bb N}$, $\varepsilon_j\downarrow 0$, such that
	$\frac{|\Re m(x+i\varepsilon_j)|}{\Im m(x+i\varepsilon_j)}$ remains bounded when $j\to\infty$.
	Lemma~\ref{lemma tech 2} yields that also
	\begin{equation}\label{K46}
		\lim_{j\to\infty}\frac{\Im M_{ll}(x+i\varepsilon_j)}{\Im m(x+i\varepsilon_j)}=
		\frac{d\mu_l}{d\mu}(x)\Big(1-\frac{d\mu_l}{d\mu}(x)\Big),\quad l=1,\dots,n
		\,.
	\end{equation}
	Now we use the sequence $\{\varepsilon_j\}_{j\in\bb N}$ to evaluate $\frac{d\rho}{d\mu}(x)$ by means of Theorem~\ref{Kac lemma}.
	This gives
	\begin{equation}\label{K30}
		\frac{d\rho}{d\mu}(x)=\lim_{j\to\infty}\frac{\Im\tr M(x+i\varepsilon_j)}{\Im m(x+i\varepsilon_j)}=
		 \sum_{l=1}^{n-1}\frac{d\mu_l}{d\mu}(x)\Big(1-\frac{d\mu_l}{d\mu}(x)\Big)
		\,.
	\end{equation}
\end{proof}

\begin{corollary}\label{K47}
	Let $x\in X_{reg}$. Then
	\[
		\frac{d\rho}{d\mu}(x)\
		\begin{cases}
			\ =0 &\hspace*{-3mm},\quad x\in X_{reg}^1,\\[2mm]
			\ >0 &\hspace*{-3mm},\quad x\in X_{reg}^{>1}.\\
		\end{cases}
	\]
\end{corollary}
\begin{proof}
	Let $x\in X_{reg}$. Then we have $x\in X_{reg}^1$ if the above alternative $(1)$ takes place, and $x\in X_{reg}^{>1}$ if
	$(${\footnotesize\raisebox{0.8pt}{$>$}}$1)$ takes place. Hence, for $x\in X_{reg}^1$ we have
	 $\sum_{l=1}^{n-1}\frac{d\mu_k}{d\mu}(x)\big(1-\frac{d\mu_k}{d\mu}(x)\big)=0$, and for $x\in X_{reg}^{>1}$ this sum is positive.
\end{proof}

\begin{lemma}\label{lemma s-reg}
	Let $x\in X_{reg}^{>1}$, and assume that
	\begin{enumerate}[$(i)$]
	\item The symmetric derivative $\omega(x)$ exists.
	\item $\frac{d\rho}{d\lambda}(x)=\infty$.
	\end{enumerate}
	Then
	\begin{equation}\label{K31}
		 \rank\omega(x)=\#\Big\{l\in\{1,\dots,n\}:\,\frac{d\mu_l}{d\mu}(x)>0\Big\}-1
		\,.
	\end{equation}
\end{lemma}

\begin{proof}
	By Theorem~\ref{Kac lemma}, the present assumptions $(i)$ and $(ii)$ ensure that
	\begin{equation}\label{K48}
		\omega(x)=\lim_{\varepsilon\downarrow 0}\frac{\Im M(x+i\varepsilon)}{\Im\tr M(x+i\varepsilon)}
		\,.
	\end{equation}
	It is easy to show that the last row and column of the matrix $\omega(x)$ vanishes:
	By Theorem~\ref{prop Herglotz}, our assumption $(ii)$ gives $\lim_{\varepsilon\downarrow 0}\Im\tr M(x+i\varepsilon)=\infty$.
	We already saw in the proof of the last lemma that $\lim_{\varepsilon\downarrow 0}\Im m(x+i\varepsilon)=\infty$, and it follows that $\lim_{\varepsilon\downarrow0}\Im M_{nn}(x+i\varepsilon)=0$ and $\omega_{nn}(x)=0$. Since $\omega(x)$ is positive semidefinite, all entries $\omega_{ij}(x)$ with $i=n$ or $j=n$ must vanish.

	To shorten notation, set $d_l(x):=\frac{d\mu_l}{d\mu}(x)$. Let $\{\varepsilon_j\}_{j\in\bb N}$ be the same sequence as in the
	proof of the previous lemma. Then not only \eqref{K46} holds, but also
	\[
		\lim_{j\to\infty}\frac{\Im M_{lk}(x+i\varepsilon_j)}{\Im m(x+i\varepsilon_j)}=
		-d_l(x)d_k(x),\quad l,k=1,\dots,n-1,l\neq k
		\,.
	\]
	Referring to \eqref{K48} and \eqref{K30}, we obtain
	\begin{multline*}
		\omega(x)=\frac1{\sum\limits_{l=1}^{n-1}d_l(x)(1-d_l(x))}\times
		\\
		\times
		\begin{pmatrix}
			d_1(x)(1-d_1(x))&-d_2(x)d_1(x)&\cdots&-d_{n-1}(x)d_1(x)&0\\
			-d_1(x)d_2(x)&d_2(x)(1-d_2(x))&\cdots&-d_{n-1}(x)d_2(x)&0\\
			\vdots&\vdots&\ddots&\vdots&\vdots\\
			 -d_1(x)d_{n-1}(x)&-d_2(x)d_{n-1}(x)&\cdots&d_{n-1}(x)(1-d_{n-1}(x))&0\\
			0&0&\cdots&0&0
		\end{pmatrix}
        .
	\end{multline*}
	Applying Lemma~\ref{lemma rank} with the matrix obtained from $\omega(x)$ by deleting all rows and columns which contain only
	zeros, gives
	\[
		\rank\omega(x)=
		\begin{cases}
			\#\big\{l:\,1\leq l\leq n\!-\!1,\,d_l(x)>0\big\} &\hspace*{-3mm},\quad \sum\limits_{l=1}^{n-1}d_l(x)\neq 1,\\[3mm]
			\#\big\{l:\,1\leq l\leq n\!-\!1,\,d_l(x)>0\big\}-1 &\hspace*{-3mm},\quad \sum\limits_{l=1}^{n-1}d_l(x)=1.
		\end{cases}
	\]
	The condition $\sum_{l=1}^{n-1}d_l(x)\neq 1$ is equivalent to $d_n(x)\neq 0$, and the formula \eqref{K31} follows.
\end{proof}

\begin{proof}[Proof of Proposition~\ref{prop s-reg}]
	Corollary~\ref{K47} and Corollary~\ref{cor Vallee-Poussin sets}, $(i)$, show that $X_{reg}^1$ is $\rho$-zero. Denote
	\[
		X_{reg}^+:=\Big\{x\in X_{reg}^{>1}:\ \omega(x)\text{ exists},\,\frac{d\rho}{d\lambda}(x)=\infty\Big\}
		\,.
	\]
	Then Lemma~\ref{lemma s-reg} says that
	\begin{equation}\label{K32}
	 \rank\omega(x)=\#\Big\{l\in\{1,\dots,n\}:\,\frac{d\mu_l}{d\mu}(x)>0\Big\}-1,\quad x\in X_{reg}^+
		\,.
	\end{equation}
	We have
	\begin{multline*}
		X_{reg}\setminus X_{reg}^+\subseteq\big\{x\in\bb R:\,\omega(x)\text{ does not exist}\big\}\cup
		\big\{x\in\bb R:\,\frac{d\rho}{d\lambda}(x)\in[0,\infty)\big\}\cup
		\\
		\cup\big\{x\in\bb R:\,\frac{d\rho}{d\lambda}(x)\text{ does not exist}\big\}\cup
		X_{reg}^1
		\,.
	\end{multline*}
	The first set in this union is $\rho$-zero by Remark~\ref{K22}. By Corollary~\ref{cor Vallee-Poussin measures}, $(iv)$, the second set is $\rho_s$-zero. The third set is $\rho$-zero by Theorem~\ref{K5}, and the last by the already proved item $(i)$.
	Together, and due to the fact that the set $X_{reg}$ itself is Lebesgue-zero, we see that $X_{reg}\setminus X_{reg}^+$ is $\rho$-zero.
	
	We have
	\[
		X_{reg}\subseteq\big\{x\in\bb R\setminus\mc E_{\rho,\mu}:\,\frac{d\rho}{d\mu}(x)\in[0,\infty)\big\}\cup\mc E_{\rho,\mu}
		\,.
	\]
	Using Corollary~\ref{cor Vallee-Poussin sets}, $(iii)$, we obtain that the intersection of every $\mu$-zero set with $X_{reg}$ is $\rho$-zero.
	Hence, $D_{\mu}\mu_l(x)=\frac{d\mu_l}{d\mu}(x)$ for $\rho$-a.a. $x\in X_{\reg}$, and \eqref{K32} implies that
	\[
		\rank\omega(x)=r(x)-1,\quad \text{for }\rho\text{-a.a.\ }x\in X_{reg}
		\,.
	\]
\end{proof}

\begin{flushleft}
	\textbf{Stage 3: Finishing the proof of the main theorem.}
\end{flushleft}
\vspace*{-2mm}
Having available Propositions~\ref{lemma mu-zero} and \ref{prop s-reg}, it is not anymore difficult to complete the proof of Theorem~\ref{K38}.

\begin{proof}[Proof of Theorem~\ref{K38}]
Let $\mu=\mu_{ac}+\mu_s$ and $\rho=\rho_{ac}+\rho_s$ be the decompositions of $\mu$ and $\rho$ with respect to the Lebesgue measure.

The set $X_{reg}$ is $\mu_{ac}$-zero, because (as we know from Corollary~\ref{cor Vallee-Poussin measures}, $(iii)$) the set $\big\{x\in\bb R:\frac{d\mu}{d\lambda}(x)=\infty\big\}$ is $\mu_{ac}$-zero. The following properties hold:
\begin{enumerate}[$(1)$]
\item $\frac{d\mu}{d\lambda}(x)=\infty$ $\mu_s$-a.e.\ (due to Corollary~\ref{cor Vallee-Poussin measures}, $(iv)$).
\item $\frac{d\rho}{d\mu}(x)\in[0,\infty)$ and the limit $\frac{d\mu_l}{d\mu}(x)$ exists $\mu$-a.e. (due to Corollary~\ref{cor Vallee-Poussin measures}, $(i)$, and Theorem~\ref{K5}).
\item $\frac{|\Re m(x+i\varepsilon)|}{\Im m(x+i\varepsilon)}\nrightarrow\infty$ as $\varepsilon\downarrow 0$ $\mu$-a.e. (due to Theorem~\ref{theorem Poltoratski-1}).
\item $\lim_{\varepsilon\downarrow 0}\frac{m_l(x+i\varepsilon)}{m(x+i\varepsilon)}=\frac{d\mu_l}{d\mu}(x)$ $\mu_s$-a.e. (due to Theorem~\ref{theorem Poltoratski-2}, (ii)).
	\end{enumerate}
Hence, the set $X_{reg}$ is $\mu_s$-full. Thus we can choose a Borel set $X\subseteq X_{reg}$ with $\mu_s(\bb R\setminus X)=0$ and $\mu_{ac}(X)=0$ (since $X_{reg}$ is Lebesgue-zero).

Let $\mc E:=\bigcup_{l=1}^n\mc E_{\mu_l,\mu}$ where $\mc E_{\mu_l,\mu}$ are exceptional sets as in Theorem~\ref{K5}. By passing	from $X$ to $X\setminus\mc E$, we may assume that the functions $\frac{d\mu_l}{d\mu}\big|_X$ are Borel measurable. Set
	\[
		X_{s,ac}:=X\cap X_{reg}^{>1}
		\,.
	\]
It is a Borel set. The set $X_{>1}=r^{-1}(\{2,\dots,n\})$ in the formulation of the theorem is determined only up to a $\mu$-zero (and hence $\rho_{s,ac}$-zero) set. Since the symmetric derivative $\frac{d\mu_l}{d\mu}$ coincides $\mu$-a.e.\ with the	 Radon-Nikodym derivative $D_{\mu}\mu_l$ (i.e., is a representative of the class of equivalent functions), we may use
	\[
		X_{>1}:=\Big\{x\in\bb R\setminus\mc E:\ \text{ for at least two indices }\frac{d\mu_l}{d\mu}(x)>0\Big\}
		\,.
	\]
Then $X_{>1}\cap X=X_{s,ac}$.

Now we can determine the Lebesgue decomposition of $\rho_s$ with respect to $\mu$. Choose a Borel set $Y$ with $\rho_s(Y^c)=0$ and $\lambda(Y)=0$, so that $\rho_s=\mathds1_Y\cdot\rho$. Our candidates for the Lebesgue decomposition are:
	\[
	\rho_{s,ac}:=\mathds{1}_{X_{s,ac}}\cdot\rho_s,\quad
    \rho_{s,s}:=\mathds{1}_{X^c}\cdot\rho_s=\mathds1_{X^c\cap Y}\cdot\rho
		\,.
	\]
Since $X\setminus X_{s,ac}\subseteq X _{reg}^1$, by Proposition \ref{prop s-reg}, (i), we have $\rho(X\setminus X_{s,ac})=0$, so $\rho_s=\rho_{s,ac}+\rho_{s,s}$. Furthermore, for every $x\in X_{s,ac}$ we have $\frac{d\rho}{d\mu}(x)\in(0,\infty)$, and hence by Corollary \ref{cor Vallee-Poussin sets}, (iii),
\begin{equation}\label{K49}
    \rho_{s,ac}\sim\mathds1_{X_s,ac}\cdot\mu=\mathds1_{X_{>1}}\cdot\mathds1_X\cdot\mu=\mathds1_{X_{>1}}\cdot\mu_s.
\end{equation}
Finally, from $\mu_s(X^c)=0$ it follows that $\rho_{s,s}\perp\mu_s$ and therefore $\rho_{s,s}\perp\mu$. Thus, indeed, $\rho_{s,ac}$ and $\rho_{s,s}$ are the absolutely continuous and singular part of $\rho_s$ with respect to $\mu$.

From the fact that $E\sim\rho$ and \eqref{K49} it follows that Item {\rm(I)} of Theorem~\ref{K38} holds. Since $X\subseteq X_{reg}$ and $\rho_{s,ac}(X^c)=0$,	 Proposition~\ref{prop s-reg} implies that
    \[
		\rank\omega(x)=r(x)-1\text{ for $\rho_{s,ac}$-a.a.\ points $x\in\bb R$}
		\,,
	\]
which gives item {\rm(II)}. Due to the fact that $\mu(X^c\cap Y)=0$, Proposition~\ref{lemma mu-zero} implies
	\[
		\rank\omega(x)=1\text{ for $\rho_{s,s}$-a.a.\ points $x\in\bb R$}
		\,,
	\]
and this is item {\rm(III)}.
\end{proof}

\appendix
\makeatletter
\renewcommand{\@seccntformat}[1]{\@nameuse{the#1}.\quad}
\makeatother
\appendixsection{Some examples}

In this appendix we provide four examples in order to
show that all possibilities for the spectrum which are admitted by Theorem~\ref{K38} indeed may occur.
We realize these examples on the level of Schr\"odinger operators.
Due to the general inverse theorem stated as the second part of Theorem~\ref{K62}, it would be somewhat simpler to realize them on the
level of boundary relations. However, in order to remain in a more intuitive setting, we decided to stick to the Schr\"odinger case.
Also we should say it very clearly that our emphasize in this appendix is on examples and methods rather than on maximal generality.

Let us first recall in some detail how a half-line Schr\"odinger operator can be considered as a boundary relation. This is of course
a (if not ``the'') standard example for boundary relations, see \cite{Gorbachuk-Gorbachuk-1984} and the references therein.
Formulated in our present language it reads as follows.

\begin{remark}\label{K33}
	Let $q$ be a real and locally integrable potential defined on $(0,\infty)$, and assume that $0$ is a regular endpoint and that
	Weyl's limit point case prevails at the endpoint $\infty$. Denote by $T_{max}$ the maximal differential operator generated in $L^2(0,\infty)$
	by the differential expression $-\frac{d^2}{dx^2}+q$. For $\alpha\in\bb R$, denote by $\Gamma_{(\alpha)}$ the relation
	\begin{multline*}
		\Gamma_{(\alpha)}:=\Big\{\Big( (u;T_{max}u);\big(u(0)\cos\alpha+u'(0)\sin\alpha;-u(0)\sin\alpha+u'(0)\cos\alpha\big)\Big):
			\\
		u\in\dom T_{max}\Big\}\subseteq L^2(0,\infty)^2\times\bb C^2
	\end{multline*}
	Then it is easy to see that $\Gamma_{(\alpha)}$ is a boundary relation (in fact, ``boundary function'') for the
	operator $T_{max}$: For $\alpha=0$ see
	\cite[Example~1.3]{Derkach-Hassi-Malamud-deSnoo-2006}. For other values of $\alpha$ note that $\Gamma_{(\alpha)}$ and $\Gamma_{(0)}$
	are related by
	\[
		\Gamma_{(\alpha)}=w_\alpha\circ\Gamma_{(0)}
	\]
	with the $J_{\mathbb C}$-unitary matrix
	\[
		w_\alpha:=
		\begin{pmatrix}
			\cos\alpha & \sin\alpha\\
			-\sin\alpha & \cos\alpha
		\end{pmatrix}
		\,,
	\]
	and apply Theorem~\ref{prop frac-lin}.

	Apparently, the selfadjoint operator
	\[
		A_\alpha:=\ker\big[\pi_1\circ\Gamma_{(\alpha)}\big]
	\]
	is nothing but the selfadjoint restriction of $T_{max}$ given by the boundary condition
	\[
		u(0)\cdot\cos\alpha+u'(0)\cdot\sin\alpha=0
		\,.
	\]
\end{remark}

\noindent
First we consider $E_{s,ac}$ and $E_{ac}$ and give an example that an arbitrary number of overlaps (embedded into absolutely continuous
spectrum or not) can be produced. This is very simple; we elaborate it only for the sake of illustration and completeness.

\begin{example}\label{K74}
	Let $\lambda_1,\lambda_2$ be measures such that ($\lambda$ is the Lebesgue measure)
	\[
		\lambda_1\perp\lambda_2,\qquad \lambda_j\perp\lambda,\ \supp\lambda_j=[0,1],\ \lambda_j(\{0\})=\lambda_j(\{1\})=0,\quad j=1,2
		\,.
	\]
	For $n,m\in\bb Z$, $n<m$, set
	\[
		 \lambda_j^{(n,m)}(\Delta):=\sum_{l=n}^{m-1}\lambda_j\big(\Delta-l\big),\quad \Delta\text{ Borel set}
		\,.
	\]
	Moreover, let $f$ be the function
	\[
		f(x):=
		\begin{cases}
			\frac 2{3\pi}x^{\frac 32} &\hspace*{-3mm},\quad x\geq\frac 92,\\
			0 &\hspace*{-3mm},\quad \text{otherwise.}\\
		\end{cases}
	\]
	We consider the measures $\mu_1,\dots,\mu_4$ defined as
	\[
	\begin{array}{llll}
\mu_1:=&\lambda_1^{(2,3)}+\lambda_1^{(4,5)}&+\lambda_2^{(3,4)}+\lambda_2^{(6,7)}&+f\cdot\lambda,
        \\[1mm]
		\mu_2:=&\lambda_1^{(2,6)}&+\lambda_2^{(6,7)}&+f\cdot\lambda, 
        \\[1mm]
		\mu_3:=&&+\lambda_2^{(0,7)}&+f\cdot\lambda, 
        \\[1mm] \mu_4:=&\lambda_1^{(0,1)}+\lambda_1^{(7,8)}&+\lambda_2^{(3,8)}&+f\cdot\lambda.
	\end{array}
    \]
	Now we appeal to the version of the Gelfand-Levitan theorem which applies to the Dirichlet boundary condition, cf.\
	\cite[\S2.9]{Levitan-1987}. The hypotheses of this result are obviously fulfilled, and we obtain a potentials $q_1,\dots,q_4$
	on the half-line, such that $\mu_j$ is the measure in the integral representation
	of the Titchmarsh-Weyl coefficient constructed from the potential $q_j$ with Dirichlet boundary conditions.

	Let $A_1,\dots,A_4$ be the corresponding non-interacting operators, and let $A$ be their pasting with standard interface conditions.
	Then support sets of the spectral measures of $A_1,\dots,A_4$ can be pictured as follows:
	\begin{center}
\setlength{\unitlength}{0.00065in}
\begingroup\makeatletter\ifx\SetFigFont\undefined%
\gdef\SetFigFont#1#2#3#4#5{%
  \reset@font\fontsize{#1}{#2pt}%
  \fontfamily{#3}\fontseries{#4}\fontshape{#5}%
  \selectfont}%
\fi\endgroup%
{\renewcommand{\dashlinestretch}{30}
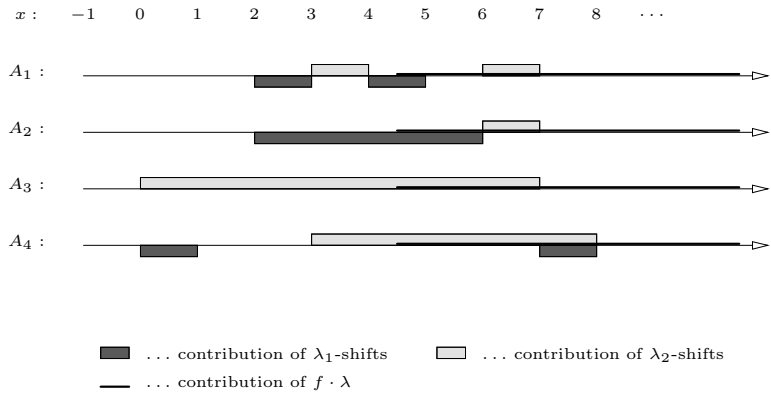
\begin{picture}(5877,3500)(0,-200)

%
%
\texture{8101010 10000000 444444 44000000 11101 11000000 444444 44000000
	101010 10000000 444444 44000000 10101 1000000 444444 44000000
	101010 10000000 444444 44000000 11101 11000000 444444 44000000
	101010 10000000 444444 44000000 10101 1000000 444444 44000000 }
\shade\path(915,1707)(4065,1707)(4065,1617)
	(915,1617)(915,1707)
\path(915,1707)(4065,1707)(4065,1617)
	(915,1617)(915,1707)
\shade\path(2265,2607)(2715,2607)(2715,2517)
	(2265,2517)(2265,2607)
\path(2265,2607)(2715,2607)(2715,2517)
	(2265,2517)(2265,2607)
\shade\path(3615,2607)(4065,2607)(4065,2517)
	(3615,2517)(3615,2607)
\path(3615,2607)(4065,2607)(4065,2517)
	(3615,2517)(3615,2607)
\shade\path(3615,2157)(4065,2157)(4065,2067)
	(3615,2067)(3615,2157)
\path(3615,2157)(4065,2157)(4065,2067)
	(3615,2067)(3615,2157)
\shade\path(2265,1257)(4515,1257)(4515,1167)
	(2265,1167)(2265,1257)
\path(2265,1257)(4515,1257)(4515,1167)
	(2265,1167)(2265,1257)

\texture{aa555555 55bbbbbb bb555555 55fefefe fe555555 55bbbbbb bb555555 55eeefee
	ef555555 55bbbbbb bb555555 55fefefe fe555555 55bbbbbb bb555555 55efefef
	ef555555 55bbbbbb bb555555 55fefefe fe555555 55bbbbbb bb555555 55eeefee
	ef555555 55bbbbbb bb555555 55fefefe fe555555 55bbbbbb bb555555 55efefef }
\shade\path(1815,2427)(2265,2427)(2265,2517)
	(1815,2517)(1815,2427)
\path(1815,2427)(2265,2427)(2265,2517)
	(1815,2517)(1815,2427)
\shade\path(2715,2427)(3165,2427)(3165,2517)
	(2715,2517)(2715,2427)
\path(2715,2427)(3165,2427)(3165,2517)
	(2715,2517)(2715,2427)
\shade\path(1815,1977)(3615,1977)(3615,2067)
	(1815,2067)(1815,1977)
\path(1815,1977)(3615,1977)(3615,2067)
	(1815,2067)(1815,1977)
\shade\path(915,1077)(1365,1077)(1365,1167)
	(915,1167)(915,1077)
\path(915,1077)(1365,1077)(1365,1167)
	(915,1167)(915,1077)
\shade\path(4065,1077)(4515,1077)(4515,1167)
	(4065,1167)(4065,1077)
\path(4065,1077)(4515,1077)(4515,1167)
	(4065,1167)(4065,1077)

%
%
\path(465,2517)(5865,2517)
\whiten\path(5745.000,2487.000)(5865.000,2517.000)(5745.000,2547.000)(5745.000,2487.000)
\path(465,2067)(5865,2067)
\whiten\path(5745.000,2037.000)(5865.000,2067.000)(5745.000,2097.000)(5745.000,2037.000)
\path(465,1617)(5865,1617)
\whiten\path(5745.000,1587.000)(5865.000,1617.000)(5745.000,1647.000)(5745.000,1587.000)
\path(465,1167)(5865,1167)
\whiten\path(5745.000,1137.000)(5865.000,1167.000)(5745.000,1197.000)(5745.000,1137.000)

%
%
\thicklines
\path(2940,2530)(5640,2530)
\path(2940,2080)(5640,2080)
\path(2940,1630)(5640,1630)
\path(2940,1180)(5640,1180)
\thinlines

%
%
\thicklines
\path(600,42)(825,42)
\thinlines
\texture{aa555555 55bbbbbb bb555555 55fefefe fe555555 55bbbbbb bb555555 55eeefee
	ef555555 55bbbbbb bb555555 55fefefe fe555555 55bbbbbb bb555555 55efefef
	ef555555 55bbbbbb bb555555 55fefefe fe555555 55bbbbbb bb555555 55eeefee
	ef555555 55bbbbbb bb555555 55fefefe fe555555 55bbbbbb bb555555 55efefef }
\shade\path(600,267)(825,267)(825,357)
	(600,357)(600,267)
\path(600,267)(825,267)(825,357)
	(600,357)(600,267)
\texture{8101010 10000000 444444 44000000 11101 11000000 444444 44000000
	101010 10000000 444444 44000000 10101 1000000 444444 44000000
	101010 10000000 444444 44000000 11101 11000000 444444 44000000
	101010 10000000 444444 44000000 10101 1000000 444444 44000000 }
\shade\path(3255,357)(3480,357)(3480,267)
	(3255,267)(3255,357)
\path(3255,357)(3480,357)(3480,267)
	(3255,267)(3255,357)

%
%
\put(915,2967){\makebox(0,0)[b]{\smash{{\SetFigFont{6}{7.2}{\rmdefault}{\mddefault}{\updefault}$0$}}}}
\put(1365,2967){\makebox(0,0)[b]{\smash{{\SetFigFont{6}{7.2}{\rmdefault}{\mddefault}{\updefault}$1$}}}}
\put(1815,2967){\makebox(0,0)[b]{\smash{{\SetFigFont{6}{7.2}{\rmdefault}{\mddefault}{\updefault}$2$}}}}
\put(2265,2967){\makebox(0,0)[b]{\smash{{\SetFigFont{6}{7.2}{\rmdefault}{\mddefault}{\updefault}$3$}}}}
\put(2715,2967){\makebox(0,0)[b]{\smash{{\SetFigFont{6}{7.2}{\rmdefault}{\mddefault}{\updefault}$4$}}}}
\put(3165,2967){\makebox(0,0)[b]{\smash{{\SetFigFont{6}{7.2}{\rmdefault}{\mddefault}{\updefault}$5$}}}}
\put(3615,2967){\makebox(0,0)[b]{\smash{{\SetFigFont{6}{7.2}{\rmdefault}{\mddefault}{\updefault}$6$}}}}
\put(4065,2967){\makebox(0,0)[b]{\smash{{\SetFigFont{6}{7.2}{\rmdefault}{\mddefault}{\updefault}$7$}}}}
\put(4515,2967){\makebox(0,0)[b]{\smash{{\SetFigFont{6}{7.2}{\rmdefault}{\mddefault}{\updefault}$8$}}}}
\put(4965,2967){\makebox(0,0)[b]{\smash{{\SetFigFont{6}{7.2}{\rmdefault}{\mddefault}{\updefault}$\cdots$}}}}
\put(465,2967){\makebox(0,0)[b]{\smash{{\SetFigFont{6}{7.2}{\rmdefault}{\mddefault}{\updefault}$-1$}}}}
\put(15,2967){\makebox(0,0)[b]{\smash{{\SetFigFont{6}{7.2}{\rmdefault}{\mddefault}{\updefault}$x:$}}}}
\put(15,2517){\makebox(0,0)[b]{\smash{{\SetFigFont{6}{7.2}{\rmdefault}{\mddefault}{\updefault}$A_1:$}}}}
\put(15,2067){\makebox(0,0)[b]{\smash{{\SetFigFont{6}{7.2}{\rmdefault}{\mddefault}{\updefault}$A_2:$}}}}
\put(15,1617){\makebox(0,0)[b]{\smash{{\SetFigFont{6}{7.2}{\rmdefault}{\mddefault}{\updefault}$A_3:$}}}}
\put(15,1167){\makebox(0,0)[b]{\smash{{\SetFigFont{6}{7.2}{\rmdefault}{\mddefault}{\updefault}$A_4:$}}}}
\put(960,267){\makebox(0,0)[lb]{\smash{{\SetFigFont{6}{7.2}{\rmdefault}{\mddefault}{\updefault}$\ldots$ contribution of $\lambda_1$-shifts}}}}
\put(3615,267){\makebox(0,0)[lb]{\smash{{\SetFigFont{6}{7.2}{\rmdefault}{\mddefault}{\updefault}$\ldots$ contribution of $\lambda_2$-shifts}}}}
\put(960,42){\makebox(0,0)[lb]{\smash{{\SetFigFont{6}{7.2}{\rmdefault}{\mddefault}{\updefault}$\ldots$ contribution of $f\cdot\lambda$}}}}

\end{picture}
}
	\end{center}
	Support sets of $E_{s,ac}$ and $E_{ac}$ including multiplicities are:
	\\
	\begin{center}
\setlength{\unitlength}{0.00065in}
\begingroup\makeatletter\ifx\SetFigFont\undefined%
\gdef\SetFigFont#1#2#3#4#5{%
  \reset@font\fontsize{#1}{#2pt}%
  \fontfamily{#3}\fontseries{#4}\fontshape{#5}%
  \selectfont}%
\fi\endgroup%
{\renewcommand{\dashlinestretch}{30}
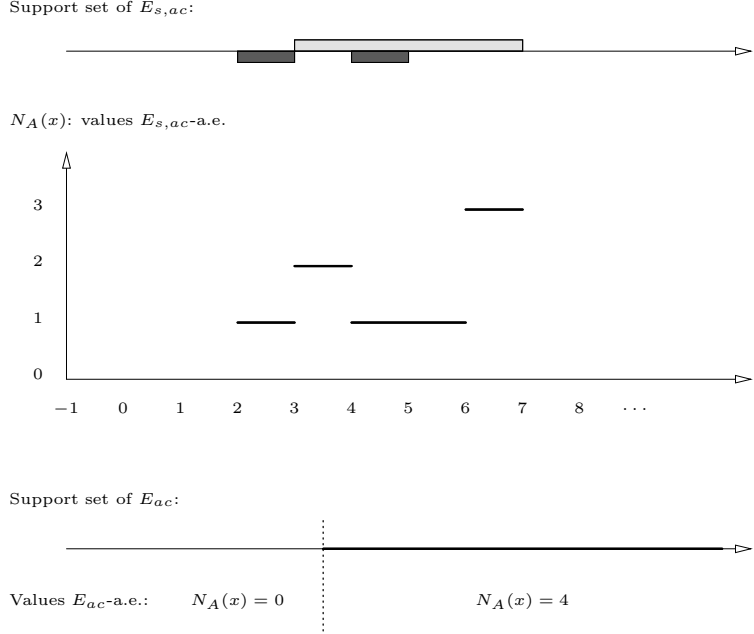
\begin{picture}(5877,5064)(0,-200)
\path(465,2037)(5865,2037)
\whiten\path(5745.000,2007.000)(5865.000,2037.000)(5745.000,2067.000)(5745.000,2007.000)
\path(465,2037)(465,3837)
\whiten\path(495.000,3717.000)(465.000,3837.000)(435.000,3717.000)(495.000,3717.000)
\thicklines
\texture{aa555555 55bbbbbb bb555555 55fefefe fe555555 55bbbbbb bb555555 55eeefee
	ef555555 55bbbbbb bb555555 55fefefe fe555555 55bbbbbb bb555555 55efefef
	ef555555 55bbbbbb bb555555 55fefefe fe555555 55bbbbbb bb555555 55eeefee
	ef555555 55bbbbbb bb555555 55fefefe fe555555 55bbbbbb bb555555 55efefef }
\path(1815,2487)(2265,2487)
\path(1815,2487)(2265,2487)
\path(2265,2937)(2715,2937)
\path(2265,2937)(2715,2937)
\path(2715,2487)(3615,2487)
\path(2715,2487)(3615,2487)
\path(3615,3387)(4065,3387)
\path(3615,3387)(4065,3387)
\thinlines
\texture{8101010 10000000 444444 44000000 11101 11000000 444444 44000000
	101010 10000000 444444 44000000 10101 1000000 444444 44000000
	101010 10000000 444444 44000000 11101 11000000 444444 44000000
	101010 10000000 444444 44000000 10101 1000000 444444 44000000 }
\shade\path(2265,4737)(4065,4737)(4065,4647)
	(2265,4647)(2265,4737)
\path(2265,4737)(4065,4737)(4065,4647)
	(2265,4647)(2265,4737)
\texture{aa555555 55bbbbbb bb555555 55fefefe fe555555 55bbbbbb bb555555 55eeefee
	ef555555 55bbbbbb bb555555 55fefefe fe555555 55bbbbbb bb555555 55efefef
	ef555555 55bbbbbb bb555555 55fefefe fe555555 55bbbbbb bb555555 55eeefee
	ef555555 55bbbbbb bb555555 55fefefe fe555555 55bbbbbb bb555555 55efefef }
\shade\path(1815,4557)(2265,4557)(2265,4647)
	(1815,4647)(1815,4557)
\path(1815,4557)(2265,4557)(2265,4647)
	(1815,4647)(1815,4557)
\shade\path(2715,4557)(3165,4557)(3165,4647)
	(2715,4647)(2715,4557)
\path(2715,4557)(3165,4557)(3165,4647)
	(2715,4647)(2715,4557)
\path(465,4647)(5865,4647)
\whiten\path(5745.000,4617.000)(5865.000,4647.000)(5745.000,4677.000)(5745.000,4617.000)
\path(465,687)(5865,687)
\whiten\path(5745.000,657.000)(5865.000,687.000)(5745.000,717.000)(5745.000,657.000)
\thicklines
\path(2490,687)(5640,687)
\path(2490,687)(5640,687)
\thinlines
\dottedline{45}(2490,912)(2490,12)
\put(15,4062){\makebox(0,0)[lb]{\smash{{\SetFigFont{6}{7.2}{\rmdefault}{\mddefault}{\updefault}$N_A(x)$: values $E_{s,ac}$-a.e.}}}}
\put(15,1047){\makebox(0,0)[lb]{\smash{{\SetFigFont{6}{7.2}{\rmdefault}{\mddefault}{\updefault}Support set of $E_{ac}$:}}}}
\put(15,237){\makebox(0,0)[lb]{\smash{{\SetFigFont{6}{7.2}{\rmdefault}{\mddefault}{\updefault}Values $E_{ac}$-a.e.:}}}}
\put(15,4962){\makebox(0,0)[lb]{\smash{{\SetFigFont{6}{7.2}{\rmdefault}{\mddefault}{\updefault}Support set of $E_{s,ac}$:}}}}
\put(4065,237){\makebox(0,0)[b]{\smash{{\SetFigFont{6}{7.2}{\rmdefault}{\mddefault}{\updefault}$N_A(x)=4$}}}}
\put(1815,237){\makebox(0,0)[b]{\smash{{\SetFigFont{6}{7.2}{\rmdefault}{\mddefault}{\updefault}$N_A(x)=0$}}}}
\put(915,1767){\makebox(0,0)[b]{\smash{{\SetFigFont{6}{7.2}{\rmdefault}{\mddefault}{\updefault}$0$}}}}
\put(1365,1767){\makebox(0,0)[b]{\smash{{\SetFigFont{6}{7.2}{\rmdefault}{\mddefault}{\updefault}$1$}}}}
\put(1815,1767){\makebox(0,0)[b]{\smash{{\SetFigFont{6}{7.2}{\rmdefault}{\mddefault}{\updefault}$2$}}}}
\put(2265,1767){\makebox(0,0)[b]{\smash{{\SetFigFont{6}{7.2}{\rmdefault}{\mddefault}{\updefault}$3$}}}}
\put(2715,1767){\makebox(0,0)[b]{\smash{{\SetFigFont{6}{7.2}{\rmdefault}{\mddefault}{\updefault}$4$}}}}
\put(3165,1767){\makebox(0,0)[b]{\smash{{\SetFigFont{6}{7.2}{\rmdefault}{\mddefault}{\updefault}$5$}}}}
\put(3615,1767){\makebox(0,0)[b]{\smash{{\SetFigFont{6}{7.2}{\rmdefault}{\mddefault}{\updefault}$6$}}}}
\put(4065,1767){\makebox(0,0)[b]{\smash{{\SetFigFont{6}{7.2}{\rmdefault}{\mddefault}{\updefault}$7$}}}}
\put(4515,1767){\makebox(0,0)[b]{\smash{{\SetFigFont{6}{7.2}{\rmdefault}{\mddefault}{\updefault}$8$}}}}
\put(4965,1767){\makebox(0,0)[b]{\smash{{\SetFigFont{6}{7.2}{\rmdefault}{\mddefault}{\updefault}$\cdots$}}}}
\put(465,1767){\makebox(0,0)[b]{\smash{{\SetFigFont{6}{7.2}{\rmdefault}{\mddefault}{\updefault}$-1$}}}}
\put(240,2037){\makebox(0,0)[b]{\smash{{\SetFigFont{6}{7.2}{\rmdefault}{\mddefault}{\updefault}$0$}}}}
\put(240,2487){\makebox(0,0)[b]{\smash{{\SetFigFont{6}{7.2}{\rmdefault}{\mddefault}{\updefault}$1$}}}}
\put(240,2937){\makebox(0,0)[b]{\smash{{\SetFigFont{6}{7.2}{\rmdefault}{\mddefault}{\updefault}$2$}}}}
\put(240,3387){\makebox(0,0)[b]{\smash{{\SetFigFont{6}{7.2}{\rmdefault}{\mddefault}{\updefault}$3$}}}}
\end{picture}
}
	\end{center}
	This example also demonstrates that from the singular spectra $\sigma_{s}(A_j)$ (which are the closures of the
	above pictured support sets), one cannot draw any conclusions about the singular spectrum $\sigma_s(A)$ or the
	spectral multiplicity function $N_A$.
\end{example}

\noindent
Next, we investigate with $E_{s,s}$, i.e., the possible appearance of new singular spectrum. This is not so straightforward.
In order to make explicit computations, we consider the situation which resembles a single half-line operator.

\begin{remark}\label{K79}
	Let $q$ be real and
	locally integrable potential defined on $(0,\infty)$, and assume that $0$ is a regular endpoint and that Weyl's limit point
	case prevails at the endpoint $\infty$. Let $n\geq 2$, and extend the potential $q$ to the star-graph with $n$ edges by symmetry
	(i.e., consider the same potential on all edges). Matching the notation of Theorem~\ref{K38}, we thus have
	\[
		\Gamma_l:=\Gamma_{(0)},\quad l=1,\cdots,n
		\,,
	\]
	where $\Gamma_{(0)}$ is defined for $q$ as in Remark~\ref{K33}. Then the Weyl functions $m_l$ are all equal, namely equal to the classical Weyl function $m_{(0)}$ constructed from the potential $q$. Thus, $m=n\cdot m_{(0)}$.

	From Proposition~\ref{K37} we see that the $n\!\times\!n$-matrix valued Weyl function $M$ corresponding to the boundary relation $\Gamma$ constructed by pasting $\Gamma_1,\cdots,\Gamma_n$ with standard interface conditions is given as
	\[
		M=\frac1{n}\!
		\begin{pmatrix}
			(n-1)m_{(0)} & -m_{(0)} & \cdots & -m_{(0)} & -1\\
			-m_{(0)} & (n-1)m_{(0)} & \cdots & -m_{(0)} & -1\\
			\vdots & \vdots & \ddots & \vdots & \vdots\\
			-m_{(0)} & -m_{(0)} & \cdots & (n-1)m_{(0)} & -1\\
			-1 & -1 & \cdots & -1 & -\frac 1{m_{(0)}}
		\end{pmatrix}
		\ .
	\]
	In particular,
	\[
		\tr M=\frac 1n\Big[(n-1)^2m_{(0)}-\frac 1{m_{(0)}}\Big]
		\,.
	\]
	Denote by $\mu_0$ and $\sigma_0$ the measures in the integral representations of $m_{(0)}$ and $-\frac 1{m_{(0)}}$,
	respectively. Then, for the trace measure $\rho$ constructed from $M$ we have
	\[
		\rho\sim\mu_0+\sigma_0
		\,.
	\]
	Using the notation of Theorem~\ref{K38}, we thus have
	\[
		\mu=n\cdot\mu_{(0)},\qquad r(x)=n\text{ for $\mu$-a.a.\ points }x\in\bb R
		\,,
	\]
	and hence (denoting by $\mu_{(0),s}$ and $\sigma_{0,s}$ the singular parts of $\mu_{(0)}$ and $\sigma_0$
	with respect to the Lebesgue measure)
	\begin{enumerate}[{\rm(I)}]
		\item $E_{s,ac}\sim\mu_{(0),s}$,
		\item $N_A(x)=n-1$ for $E_{s,ac}$-a.a.\ points $x\in\bb R$,
	\end{enumerate}
	\begin{equation}\label{K72}
		E_{s,s}\sim\sigma_{0,s}
		\,.
	\end{equation}
\end{remark}

\noindent
Using appropriately chosen potentials $q$, we can now provide examples that new singular spectrum (embedded or not) does appear,
or that no new singular spectrum appears.

\begin{example}[Appearance of new spectrum, partially embedded]\label{K70}
	In \cite[Theorem~3.5]{Remling-1999} a class of potentials is given, such that for every boundary condition the corresponding
	Schr\"odinger operator $A_\alpha$ satisfies ($F$ denotes a certain Smith-Volterra-Cantor-type set with positive Lebesgue measure)
	\[
		\sigma_{sc}(A_\alpha)=[0,\infty),\quad \sigma_{ac}(A_\alpha)=F^2,\quad \sigma_p(A_\alpha)\cap(0,\infty)=\emptyset
		\,.
	\]
	Using measure theoretic terms, we may thus say that there exist minimal supports of the corresponding singular continuous parts
	$\mu_{(\alpha),sc}$ of the spectral measures $\mu_{(\alpha)}$ which are Lebesgue-zero sets, mutually disjoint, and dense in $[0,\infty)$.
	
	For such potentials we see from \eqref{K72} that $E_{s,s}\neq 0$, i.e., new singular spectrum appears (precisely on a minimal support
	of $\mu_{(\frac\pi2),s}$). The part of this new spectrum located on the positive half-line is singular continuous. Moreover,
	since $\emptyset\neq\supp E_{ac}=F^2\subseteq[0,\infty)$, some part of it is embedded into the absolutely continuous spectrum.
	The spectrum originating from overlaps (which happens precisely on a minimal support of $\mu_{(0)}$) shares these properties.
\end{example}

\begin{example}[Appearance of new spectrum, not embedded]\label{K80}
	Consider the potential
	\[
		q(x):=
		\begin{cases}
			k &\hspace*{-3mm},\quad \big|x-\exp(2k^{\frac 32})\big|<\frac 12,\\
			0 &\hspace*{-3mm},\quad \text{otherwise}.\\
		\end{cases}
	\]
	This potential was studied in \cite{Simon-Stolz-1996}, and it turned out that for every boundary condition the corresponding
	selfadjoint operator $A_\alpha$ satisfies
	\[
		\sigma_{ac}(A_\alpha)=\emptyset,\quad \sigma_{sc}(A_\alpha)=[0,\infty)
		\,.
	\]
	Moreover, depending on the boundary condition either $\sigma_p(A_\alpha)$ is empty or consists of one negative eigenvalue. Expressed in measure theoretic terms, this means that there exist minimal supports of the corresponding spectral measures
	$\mu_{(\alpha)}$ which are all Lebesgue-zero sets, are mutually disjoint, whose intersection with $[0,\infty)$ is dense in $[0,\infty)$,
	which contain at most one point on the negative half-line, and that $\mu_{(\alpha)}$ has no point masses in $[0,\infty)$.

	For this potential we see from \eqref{K72} that $E_{s,s}\neq 0$, i.e., new singular spectrum appears. This new part of the
	spectrum is singular continuous and embedded into the spectrum originating from overlaps (the closures of minimal supports of
	$E_{s,ac}$ and $E_{s,s}$ are both equal to $[0,\infty)$).
\end{example}

\begin{example}[Non-appearance of new spectrum]\label{K81}
	We follow the idea given in \cite{Donoghue-1965} to construct examples, and use the type of measures discussed in
	\cite[Example~1]{Donoghue-1965} and the Gelfand-Levitan theorem. However, we need to refer to the version of the Gelfand-Levitan theorem
	which applies to the Dirichlet boundary condition, cf.\ \cite[\S2.9]{Levitan-1987}.

	Let us first recall the argument made in \cite[Example~1]{Donoghue-1965}\footnote{Again, we do not aim for maximal generality.}.
	Let $f$ be a continuous, bounded, and positive function on $\bb R$, let $\lambda_0$ be a positive measure with
	$\int_{\bb R}\frac{d\lambda_0(t)}{1+t^2}<\infty$, and consider the measure
	\[
		\nu:=\lambda_0+f\cdot\lambda
		\,,
	\]
	where $\lambda$ is the Lebesgue measure. Denote by $m$ the corresponding Herglotz-function
	\[
		m(z):=\int_{\bb R}\Big(\frac 1{t-z}-\frac t{1+t^2}\Big)\,d\nu(t),\quad z\in\bb C\setminus\bb R
		\,.
	\]
	Then, since $\nu\geq f\cdot\lambda$, we have
	\[
		\liminf_{\varepsilon\downarrow 0}\Im m(x+i\varepsilon)\geq\pi f(x),\quad x\in\bb R
		\,.
	\]
	In particular, the function $m(z)$ never approaches a real boundary value when $z$ tends to a real point (along a perpendicular ray).

	Consider the Herglotz functions
	\[
		m_\tau(z):=\frac{\tau m(z)-1}{m(z)+\tau},\quad \tau\in\bb R
		\,,
	\]
	and let $\nu_\tau$ be the measure in the integral representation of $m_\tau$.
	Then, by Aronszajn-Donoghue (cf.\ \cite[Theorem~3.2,(3.17)]{Gesztesy-Tsekanovskii-2000}), the measures $\nu_\tau$
	are all absolutely continuous with respect to the Lebesgue measure.

	Next, we make an appropriate choice of $f$ and $\lambda_0$, so to allow an application of the Gelfand-Levitan theorem.
	Set, for example,
	\[
		f(x):=\frac 2{3\pi}\cdot
		\begin{cases}
			e^{x-1} &\hspace*{-3mm},\quad x<1,\\
			x^{\frac 32} &\hspace*{-3mm},\quad x\geq 1,\\
		\end{cases}
	\]
	and let $\lambda_0$ be compactly supported. Then the hypotheses of the Gelfand-Levitan theorem are obviously fulfilled, and
	we obtain a potential $q$ on the half-line, such that the measure $\nu$ is the measure in the integral representation
	of the Weyl function constructed from the potential $q$ with Dirichlet boundary conditions.

	Symmetrically extending the potential $q$ constructed in the above paragraph, yields examples with
	\begin{enumerate}[$(i)$]
		\item $E_s=0$ (choose $\lambda_0=0$),
		\item $E_{s,ac}\neq 0$ but $E_{s,s}=0$ (choose $\lambda_0$ to be singular).
	\end{enumerate}
\end{example}

\appendixsection{Other boundary/interface conditions}

In the first statement of this section we show how Theorem~\ref{K38} can be used to deduce the classical result of Aronszajn-Donoghue that
the singular parts of the spectral measures corresponding to different boundary conditions in a half-line problem are mutually
singular. This approach is of course more complicated than the original one, and hence should not be viewed as a ``new proof of an old result''.
The reasons why we still find it worth to be elaborated are: (1) reobtaining previously known results gives a hint that the new result is not
unnecessarily weak, and (2) it demonstrates the usage of boundary relations with a nontrivial multivalued part.

\begin{corollary}[Aronszajn-Donoghue]\label{K71}
 	Let $q$ be real and locally integrable potential defined on $(0,\infty)$, assume that $0$ is a regular endpoint and that Weyl's limit point case prevails at the endpoint $\infty$. Let $\alpha_1,\alpha_2\in[0,\pi)$, $\alpha_1\neq\alpha_2$, be given,
	and let $A_{\alpha_j}$ denote the selfadjoint operators given by the corresponding boundary condition. Then the singular parts of the corresponding spectral measures are mutually singular.
\end{corollary}
\begin{proof}
	Let $T_{max}$ be the maximal operator associated with differential expression $-\frac{d^2}{dx^2}+q$, and
	let $\Gamma_{\alpha_j}$ be the boundary relations constructed in Remark~\ref{K33}. Moreover,
	denote by $\mu_{(\alpha_j)}$, $j=1,2$, the measure in the integral representation of the Weyl function of $\Gamma_{\alpha_j}$.

	Set $\beta:=\alpha_1-\alpha_2$, and denote by $\mr\Gamma$ the boundary relation
	\begin{equation}\label{K76}
		\mr\Gamma:=\big\{\big((0;0);(-w\sin\beta;w\cos\beta)\big):\, w\in\bb C\big\}\subseteq
		\{0\}^2\times\bb C^2
		\,.
	\end{equation}
	Then the pasting $\Gamma$ of $\mr\Gamma$ and $\Gamma_{(\alpha_1)}$ with standard interface conditions is given as
	\begin{multline*}
		\Gamma=\bigg\{\bigg(\Big(\binom 0{u};\binom 0{T_{max}u}\Big);
		\\[2mm]
		 \Big(\binom{w\sin\beta\!+\![u(0)\cos\alpha_1\!+\!u'(0)\sin\alpha_1]}{w\cos\beta\!+\![-u(0)\sin\alpha_1\!+\!u'(0)\cos\alpha_1]};
		 \binom{-w\cos\beta}{\!-\![u(0)\cos\alpha_1\!+\!u'(0)\sin\alpha_1]}\Big)\bigg):
		\\[2mm]
		w\in\bb C,\ u\in\dom T_{max}\bigg\}\subseteq\big(\{0\}\!\times\! L^2(0,\infty)\big)^2\times\bb C^2
		\,.
	\end{multline*}
	A short computation shows that (we identify $\{0\}\times L^2(0,\infty)$ with $L^2(0,\infty)$) $ker\big[\pi_1\circ\Gamma\big]=A_{\alpha_2}$.
	The Weyl function $\mr m$ of $\mr\Gamma$ is equal to the real constant $-\cot\beta$; note here that $\beta\in(-\pi,\pi)\setminus\{0\}$.
	The measure $\mr\mu$ in its integral representation is thus equal to $0$. Using the notation of Theorem~\ref{K38}, we have
	\[
		\mu=\mu_{(\alpha_1)},\qquad r(x)=1\text{ for $\mu$-a.a.\ points }x\in\bb R
		\,,
	\]
	and hence $E_{s,ac}=0$. However, $E\sim\mu_{(\alpha_2)}$, and we see that $\mu_{(\alpha_2),s}=\mu_{(\alpha_2),s,s}\perp\mu_{(\alpha_1)}$. Therefore $\mu_{(\alpha_1),s}\perp\mu_{(\alpha_2),s}$.
\end{proof}

\noindent
Finally, we provide some knowledge on other interface conditions than the standard ones.
In the context of this example, it is however important to add two remarks:
\\[1mm]
(1) The formula for spectral multiplicity will not be given in terms of the spectral measures of the non-interacting operators,
but in terms of the corresponding Weyl functions. Hence, the below result cannot be viewed as a strict analogue of Theorem~\ref{K38} for
other interface conditions.
\\[1mm]
(2) The applied method provides knowledge only about a particular (small) class of interface conditions. It does not lead to a treatment
of arbitrary interface conditions on a star-graph, and even less to a formula for multiplicity on graphs with a more complicated geometry.

\begin{proposition}\label{K73}
	Let $n\geq 2$ and real valued locally integrable potentials $q_l$, $l=1,\dots,n$, on the half-line be given, such that $q_l$ is regular at $0$ and in Weyl's limit point case at $\infty$. Let $a_1,\dots,a_n\in[0,2\pi)$ and $b\in(0,\pi)$ be given,	and consider the selfadjoint matrix Schr\"odinger operator $A$ defined on $L^2(0,\infty)^n$ by the differential expression	 $-\frac{d^2}{dx^2}+V$ with the diagonal matrix potential
	\[
		V(x):=
		\begin{pmatrix}
			q_1(x)\!\!\! & & \\[-2mm]
			& \!\!\ddots\!\! & \\[-2mm]
			& & \!\!q_n(x)
		\end{pmatrix}
	\]
	and the interface conditions
	\begin{equation}\label{K75}
	\begin{gathered}
		u_1(0)\cos a_1+u_1'(0)\sin a_1=\dots=u_n(0)\cos a_n+u_n'(0)\sin a_n\,,
		\\
		\sum_{l=1}^n\big[u_l(0)\cos(a_l-b)+u_l'(0)\sin(a_l-b)\big]=0
		\,.
	\end{gathered}
	\end{equation}
	Denote the Titchmarsh-Weyl coefficient constructed from the potential $q_l$ (with Dirichlet boundary conditions) as $m_l$, and set
	\[
		S_l:=
		\begin{cases}
			\big\{x\in\bb R:\,\lim_{\varepsilon\downarrow 0}m_l(x+i\varepsilon)=-\cot a_l\big\} &\hspace*{-3mm},\quad
				a_l\not\in\{0,\pi\},\\[2mm]
			\big\{x\in\bb R:\,\lim_{\varepsilon\downarrow 0}\Im m_l(x+i\varepsilon)=\infty\big\} &\hspace*{-3mm},\quad
				a_l\in\{0,\pi\}.\\
		\end{cases}
	\]
	Then
	\[
		N_A(x)=
		\begin{cases}
			\#\big\{l\in\{1,\dots,n\}:\,x\in S_l\big\}-1 &\hspace*{-3mm},\quad E_{s,ac}\text{-a.e.},\\
			1 &\hspace*{-3mm},\quad E_{s,s}\text{-a.e.}\\
		\end{cases}
	\]
\end{proposition}
\begin{proof}
	First of all, let us explicitly state how the operator $A$ acts:
	\[
		A\begin{pmatrix}
			u_1 \\ \vdots \\ u_n
		\end{pmatrix}
		:=
		-\frac{d^2}{dx^2}
		\begin{pmatrix}
			u_1 \\ \vdots \\ u_n
		\end{pmatrix}
		+
		\begin{pmatrix}
			q_1u_1 \\ \vdots \\ q_nu_n
		\end{pmatrix}
		\,,
	\]
	on the domain
	\begin{align*}
		\dom A := \bigg\{&(u_1, \dots,u_n)\in\prod_{l=1}^n L_2(0,\infty):
			\\
		& u_l,u_l'\text{ are absolutely continuous}, -u_l''+q_lu_l\in L_2(0,\infty),
			\\
		& u_1,\dots,u_n\text{ satisfy the interface conditions \eqref{K75}}\bigg\}
			\ .
	\end{align*}
	For $l=1,\dots,n$ let $\Gamma_l$ be the boundary relation which is defined from the potential $q_l$ as $\Gamma_{(a_l)}$, cf.\ Remark~\ref{K33}. Moreover, let $\mr\Gamma$ be the boundary relation defined in \eqref{K76}. We consider the pasting of $\Gamma_1,\dots,\Gamma_n,\mr\Gamma$ with standard interface conditions. Then an element
	$(u_1,\dots,u_n)$ belongs to the domain of the operator\footnote{Similar to the previous example: we identify $L^2(0,\infty)^n\!\times\!\{0\}$ with $L^2(0,\infty)^n$.} $\ker[\pi_1\circ\Gamma]$ if and only if
	\[
		\exists\ w\in\bb C:\quad
		\left\{
		\begin{array}{ll}
			u_l(0)\cos a_l+u_l'(0)\sin a_l=-w\sin b,\quad l=1,\dots,n\,,
			\\[2mm]
			\sum\limits_{l=1}^n\big[-u_l(0)\sin a_l+u_l(0)\cos a_l\big]=-w\cos b\,.
		\end{array}
		\right.
	\]
	Eliminating $w$ from these equations, yields the assertion.
\end{proof}

\begin{remark}\label{K82}
	Some observations are in order:
	\begin{enumerate}[$(i)$]
		\item It is interesting to notice that the support set and the multiplicity function corresponding to $E_{s,ac}$
			does not depend on the choice of the parameter $b$.
		\item The fact that the value ``$b=0$'' is excluded is natural. For this value the conditions
			\eqref{K75} reduce to
			\[
				u_1(0)\cos a_1+u_1'(0)\sin a_1=\dots=u_n(0)\cos a_n+u_n'(0)\sin a_n=0
				\,,
			\]
			i.e., the operator $A$ is equal to the direct sum of the non-interacting operators $A_l$ defined by the potentials
			$q_l$ using the boundary condition $u_l(0)\cos a_l+u_l'(0)\sin a_l=0$.
		\item The case that ``$b=\frac\pi 2$'' could be treated somewhat simpler. For this value the operator $A$ coincides with
			$\ker[\pi_1\circ\Gamma]$ where $\Gamma$ is the pasting of $\Gamma_1,\dots,\Gamma_n$ with standard interface conditions.
	\end{enumerate}
\end{remark}



\subsection*{Acknowledgements}
We thank Anton Baranov for pointing our attention to the work of Alexei Poltoratski, and Sergey Naboko for valuable comments.

The first author was supported by the Chebyshev Laboratory (Department of Mathematics and
Mechanics, Saint-Petersburg State University) under the grant 11.G34.31.0026 of
the Government of the Russian Federation, by grants RFBR-09-01-00515-a and
11-01-90402-Ukr\_f\_a and by the Erasmus Mundus Action 2 Programme of the
European Union.

{\footnotesize
\begin{flushleft}
	S.\,Simonov\\
	Chebyshev Laboratory, Department of Mathematics and Mechanics\\
	Saint-Petersburg State University\\
	14th Line, 29b\\
	199178 Saint-Petersburg\\
	RUSSIA\\
	email: sergey.a.simonov@gmail.com\\[5mm]
\end{flushleft}
}
{\footnotesize
\begin{flushleft}
	H.\,Woracek\\
	Institut f\"ur Analysis und Scientific Computing\\
	Technische Universit\"at Wien\\
	Wiedner Hauptstra{\ss}e.\ 8--10/101\\
	1040 Wien\\
	AUSTRIA\\
	email: harald.woracek@tuwien.ac.at\\[5mm]
\end{flushleft}
}

\end{document}